\documentclass[12pt]{amsart}
\usepackage{amssymb,amsmath,amsfonts,amscd}
\usepackage{times,graphicx,epsfig}
\usepackage{soul}

\theoremstyle{plain}
\newtheorem{theorem}{Theorem}[section]
\newtheorem{stheorem}{Theorem}
\newtheorem{proposition}[theorem]{Proposition}
\newtheorem{definition}[theorem]{Definition}

\newtheorem{remark}[theorem]{Remark}
\newtheorem{lemma}[theorem]{Lemma}

\def\eps{\epsilon}

\usepackage{pdfsync}
\usepackage{flexisym}

\newcommand{\E}{\mathbb E}

\newcommand{\rn}[1]{{\mathbb R}^{#1}}
\newcommand{\R}{\mathbb R}

\newcommand{\G}{\mathbb G}

\newcommand{\Z}{\mathbb Z}
\newcommand{\N}{\mathbb N}

\newcommand{\supp}{\mathrm{supp}\;}

\newcommand{\cov}[1]{{\bigwedge\nolimits^{#1}{\mfrak h}}}

\newcommand{\covH}[1]{{\bigwedge\nolimits^{#1}{\mfrak h}}}

\newcommand{\he}[1]{{\mathbb H}^{#1}}

\newcommand{\scal}[2]{\langle {#1} , {#2}\rangle}

\newcommand{\Scal}[2]{\langle {#1} \vert {#2}\rangle}
\newcommand{\scalp}[3]{\langle {#1} , {#2}\rangle_{#3}}

\newcommand{\res}{\mathop{\hbox{\vrule height 7pt width .5pt depth 0pt
\vrule height .5pt width 6pt depth 0pt}}\nolimits}

\newcommand{\ccheck}{{\vphantom i}^{\mathrm v}\!\,}
\newcommand{\mc}{\mathcal }

\newcommand{\mfrak}{\mathfrak}

\newcommand{\WO}[3]{\mathop{W}\limits^\circ{}\!^{{#1},{#2}}
(#3)}

\newcommand{\BL}[2]{BL^{{#1},{#2}}
(\he n)}

\newcommand{\BLh}[3]{BL^{{#1},{#2}}
(\he n, E_0^{#3})}

\usepackage[usenames]{color}

\begin{document}

%\today

\bigskip

\title[Continuous primitives for higher degree differential forms]{Continuous primitives for higher degree differential forms in Euclidean spaces, Heisenberg groups and applications}

\author[Annalisa Baldi, Bruno Franchi, Pierre Pansu]{
Annalisa Baldi\\
Bruno Franchi\\ Pierre Pansu
}

\begin{abstract}
It is shown that higher degree exact differential forms on compact Riemannian $n$-manifolds possess continuous primitives whose uniform norm is controlled by their $L^n$ norm. A contact sub-Riemannian analogue is proven, with differential forms replaced with Rumin differential forms. 

\bigskip

Il est d\'emontr\'e que les formes diff\'erentielles exactes de degr\'e sup\'erieur sur des vari\'et\'es riemanniennes compactes de dimension $n$ poss\`edent des primitives continues dont la norme uniforme est contr\^ol\'ee par leur norme $L^n$. Un analogue dans des vari\'et\'es sous-riemanniennes de contact est prouv\'e, les formes diff\'erentielles \'etant remplac\'ees par des formes diff\'erentielles de Rumin.

\end{abstract}
 
\keywords{Heisenberg groups, differential forms, Poincar\'e inequalities, Beppo Levi spaces}

\subjclass{58A10,  35R03, 26D15,  
46E35}

\maketitle

\tableofcontents

\section{Introduction}

\subsection{The problem}

Poincar\'e's inequality in the flat $n$-torus $\R^n/\Z^n$ states that a function $u$ whose gradient belongs to $L^p$, $p<n$, is itself in $L^q$, up to an additive constant $c_u$, provided that $\frac{1}{p}-\frac{1}{q}=\frac{1}{n}$. Moreover
$$
\|u-c_u\|_q \le C_{p,q}\|\nabla u\|_p.
$$

Such an estimate fails to hold when $p=n$ and $q=\infty$ (see e.g.  \cite{wheeden}, p.484 and  \cite{trudinger}). However, when one passes to higher degree differential forms, the limiting case holds. This discovery, at least for top degree differential forms, is due to J. Bourgain and H. Brezis, \cite{BB2003}. The statement takes the following form: \emph{let $\omega$ be an exact $n$-form on the $n$-torus, which belongs to $L^n$, then there exists a bounded differential $(n-1)$-form $\phi$ on the torus such that $d\phi=\omega$ and}
$$
\|\phi\|_\infty \le C \|\omega\|_n.
$$

Furthermore, Bourgain and Brezis show that the primitive can be taken to be \emph{continuous}, with a similar estimate. 

The global version on $\R^n$ itself is even slightly stronger: the primitive can be taken to be continuous and to \emph{tend to zero at infinity}, with a similar estimate, as shown by L.~Moonens \& T.~Picon \cite{Moonens-Picon}, and T.~De Pauw \& M.~Torres \cite{DePaw-Torres}.

In this paper, we shall extend Moonens and Picon's result to differential forms of all degrees $>1$ on $\R^n$, and prove a version for the Heisenberg group $\he{n}$, with Rumin's complex $d_c$ replacing de Rham's differential $d$. A special instance, in general
Carnot groups for top degree Rumin forms, is considered by A.~Baldi \& F.~Montefalcone \cite{BM2}.

In the next statement, $\mc C_0$ denotes spaces of continuous differential forms that tend to zero at infinity.

\begin{stheorem}\label{main}
\begin{itemize}
The following global Poincar\'e inequalities hold.
\item[i)] (Euclidean spaces:) if $\,2\le  h \le n$, then a $d$-exact $h$-form $\omega\in L^n(\R^n, \bigwedge^h)$
admits a primitive $\phi\in \mc C_0(\R^n,\bigwedge^{h-1})$ such that
$$
\|\phi\|_{\mc C_0(\R^n,\bigwedge^{h-1})} \le C \|\omega\|_{L^n(\R^n, \bigwedge^h)};
$$

\item[ii)] (Heisenberg groups:) if $\,2\le  h \le { 2n+1}$, $h\neq n+1$, then a $d_c$-exact Rumin $h$-form $\omega\in L^{2n+2}(\he n, E_0^h)$
admits a primitive $\phi\in \mc C_0(\he n, E_0^{h-1})$ such that
$$
\|\phi\|_{\mc C_0(\he n, E_0^{h-1})} \le C \|\omega\|_{L^{2n+2}(\he n, E_0^h)};
$$
\item[iii)] (Heisenberg groups:) a $d_c$-exact Rumin $(n+1)$-form $\omega\in$ \break $L^{n+1}(\he n, E_0^{n+1})$ admits a primitive $\phi\in \mc C_0(\he n, E_0^{n})$ such that
$$
\|\phi\|_{\mc C_0(\he n, E_0^{n})} \le C \|\omega\|_{L^{n+1}(\he n, E_0^{n+1})}.
$$
\end{itemize}

\end{stheorem}

Note the game with critical exponents: they are determined by the degree of volume growth, which is equal to $n$ for Euclidean space $\R^n$, but $Q=2n+2$ for Heisenberg group. Moreover, since Rumin's differential $d_c$ from $n$-forms to $(n+1)$-forms has order $2$, the critical exponent turns out to be $Q/2 =n+1$ instead of $Q$ in this special degree.

\medskip

We shall also prove a local version, from which Poincar\'e inequalities can be deduced on arbitrary compact Riemannian or compact contact subRiemannian manifolds. Detailed statements are given in Section \ref{Poincare}. In particular, we recover Bourgain-Brezis' original result on the torus. 

\medskip

\textbf{Question}. Can one improve Theorem \ref{main} by establishing a uniform modulus of continuity for primitives?

\subsection{The method}

Let us focus on Heisenberg groups (the Euclidean case follows the same lines, with less technicalities). A Poincar\'e inequality providing a bounded measurable primitive can be found in \cite{BFP5}. However one cannot take the results there as a black-box. 

Following Bourgain-Brezis, the proof of Theorem \ref{main} goes by duality. The general principle, that a closed unbounded operator with dense domain between Banach spaces has a closed range if and only if its adjoint does, is applied to Rumin's differential $d_c:\mc C_0 \to L^{p}$, $p=Q$ or $Q/2$. The adjoint operator maps $L^{p'}$ to the space $\mc M$ of \emph{Rumin currents} of finite mass. The $L^1$-Poincar\'e inequalities proven in \cite{BFP3} provide $L^{p'}$ primitives $\phi$ to $L^1$-Rumin forms $\omega$, $d_c\phi=\omega$, and 
$$
\|\phi\|_{p'}\le C\|\omega\|_1.
$$
By approximation of Rumin currents of finite mass with currents defined by $L^1$-Rumin forms, such an estimate can be upgraded to provide $L^{p'}$ primitives $\phi$ to $L^1$-Rumin currents $T$ of the form $T=\partial_c S$, with
$$
\|\phi\|_{p'}\le C\|T\|_{\mc M}.
$$
This is the required closed range estimate for the adjoint operator of $d_c$.

The procedure needs be modified in degree $n+1$. Indeed, there is a technical point when identifying the domain of the adjoint: when handling the second order operator $d_c$, cut-off arguments require extra control on first derivatives. Hence the space $L^{p'}$ is replaced with the Beppo Levi space (see e.g. \cite{DL}) $\BL{1}{Q'}$ of Rumin forms whose coefficients have their first derivatives in $L^{Q'}$. One therefore needs an $L^1$ Poincar\'e inequality of the form
$$
\|\phi\|_{\BL{1}{Q'}}\le C\|\omega\|_1.
$$
Unfortunately, one cannot use \cite{BFP3} as a black box. Instead, one follows the strategy of \cite{BFP3}, based on the Gagliardo-Nirenberg inequalities of \cite{BFP1}.

The local version builds upon Theorem \ref{main} and the local $L^\infty-L^p$-Poincar\'e inequality of \cite{BFP5}. 

\subsection{Organization of the paper}

As already stressed above, all proofs are given for $\he n$ since the case of $\mathbb R^n$ just requires
less technicalities. After fixing notations in Section \ref{HeisenbergGroups}, a vademecum on Rumin's complex is provided in Section \ref{rumin}, focussing on Leibniz' formula and commutation properties of Rumin's Laplacian. Detailed statements of Theorem \ref{main}, its local version and its avatar for compact manifolds are given in section \ref{Poincare}. The duality principle is formulated in Section \ref{BB duality}. The spaces of currents to which this principle is applied are defined in Section \ref{currents and measures}, whose main result is a Poincar\'e inequality: solving $d_c$ in $\mc C_0$ for a current datum. The proof of Theorem \ref{main} appears in Section \ref{proof global}, except for degree $(n+1)$-forms. In this special degree, a new Gagliardo-Nirenberg inequality is proven as Theorem \ref{bfp1}; a Poincar\'e inequality, first for $L^1$ forms and then for currents of finite mass follows in Section \ref{caso 2}. A density result needed to identify the domain of an adjoint operator is proven in Section \ref{more}. Section \ref{e0n} contains the proof of Theorem \ref{main} in the middle degree $n+1$. Section \ref{local} contains the proof of the local version of Theorem \ref{main}, and Section \ref{compact} draws consequences for bouded geometry manifolds. 

\section{Heisenberg groups}
\label{HeisenbergGroups}

\subsection{Definitions and preliminary results}\label{preliminary}
We denote by $\he n$  the $n$-dimensional Heisenberg
group, identified with $\rn {2n+1}$ through exponential
coordinates. A point $p\in \he n$ is denoted by
$p=(x,y,t)$, with both $x,y\in\rn{n}$
and $t\in\R$.
   If $p$ and
$p'\in \he n$,   the group operation is defined by
\begin{equation*}
p\cdot p'=(x+x', y+y', t+t' + \frac12 \sum_{j=1}^n(x_j y_{j}'- y_{j} x_{j}')).
\end{equation*}
The unit element of $\he n$ is the origin, that will be denoted by $e$.
For
any $q\in\he n$, the {\it (left) translation} $\tau_q:\he n\to\he n$ is defined
as $$ p\mapsto\tau_q p:=q\cdot p. $$
The Lebesgue measure in $\mathbb R^{2n+1}$ 
is a Haar measure in $\he n$ (i.e., a bi-invariant measure on the group). It is denoted by $\mc L^{2n+1}$, and when we need to stress the integration variable $p$, will be denoted also by $dp$.
%in $\mathbb R^{2n+1}$
%is a bi-invariant measure on the group

For a general review on Heisenberg groups and their properties, we
refer to \cite{Stein}, \cite{GromovCC}, \cite{BLU}, and to \cite{VarSalCou}.
We limit ourselves to fix some notations, following \cite{FSSC_advances}.

The Heisenberg group $\he n$ can be endowed with the homogeneous
norm (Cygan-Kor\'anyi norm)
\begin{equation}\label{gauge}
\varrho (p)=\big(|p'|^4+ 16\, p_{2n+1}^2\big)^{1/4},
\end{equation}
and we define the gauge distance (a true distance, see
 \cite{Stein}, p.\,638), that is left invariant i.e. $d(\tau_q p,\tau_q p')=d(p,p' )$ for all $p,p'\in\he n$)
as
\begin{equation}\label{def_distance}
d(p,q):=\varrho ({p^{-1}\cdot q}).
\end{equation}
Finally, the balls for the metric $d$ are le so-called Kor\'anyi balls
\begin{equation}\label{koranyi}
B(p,r):=\{q \in  \he n; \; d(p,q)< r\}.
\end{equation}

Notice that Kor\'anyi balls are smooth convex sets.

A straightforward computation shows that there exists $c_0>1$ such that
\begin{equation}\label{c0}
c_0^{-2} |p| \le \rho(p) \le |p|^{1/2},
\end{equation}
provided $p$ is close to $e$.
In particular, for $r>0$ small, if we denote by $B_{\mathrm{Euc}}(e,r)$ the Euclidean ball centred at $e$ of radius $r$,
\begin{equation}\label{balls inclusion}
B_{\mathrm{Euc}}(e,r^2) \subset B(e,r) \subset B_{\mathrm{Euc}}(e, c_0^2 r).
\end{equation}

    We denote by  $\mfrak h$
 the Lie algebra of the left
invariant vector fields of $\he n$. The standard basis of $\mfrak
h$ is given, for $i=1,\dots,n$,  by
\begin{equation*}
X_i := \partial_{x_i}-\frac12 y_i \partial_{t},\quad Y_i :=
\partial_{y_i}+\frac12 x_i \partial_{t},\quad T :=
\partial_{t}.
\end{equation*}
The only non-trivial commutation  relations are $
[X_{i},Y_{i}] = T $, for $i=1,\dots,n.$ 
The {\it horizontal subspace}  $\mfrak h_1$ is the subspace of
$\mfrak h$ spanned by $X_1,\dots,X_n$ and $Y_1,\dots,Y_n$:
${ \mfrak h_1:=\mathrm{span}\,\left\{X_1,\dots,X_n,Y_1,\dots,Y_n\right\}\,.}$

\noindent Coherently, from now on, we refer to $X_1,\dots,X_n,Y_1,\dots,Y_n$
(identified with first order differential operators) as
the {\it horizontal derivatives}. Denoting  by $\mfrak h_2$ the linear span of $T$, the $2$-step
stratification of $\mfrak h$ is expressed by
\begin{equation*}
\mfrak h=\mfrak h_1\oplus \mfrak h_2.
\end{equation*}

\bigskip

The stratification of the Lie algebra $\mfrak h$ induces a family of non-isotropic dilations 
$\delta_\lambda: \he n\to\he n$, $\lambda>0$ as follows: if
$p=(x,y,t)\in \he n$, then
\begin{equation}\label{dilations}
\delta_\lambda (x,y,t) = (\lambda x, \lambda y, \lambda^2 t).
\end{equation}
Notice that the gauge norm \eqref{gauge} is positively $\delta_\lambda$-homogenous,
so that the Lebesgue measure of the ball $B(x,r)$ is $r^{2n+2}$ up to a geometric constant
(the Lebesgue measure of $B(e,1)$).

The constant 
$$
Q:=2n+2,
$$
is said the {\sl homogeneous dimension}  of $\he n$
with respect to $\delta_\lambda$, $\lambda>0$,
It is well known that the topological dimension of $\he n$ is $2n+1$,
since as a smooth manifold it coincides with $\R^{2n+1}$, whereas
the Hausdorff dimension of $(\he n,d)$ is $Q$.

The vector space $ \mfrak h$  can be
endowed with an inner product, indicated by
$\scalp{\cdot}{\cdot}{} $,  making
    $X_1,\dots,X_n$,  $Y_1,\dots,Y_n$ and $ T$ orthonormal.
    
Throughout this paper, we write also
\begin{equation}\label{campi W}
W_i:=X_i, \quad W_{i+n}:= Y_i\quad { \mathrm{and} } \quad W_{2n+1}:= T, \quad \text
{for }i =1, \dots, n.
\end{equation}

As in \cite{folland_stein},
we also adopt the following multi-index notation for higher-order derivatives. If $I =
(i_1,\dots,i_{2n+1})$ is a multi--index, we set  
\begin{equation}\label{WI}
W^I=W_1^{i_1}\cdots
W_{2n}^{i_{2n}}\;T^{i_{2n+1}}.
\end{equation}
By the Poincar\'e--Birkhoff--Witt theorem, the differential operators $W^I$ form a basis for the algebra of left invariant
differential operators in $\he n$. 
Furthermore, we set 
$$
|I|:=i_1+\cdots +i_{2n}+i_{2n+1}$$
 the order of the differential operator
$W^I$, and   
$$d(I):=i_1+\cdots +i_{2n}+2i_{2n+1}$$
 its degree of homogeneity
with respect to group dilations.

Let  $U\subset \he n$ be an open set. We shall use the following classical notations:
$\mc E(U)$ is the space of all smooth function on $U$,
and $\mc D(U)$  is the space of all compactly supported smooth functions on $U$,
endowed with the standard topologies (see e.g. \cite{treves}).
The spaces $\mc E'(U)$ and $\mc D'(U)$ are their
dual spaces of distributions.
 We recall that 
$\mc E'(U)$ is the class of compactly supported distributions.

%
%{\color{green}For sake of completeness, we collect now a few remark on measure and distributions
%on $\he n$.
%
%
%\begin{proposition} We have:
%\begin{itemize} 
%\item[i)] The Lebesgue measure $\mc L^{2n+1}$  in $\mathbb R^{2n+1}$
%is a bi-invariant measure on the group;
%%\item[ii)] a function $f\in L^1_{\mathrm{loc}}(\he n)$ is identified with
%%the distribution $f\, \mc L^{2n+1}$;
%\item[ii)] let  $U\subset \he n$ be an open set. If $f\in \mc D'(U)$, then the distributional derivative  of $f$ along a left-invariant 
%vector field $W$ is defined by $\Scal{Wf}{\phi}= - \Scal{f}{W\phi}$ for any test function $\phi\in\mc D(U)$.
%
%\end{itemize}
%
%\end{proposition}
%
%If $\phi\in \mc E(\he n,E_0^h)$ and $k\in\mathbb N$, then we set
%$$
%|W^k\phi|:= \sum_{d(I)=k} |W^I\phi|.
%$$
%}

Let $1\le p \le\infty$
and $m\in\mathbb N$,  $W^{m,p}_{\mathrm{Euc}}(U)$
denotes the usual Sobolev space. 

We recall also the notion of 
(integer order) Folland-Stein Sobolev space (for a general presentation, see e.g. \cite{folland} and \cite{folland_stein}).

\begin{definition}\label{integer spaces} If $U\subset \he n$ is an open set, $1\le p \le\infty$
and $m\in\mathbb N$, then
the space $W^{m,p}(U)$
is the space of all $u\in L^p(U)$ such that, with the notation of \eqref{WI},
$$
W^Iu\in L^p(U)\quad\mbox{for all multi-indices $I$ with } d(I) \le m,
$$
endowed with the natural norm  
$$\|\,u\|_{W^{k,p}(U)}
:= \sum_{d(I) \le m} \|W^I u\|_{L^p(U)}.
$$

\end{definition}

{  Folland-Stein Sobolev spaces enjoy the following properties akin to those
of the usual Euclidean Sobolev spaces (see \cite{folland}, and, e.g. \cite{FSSC_houston}).}

\begin{theorem}\label{folland stein varia}If $U\subset \he n$  is an open set, $1\le p \le  \infty$, and $k\in\mathbb N$, then
\begin{itemize}
\item[i)] $ W^{k,p}(U)$ is a Banach space.
\end{itemize}
In addition, if $p<\infty$,
\begin{itemize}
\item[ii)] $ W^{k,p}(U)\cap C^\infty (U)$ is dense in $ W^{k,p}(U)$;
\item[iii)] if $U=\he n$, then $\mc D(\he n)$ is dense in $ W^{k,p}(U)$;
\item[iv)]  if $1<p<\infty$, then $W^{k,p}(U)$ is reflexive.
\item[v)]
$ W^{k,p}_{\mathrm{Euc, loc}}(U)\subset W^{k,p}(U)$, i.e. for any compact set
$K\subset U$ and for any $u\in W^{k,p}_{\mathrm{Euc, loc}}(U)$,
$$
\| u\|_{W^{k,p}(K)}\le C_K \| u\|_{ W^{k,p}_{\mathrm{Euc, loc}}(K)}.
$$
\item[vi)] $ W^{2k,p}(U)\subset W^{k,p}_{\mathrm{Euc, loc}}(U)$, 
 i.e. for any compact set
$K\subset U$ and for any $u\in W^{2k,p}(U)$,
$$
\| u\|_{W^{k,p}_{\mathrm{Euc}}(U)(K)}\le C_K \| u\|_{ W^{2k,p}(K)}.
$$

\end{itemize}

\end{theorem}

\begin{theorem}\label{donne}[see \cite{folland}, Theorem 5.15] If $p>Q$, then 
$$
W^{1,p}(\he n) \subset L^\infty (\he n)
$$
algebraically and topologically.
\end{theorem}

\begin{definition} If $U\subset \he n$ is open and if 
$\,1\le p<\infty$,
we denote  by $\WO{k}{p}{U}$
the completion of $\mc D(U)$ in $W^{k,p}(U)$.
\end{definition}
  \begin{remark}
If $U\subset \he n$ is bounded, by (iterated) Poincar\'e inequality (see e.g. \cite{jerison}), it follows that the norms
\begin{equation*}
\|u\|_{W^{k,p}(U)} \quad\mbox{and}\quad
\sum_{d(I)=k}\| W^I u\|_{ L^p(U)}
\end{equation*}
are equivalent on $\WO{k}{p}{U}$ when $1\le p<\infty$.
\end{remark}

\subsection{Convolution in Heisenberg groups}\label{sobolev kernels}

 If $f:\he n\to\mathbb R$, we set $\ccheck f(p) = f(p^{-1})$, and,
if $T\in \mc D'(\he n)$, then $\Scal{\ccheck T}{\phi} := \Scal{T}{\ccheck \phi}$
for all $\phi\in\mc D(\he n)$. Obviously, the map $T\to\ccheck T$ is
continuous from $\mc D'(\he n)$ to $\mc D'(\he n)$.

 Following e.g. \cite{folland_stein}, p. 15, we can define a group
convolution in $\he n$: if, for instance, $f\in\mc D(\he n)$ and
$g\in L^1_{\mathrm{loc}}(\he n)$, we set
\begin{equation}\label{group convolution}
f\ast g(p):=\int f(q)g(q^{-1}\cdot p)\,dq\quad\mbox{for $q\in \he n$}.
\end{equation}
We recall that, if, say, $g$ is a smooth function and $P$
is a left invariant differential operator, then
$$
P(f\ast g)= f\ast Pg.
$$
We also recall that the convolution is well defined
when $f,g\in\mc D'(\he n)$, provided at least one of them
has compact support.

In this case the following identities
hold:
\begin{equation}\label{convolutions var}
\Scal{f\ast g}{\Phi} = \Scal{g}{\ccheck f\ast\Phi}
\quad
\mbox{and}
\quad
\Scal{f\ast g}{\Phi} = \Scal{f}{\Phi\ast \ccheck g}
\end{equation}
 for any test function $\Phi$.
 
 \begin{definition}\label{type} A kernel of type $\mu$ is a 
homogeneous distribution of degree $\mu-Q$
(with respect to group dilations $\delta_r$),
that is smooth outside of the origin.

The convolution operator with a kernel of type $\mu$
is  called an operator of type $\mu$.
\end{definition}

\begin{proposition}\label{kernel}
Let $K\in\mc D'(\he n)$ be a kernel of type $\mu$.
\begin{itemize}
\item[i)] $\ccheck K$ is again a kernel of type $\mu$;
\item[ii)] $WK$ and $KW $ are associated with  kernels of type $\mu-1$ for
any horizontal derivative $W$;
\item[iii)]  If $\mu>0$, then $K\in L^1_{\mathrm{loc}}(\he n)$.
\end{itemize}
\end{proposition}

\begin{theorem} \label{hls folland} Suppose $0<\alpha<Q$, and let
$K$ be a kernel of type $\alpha$. Then
\begin{itemize}
\item[i)] if $1<p<Q/\alpha$, and $1/q:= 1/p-\alpha/Q$, then { 
there exists $C=C(p,\alpha)>0$ such that}
$$
\| u\ast K\|_{L^q(\he n)} \le C \| u\|_{L^p(\he n)}
$$
for all $u\in L^p(\he n)$.
\item[ii)] If $p\ge Q/\alpha$ and $B, B' \subset \he n$ are fixed balls with $B\subset B'$, then
for any $q\ge p$ {  there exists $C=C(B,B',p,q,\alpha)>0$}
$$
\| u\ast K\|_{L^q(B')} \le C \| u\|_{L^p({ B})}
$$
for all $u\in L^p(\he n)$ with $\supp u\subset B$.
\item[iii)] If $K$ is a kernel of type 0 and $1<p<\infty$, then
{  there exists $C=C(p)>0$ such that}
 $$
\| u\ast K\|_{L^p(\he n)} \le C \| u\|_{L^p(\he n)}.
$$
\end{itemize}

\end{theorem}

\begin{theorem}[\cite{CvS2009}, Theorem 1]\label{chanillo_van} Let $\Phi\in \mc D(\he n,\mathfrak h_1)$ be a smooth
compactly supported horizontal vector field. Suppose  $G\in L^1_{\mathrm{loc}}(\he n,\mathfrak h_1)$
is $\he{}$-di\-ver\-gen\-ce free, i.e. if 
$$
G= \sum_i G_iW_i, \quad\mbox{then} \quad \sum_iW_iG_i=0 \quad\mbox{in $\mc D'(\he n)$}.
$$
Then
$$
\big|\scal{G}{\Phi}_{L^2(\he{n},\mathfrak h_1)} \big| \le C\|G\|_{L^1(\he{n},\mathfrak h_1)} \|\nabla_{\he{}} \Phi\|_{L^Q(\he{n},\mathfrak h_1)}.
$$

\end{theorem}

 \section{Rumin's complex of differential forms}\label{rumin}

When dealing 
with differential forms in $\he n$, 
the de Rham complex lacks scale invariance under anisotropic dilations (see \eqref{dilations}). 
M. Rumin, in \cite{rumin_jdg} has defined a substitute of de Rham's complex for arbitrary contact manifolds, that recovers scale invariance under $\delta_\lambda$. 
 In the present section, we shall merely list a few properties of Rumin's complex that we need in this paper. We send a reader, interested to understand better Rumin's complex, to the Appendix of \cite{BFP5} for a quick review, or to \cite{rumin_jdg} and \cite{BFTT}, \cite{BFP2} for more details of the construction.

\medskip

The dual space of $\mfrak h$ is denoted by $\covH 1$.  The  basis of
$\covH 1$,  dual to  the basis $\{X_1,\dots , Y_n,T\}$,  is the family of
covectors $\{dx_1,\dots, dx_{n},dy_1,\dots, dy_n,\theta\}$ where 
$$ \theta
:= dt - \frac12 \sum_{j=1}^n (x_jdy_j-y_jdx_j)$$ is  the {\it contact
form} in $\he n$. 
We denote by $\scalp{\cdot}{\cdot}{} $ the
inner product in $\covH 1$  that makes $(dx_1,\dots, dy_{n},\theta  )$ 
an orthonormal basis  and by
$dV$ the associated volume form
$$
dV:= dx_1\wedge\dots \wedge dx_{n}\wedge dy_1\wedge\dots\wedge dy_n \wedge \theta.
$$

Throughout this paper, $\bigwedge^h\mathfrak h$ denotes the $h$-th exterior power of the Lie algebra $\mathfrak h$.  Keeping in mind that the Lie algebra $\mathfrak h$ can be identified with the
tangent space to $\he n$ at $x=e$ (see, e.g. \cite{GHL}, Proposition 1.72), 
starting from $\cov h$ we can define by left translation  a fiber bundle
over $\he n$  that we can still denote by $\cov h \simeq \bigwedge^hT^*\he{n}$. 
Moreover, a scalar product in $\mathfrak h$ induces a scalar product and a norm on $\bigwedge^h\mathfrak h$.

We can think of $h$-forms as sections of 
$\cov h$ and we denote by $\Omega^h$ the
vector space of all smooth $h$-forms.

\begin{itemize}
\medskip\item For $h=0,\ldots,2n+1$, the space of Rumin $h$-forms, $E_0^{h}$, is the space of smooth 
sections of a left-invariant subbundle of $\bigwedge^h\mathfrak h$ (that we still denote by $E_0^{h}$). 
Hence it inherits the inner product and the norm of $\bigwedge^h\mathfrak h$.
\medskip
\item If we denote by $\star$ the Hodge duality operator associated with the inner product in $E_0^\bullet$
and the volume form $dV$, then $\star E_0^h=E_0^{2n+1-h}$.

\medskip

In particular we have

\begin{remark}\label{mar 3 rem:1}
If $\alpha\in E_0^{h}$, then
$\star\star\alpha= (-1)^{(2n+1-h)h}\alpha = \alpha$. Thus
$$\alpha\wedge\phi =  \phi\wedge(\star\star\alpha) = \scal{\star\alpha}{\phi}\, dV.$$
Moreover, if $\beta\in E_0^{h}$
\begin{equation*}\begin{split}
\scal{\star\alpha}{\star\beta}\, dV&=\alpha\wedge\star\beta
=
(-1)^{h(2n+1-h)} \star\beta\wedge \alpha 
\\&
= \scal{\star\star\beta}{\alpha}\, dV
= \scal{\beta}{\alpha}\, dV = \scal{\alpha}{\beta}\, dV .
\end{split}\end{equation*}

\end{remark}

\medskip 

\item A differential operator $d_c:E_0^{h}\to E_0^{h+1}$ is defined. It is left-invariant, 
homogeneous with respect to group dilations. It is a first order homogeneous operator
in the horizontal derivatives in degree $\neq n$, whereas \emph{it is a second
order homogeneous horizontal operator in degree $n$}. 
\medskip\item Altogether, operators $d_c$ form a complex: $d_c\circ d_c=0$.
\medskip\item This complex is homotopic to de Rham's complex $(\Omega^\bullet,d)$. 
More precisely there exist a sub-complex $(E,d)$ of the de Rham complex and a suitable 
``projection'' $\Pi_E:\Omega^\bullet \to E^\bullet$
such that
 $\Pi_E$ is a differential operator of order $\le 1$ in the horizontal  derivatives.
\item $\Pi_{E}$ is a chain map, i.e.
$$
d\Pi_{E} = \Pi_{E}d.
$$
\item Let $\;\Pi_{E_0}$ be the orthogonal projection 
on $E_0^\bullet$.  Then
 $$\Pi_{E_0}\Pi_{E}\Pi_{E_0}=\Pi_{E_0}\qquad\mbox{and}\qquad \Pi_{E}\Pi_{E_0}\Pi_{E}=\Pi_{E}.$$
(we stress that $\Pi_{E_0}$ is an algebraic operator).
\item The exterior differential $d_c$ can be written as
$$d_c =\Pi_{E_0}d \Pi_E \Pi_{E_0}.$$
\medskip
 \item The $L^2$-formal adjoint $d_c^*$ of $d_c$ on $E_0^h$ satisfies 
\begin{equation}\label{jan 9 eq:2}
d_c^* = (-1)^h \star d_c\star.
\end{equation}

\end{itemize}

Let us list a bunch of notations for vector-valued function spaces (for the scalar case,
we refer to Section \ref{preliminary}).

\begin{definition} \label{dual spaces forms-no} If $U\subset \he n$ is an open set, $0\le h\le 2n+1$, $1\le p\le \infty$ and $m\ge 0$,
we denote by $L^{p}(U,\cov{h})$, $\mc E (U,\cov{h})$,  $\mc D (U,\cov{h})$, 
$W^{m,p}(U,\cov{h})$ (by $\WO{m}{p}{U,\cov{h}}$)
the space of all sections of $\cov{h}$ such that their
components with respect to a given left-invariant frame  belong to the corresponding scalar spaces.
The spaces $L^{p}(U,E_0^h)$, $\mc E (U,E_0^h)$,  $\mc D (U,\E_0^h)$,  
$W^{m,p}(U,E_0^{h})$ and $\WO{m}{p}{U,E_0^h}$ are defined in the same way. 

 In other words, for the definition of Sobolev spaces of nonnegative order,
we  identify a differential form with the vector-valued function of its coordinates
with respect to a left-invariant frame, keeping in mind that these coordinates belong
to $L^1_{\mathrm{loc}}(\he n)$. This makes possible to avoid the use of the
notion of currents associated with Rumin's complex (for this notion,
see, e.g., \cite{BFTT}, \cite{FSSC_advances},
\cite{vittone}). With this identification, $d_c$ and its $L^2$-formal adjoint $d_c^*$ are identified with
matrix-valued left-invariant differential operators.

%Finally, the spaces $W^{m,p}_{\mathrm{Euc}}(U,\cov{h})$,  $\stackrel{\circ}{W}\!\!^{m,p}_{\mathrm{Euc}} (U,\cov{h})$,
 %$W^{m,p}_{\mathrm{Euc}}(U,E_0^{h})$ and $\stackrel{\circ}{W}\!\!^{m,p}_{\mathrm{Euc}}(U,E_0^h)$
%are defined replacing Folland-Stein Sobolev spaces by usual Sobolev spaces.

Clearly, all these definitions
are independent of the choice of frame.
\end{definition}

When $d_c$ is second order (when acting on forms of degree $n$), $(E_0^\bullet,d_c)$ stops behaving like a differential module. This is the source of many complications. In particular, the classical Leibniz formula for the de Rham complex $d(\alpha\wedge\beta)=d\alpha\wedge\beta\pm \alpha\wedge d\beta$ 
is true in Rumin's complex only in special degrees, as shown in \cite{BFT2}, Proposition A.1 and 
\cite{PT}, Proposition 4.1.
However, in general, the Leibniz formula fails to hold (see \cite{BFT2}-Proposition A.7). This causes several technical difficulties when we want to localize our estimates by means of cut-off functions.

The following Leibniz' formula for Rumin's differential forms holds.

\begin{lemma} [see also \cite{BFP3}, Lemma 4.1] \label{leibniz} If $\zeta$ is a smooth real function, then the following formulae hold in the sense of distributions:
\begin{itemize}
\item[i)] if $h\neq n$, then on $E_0^h$ we have
$$
[d_c,\zeta] = P_0^h(W\zeta),
$$
where $P_0^h(W\zeta): E_0^h \to E_0^{h+1}$ is a linear homogeneous differential operator of order zero with coefficients depending
only on the horizontal derivatives of $\zeta$. If $h\neq n+1$, an analogous statement holds if we replace
$d_c$ { in degree $h$ with $d^*_c$ in degree $h+1$};
\item[ii)] if $h= n$, then on $E_0^n$ we have
$$
[d_c,\zeta] = P_1^n(W\zeta) + P_0^n(W^2\zeta) ,
$$
where $P_1^n(W\zeta):E_0^n \to E_0^{n+1}$ is a linear homogeneous differential operator of order 1 (and therefore
horizontal) with coefficients depending
only on the horizontal derivatives of $\zeta$, and where $P_0^h(W^2\zeta): E_0^n \to E_0^{n+1}$ is a linear homogeneous differential operator in
the horizontal derivatives of order 0
 with coefficients depending
only on second order horizontal derivatives of $\zeta$. If $h = n+1$, an analogous statement holds if we replace
$d_c$ { in degree $n$ with $d^*_c$ in degree $n+1$}.

\end{itemize}

\end{lemma}

 \begin{lemma}\label{by parts}
If $0\le h\le 2n$ and $\phi\in \mc D(\he n,E_0^h)$ and 
$\alpha\in W^{1,1}_{\mathrm{loc}}(\he n,E_0^{2n-h}) $, then
 we have:
\begin{equation}\label{cor main rumin 2}
\int_{\he n} d_c\alpha\wedge\phi = (-1)^h \int_{\he n}\alpha\wedge d_c\phi \, .
 \end{equation}
 
\end{lemma}

\begin{proof}
The statement follows by the formula of ``integration by parts'' of  \cite{BFTT}, Remarks
 2.16, together with an approximation argument. 

\end{proof}
The wedge of two closed forms of complementary degrees has the following property.
\begin{lemma}
\label{parabolic}
Let $1\le h\le 2n+1$ and let $\alpha\in L^p(\he n, E_0^h)$ and $\omega\in L^{p'}(\he n, E_0^{2n+1 -h})$ be  closed forms with $1<p <Q$    
and $\frac1p + \frac1{p'} = 1$. Then 
\begin{equation}\label{mar 14 eq:1}
\int \alpha\wedge\omega=0.
\end{equation}
 
In addition, if $p=Q$, then \eqref{mar 14 eq:1} still holds provided $\, 2\le h\le 2n+1$.
\end{lemma}

\begin{proof}
Since the differential $d_c$ is left-invariant, by mollification (convolution with smooth kernels), one can assume that $\alpha$ and $\omega$ are smooth. 

Suppose first $1<p<Q$. Without loss of generality, we can assume $h\neq n+1$ (indeed, if $h=n+1$, then $2n+1-h= n\neq n+1$ so we can consider $\omega$ in place of $\alpha$ in the following argument).
 By Theorem 1.1 of \cite{BFP2}, there exists $\phi \in L^{pQ/(Q-p)}(\he n, E_0^{h-1})$ such that
$d_c\phi = \alpha$ and
\begin{equation}\label{mar 7 eq:1}
\| \phi\|_{L^{pQ/(Q-p)}(\he n, E_0^{h-1})} \le C \| \alpha\|_{ L^p(\he n, E_0^h)}.
\end{equation}
If $N>0$, let now $\chi_N$ be a smooth cut-off function supported in $B(e,2N)$, $\chi_N\equiv 1$ on
$B(e,N)$, $|W \chi_N| \le 2/N$. Obviously   
$$
\int_{\he n} \chi_N\alpha \wedge\omega  \to \int_{\he n} \alpha \wedge\omega.
$$
On the other hand, by Lemma \ref{leibniz},
\begin{equation}\label{mar 6 eq:3}\begin{split}
\int_{\he n} & \chi_N\alpha \wedge\omega
\\& =
\int_{\he n} \chi_Nd_c\phi \wedge\omega
\\&
=
\int_{\he n}d_c(\chi_N \phi)\wedge \omega
+ \int_{\he n} P_0^{h-1}(W\chi_N )(\phi)\wedge\omega
\\&= \int_{\he n} P_0^{h-1}(W\chi_N )(\phi)\wedge\omega
\qquad\mbox{by \eqref{cor main rumin 2}.}
\end{split}\end{equation}

%d_c(\chi_N \phi)\wedge \omega \to \alpha\wedge \omega \qquad\mbox{ in $L^1(\he n, E_0^{2n+1})$.}
%\end{equation}
%Indeed, by Lemma \ref{leibniz}
%$$
%d_c(\chi_N \phi)= P_0^{h-1}(W\chi_N )\phi + \chi_N d_c\phi =P_0^{h-1}(W\chi_N )\phi + \chi_N \alpha.
%$$

We are then left to prove that
\begin{equation}\label{mar 6 eq:3-1}
\int_{\he n} P_0^{h-1}(W\chi_N )(\phi)\wedge\omega \to 0\qquad \mbox{as $N\to\infty$}.
\end{equation}
Obviously, 
$$
P_0^{h-1}(W\chi_N )(\phi)\wedge\omega \to 0\qquad \mbox{pointwise as $N\to\infty$,}
$$
since $W\chi_N\equiv 0$ on $B(e,N)$.
On the other hand, if $1<p<Q$, by H\"older inequality and \eqref{mar 7 eq:1},
\begin{equation*}\begin{split}
\int_{\he n} & |P_0^{h-1}(W\chi_N )| |\phi | |\omega| \, dV
\\&
\le \|P_0^{h-1}(W\chi_N )\|_{L^{Q}(\he n, E_0^{1})}\cdot
\|\phi\|_{L^{pQ/(Q-p)}(\he n, E_0^{h-1})} 
\cdot  \|\omega\|_{L^{p'}(\he n, E_0^{2n+1-h})},
\end{split}\end{equation*}
since $\frac1Q+\frac{Q-p}{pQ} + \frac{1}{p'}=1$.
But
$$
\|P_0^{h-1}(W\chi_N )\|_{L^{Q}(\he n, E_0^{1})} \le C\dfrac1N
|B(e,2N)|^{1/Q} \le C.
$$
This proves \eqref{mar 6 eq:3-1}, and hence \eqref{mar 14 eq:1}, when $1<p<Q$. 

The proof in the case $p=Q$ needs slightly different arguments. Indeed
we distinguish  the case $h=n+1$, i.e., $\alpha\in L^Q(\he n, E_0^{n+1})$, from the case $\alpha\in L^Q(\he n, E_0^{h})$ with $h\neq n+1$. In the latter case, by \cite{BFP5}, Theorem 1.8 (see also Theorem \ref{poincare infty} below), there exists $\phi \in L^{\infty}(\he n, E_0^{h-1})$ such that
$d_c\phi = \alpha$ and
\begin{equation}\label{mar 7 eq:2}
\| \phi\|_{L^{\infty}(\he n, E_0^{h-1})} \le C \| \alpha\|_{ L^Q(\he n, E_0^h)}.
\end{equation}
Hence the proof 
can be carried out in the same way, using the estimate
\eqref{mar 7 eq:2},  since $\frac1Q +\frac1\infty+ \frac{Q-1}{Q}=1$.
   Suppose now $\alpha\in L^Q(\he n, E_0^{n+1})$, hence $\omega\in L^{Q/(Q-1)}(\he n, E_0^{n})$.
   Since $1<\frac{Q}{Q-1} <Q$, we can apply   Theorem 1.1 of \cite{BFP2} to $\omega$,
   and we conclude arguing as we have done above for the case
   $1<p<Q$ (replacing $\alpha$ by $\omega$).

\end{proof}

 \begin{remark} The above lemma still holds for $p=1$: see  \cite{PT}.
 
 \end{remark}
\subsection{Rumin's Laplacian}

 \begin{definition}\label{rumin laplacian} 
In $\he n$, Rumin defined in \cite{rumin_jdg} 
the operators $\Delta_{\he{},h}$  acting on $E_0^h$,  by setting
\begin{equation*}
\Delta_{\he{},h}=
\left\{
  \begin{array}{lcl}
     d_cd^*_c+d^*_c d_c\quad &\mbox{if } & h\neq n, n+1;
     \\ (d_cd^*_c)^2 +d^*_cd_c\quad& \mbox{if } & h=n;
     \\d_cd^*_c+(d^*_c d_c)^2 \quad &\mbox{if }  & h=n+1.
  \end{array}
\right.
\end{equation*}

\end{definition}

Notice that $-\Delta_{\he{},0} = \sum_{j=1}^{2n}(W_j^2)$ is the usual sub-Laplacian of
$\he n$. 

We stress that   the order of $\Delta_{\he{},h}$ (with respect to group dilations) is 
$2$ if $h\neq n, n+1$ and  $4$ 
if $h=n, n+1$.
Notice also that, once a basis  of $E_0^h$
is fixed, the operator $\Delta_{\he{},h}$ can be identified with a matrix-valued map, still denoted
by $\Delta_{\he{},h}$
\begin{equation}\label{matrix form}
\Delta_{\he{},h} = (\Delta_{\he{},h}^{ij})_{i,j=1,\dots,N_h}: \mc D'(\he{n}, \rn{N_h})\to \mc D'(\he{n}, \rn{N_h}),
\end{equation}
where $\mc D'(\he{n}, \rn{N_h})$ is the space of vector-valued distributions on $\he n$, and $N_h$ is
the dimension of $E_0^h$ (see \cite{BBF}).

 { \begin{theorem}[see \cite{BFP1}, Theorem 4.6] \label{global solution}
If $0\le h\le 2n+1$, denote by $a$ the order of $\Delta_{\he{},h}$ with respect to group dilations (
$a=2$ if $h\neq n, n+1$ and  $a=4$ 
if $h=n, n+1$). Then
there exist
\begin{equation}\label{numero}
  K_{ij}\in\mc D'(\he{n})\cap \mc
C^\infty(\he{n} \setminus\{0\})
\qquad\mbox{for $i,j =1,\dots,N_h$,}
\end{equation}
  with the following properties:
\begin{enumerate}
\item[i)] if $a<Q$ then the $K_{ij}$'s are
kernels of type $a$,
 for
$i,j
=1,\dots, N_h$.
 If $a=Q$,
then the $K_{ij}$'s satisfy the logarithmic estimate
$|K_{ij}(p)|\le C(1+|\ln\rho(p)|)$ and hence
belong to $L^1_{\mathrm{loc}}(\he{n})$.
Moreover, their horizontal derivatives  $W_\ell K_{ij}$,
$\ell=1,\dots,2n$, are
kernels of type $Q-1$;
\item[ii)] when $\alpha = (\alpha_1,\dots, \alpha_{N_h}) \in
\mc D(\he{n},\rn {N_h})$ (here again $N_h = \mathrm{dim}\, E_0^h$),
if we set
\begin{equation}\label{numero2}
   \Delta_{\he{},h}^{-1} \alpha:=
\big(    
    \sum_{j}\alpha_j\ast  K_{1j},\dots,
     \sum_{j}\alpha_j\ast  K_{N_hj}\big),
\end{equation}
 then 
 $$ 
 \Delta_{\he{},h}\Delta_{\he{},h}^{-1}\alpha =  \alpha. 
 $$
Moreover, if $a<Q$, also 
$$
\Delta_{\he{},h}^{-1}\Delta_{\he{},h} \alpha =\alpha.
$$

\item[iii)] if $a=Q$, then for any $\alpha\in
\mc D(\he{n},\rn {N_h})$ there exists \break
$\beta_\alpha:=(\beta_1,\dots,\beta_{N_h})\in \rn{N_h}$,  such that
$$\Delta_{\he{},h}^{-1} \Delta_{\he{},h}\alpha - \alpha = \beta_\alpha.$$
\end{enumerate}
\end{theorem}

The following lemma states that $d_c$ and $\Delta_{\mathbb H, h}$ commute.
}

{  \begin{lemma}[\cite{BFP2}, Lemma 3.14]\label{his-her}
We notice that the Laplace operator commutes with the exterior differential $d_c$. More precisely, if $\alpha\in\mc C^\infty(\he n, E_0^h)$ and $n\ge 1$,
\begin{itemize}
\item[i)]$
d_c \Delta_{\mathbb H, h}\alpha = \Delta_{\mathbb H, h+1} d_c\alpha$, \qquad $h=0,1,\dots, 2n$, 
\qquad $h\neq n-1, n$.

\item[ii)]$
d_c d_c^* d_c \Delta_{\mathbb H, n-1}\alpha = \Delta_{\mathbb H, n} d_c\alpha$,
 \qquad ($h=n-1$).
\item[iii)] $d_c \Delta_{\mathbb H, n}\alpha = d_c d_c^*\Delta_{\mathbb H, n+1} d_c\alpha$ \qquad ($h=n$).
\item[iv)]$d_c d_c^* \Delta_{\mathbb H, n}\alpha = \Delta_{\mathbb H, n} d_c d_c^*\alpha$ \qquad ($h=n$).

\end{itemize}
\end{lemma}

Coherently with formula \eqref{matrix form}, the matrix-valued operator $\Delta_{\he{},h}^{-1}$ can be identified
with an operator (still denoted by $\Delta_{\he{},h}^{-1}$) acting on smooth compactly
supported differential forms in $\mc D(\he n, E_0^h)$. Moreover, when the notation
will not be misleading, we shall denote by $\Delta^{-1}_{\mathbb H,h }$  its kernel.

The commutation of $d_c$ and $d_c^*$ with $\Delta^{-1}_{\mathbb H,h}$ follows from the previous lemma:}
\begin{lemma}[\cite{BFP2}, Lemma 3.15]\label{comm} 
If $\alpha\in\mc D(\he n, E_0^h)$ and $n\ge 1$,
\begin{itemize}
\item[i)]$
d_c \Delta^{-1}_{\mathbb H, h}\alpha = \Delta^{-1}_{\mathbb H, h+1} d_c\alpha$, \qquad $h=0,1,\dots, 2n$, 
\qquad $h\neq n-1, n+1$.

\item[ii)] $d_c \Delta^{-1}_{\mathbb H, n-1}\alpha = d_c d_c^*\Delta^{-1}_{\mathbb H, n} d_c\alpha$ \qquad ($h=n-1$).

\item[iii)]$
d_c d_c^* d_c \Delta^{-1}_{\mathbb H, n+1}\alpha = \Delta^{-1}_{\mathbb H, n+2} d_c\alpha$,
 \qquad ($h=n+1$).

\item[iv)]$d_c^* \Delta^{-1}_{\mathbb H, h}\alpha = \Delta^{-1}_{\mathbb H, h-1} d_c^*\alpha$
 \qquad $ h=1,\dots, 2n+1$, \qquad $h\neq n, n+2$.
 
 \item[v)] $d_c^* \Delta^{-1}_{\mathbb H, n+2}\alpha = d_c^* d_c\Delta^{-1}_{\mathbb H, n+1} 
 d_c^*\alpha$ \qquad ($h=n+2$).

\item[vi)]$
d_c^* d_c d_c^*  \Delta^{-1}_{\mathbb H, n}\alpha = \Delta^{-1}_{\mathbb H, n-1} d_c^* \alpha$,
 \qquad ($h=n$).
\end{itemize}
\end{lemma}

\section{Poincar\'e inequalities in Heisenberg groups}
\label{Poincare}

In \cite{BFP2} we gave the definition of what we meant by  Poincar\'e inequalities for differential forms in $\he n$. 
We also distinguished from  global Poincar\'e inequalities versus its interior (or local) version. More precisely we give the following  definitions.
 \begin{definition}\label{poincare global def}

 If $1\le h\le 2n+1$ and $1\le p\le q \le \infty$, we say that the global $\he{}$-$\mathrm{Poincar\acute{e}}_{p,q}$ inequality holds in $E_0^h$ if 
there exists a constant $C$ such that,
 for every $d_c$-exact differential $h$-form $\omega$ in $L^p(\he n;E_0^h)$ there exists a differential $(h-1)$-form $\phi$ in $L^q(\he n,E_0^{h-1})$ such that $d_c\phi=\omega$ and
\begin{eqnarray}\label{H global poincare}
\|\phi\|_{L^{q}(\he n,E_0^{h-1})}\leq C\,\|\omega\|_{L^{p}(\he n,E_0^h)} \quad \mbox{ 
$\mathrm{global}\, \he{}$-$\mathrm{Poincar\acute{e}}_{p,q}(h)$.
}
\end{eqnarray}
\end{definition}

\begin{definition}\label{poincare def}

Given $1\le h\le 2n+1$ and $1\le p\le q \le \infty$, we say that the interior $\he{}$-$\mathrm{Poincar\acute{e}}_{p,q}$ inequality holds in $E_0^h$ if there exists $\lambda>1$ such that, if we set $B:=B(e,1)$ and $B_\lambda:=B(e,\lambda)$,
 there exists a constant $C>0$ such that,
 for every $d_c$-exact differential $h$-form $\omega$ in $L^p(B_\lambda;E_0^h)$ there exists a differential $(h-1)$-form $\phi$ in $L^q(B,E_0^{h-1})$ such that $d_c\phi=\omega$ and
\begin{eqnarray}\label{H poincare}
\|\phi\|_{L^{q}(B,E_0^{h-1})}\leq C\,\|\omega\|_{L^{p}(B_\lambda,E_0^h)} \quad 
\mbox{ 
$\mathrm{interior}\, \he{}$-$\mathrm{Poincar\acute{e}}_{p,q}(h)$.
}
\end{eqnarray}
\end{definition}

In \cite{BFP2} we  proved the validity of Poincar\'e inequalities when $p>1$, and in \cite{BFP3} we dealt with the case $p=1$. Finally, in
\cite{BFP5} we proved the following end-point Poincar\'e inequalities:

\begin{theorem}\label{poincare infty}
If \, $2\le  h \le { 2n}$,
we have:
\begin{itemize}
\item[i)] if $h \neq n+1$, then the global
$\he{}$-$\mathrm{Poincar\acute{e}}_{Q, \infty}(h)$ holds;
\item[ii)] if $h = n+1$, then the global
$\he{}$-$\mathrm{Poincar\acute{e}}_{n+1, \infty}(n)$ holds.
\end{itemize}
Analogous statements hold for interior Poincar\'e inequalities on $\he n$.
\end{theorem}

The aim of the present paper is to improve Theorem \ref{poincare infty}
by proving the following two statements, the first one of which is a global estimate and the 
second one that is  its local counterpart.
\begin{theorem}\label{main global}
\begin{itemize}
The following global Poincar\'e inequalities hold:
\item[i)] if $\,2\le  h \le { 2n+1}$, $h\neq n+1$, then a $d_c$-exact  form  $\omega\in L^Q(\he n, E_0^h)$
admits a primitive in $\phi\in \mc C_0(\he n, E_0^{h-1})$ such that
$$
\|\phi\|_{\mc C_0(\he n, E_0^{h-1})} \le C \|\omega\|_{L^Q(\he n, E_0^h)};
$$
\item[ii)] a $d_c$-exact  form  $\omega\in L^{Q/2}(\he n, E_0^{n+1})$
admits a primitive in $\phi\in \mc C_0(\he n, E_0^{n})$ such that
$$
\|\phi\|_{\mc C_0(\he n, E_0^{n})} \le C \|\omega\|_{L^{Q/2}(\he n, E_0^{n+1})}.
$$
\end{itemize}

\end{theorem}

%\begin{theorem}\label{main global bis} 
%Let $\Omega$ be a $n$-current identified with an element
%of  $( \BLh{1}{Q/(Q-1)}{n} )^*$ (i.e. of the form described in
%Proposition \ref{BL dual}). 
 %
 %If $\partial\Omega=0$, then there exists $\phi\in \mc C_0(\he n, E_0^{n})$
%such that $\partial T_\phi = \Omega$ and
%$$
%\|\phi\|_{C_0(\he n, E_0^{n})} \le C \|\Omega\|_{( \BLh{1}{Q/(Q-1)}{n} )^*}.
%$$
%
%
%
%\end{theorem}
\begin{remark}\label{r: dopo}
Statement ii) of 
Theorem \ref{main global} will follow from a more general result presented in Theorem \ref{main global bis} below.
\end{remark}

We are able to prove also an interior Poincar\'e inequality which is a local version of previous theorem (i.e.  that there exists a continuous primitive). In the following statements, exact means the $d$ or $d_c$ of a form with distributional coefficients.

\begin{theorem}\label{main local}
There exists $\lambda>1$ such that, if we set $B:=B(e,1)$ and $B_\lambda:=B(e,\lambda)$,   
there exists $C>0$:
\begin{itemize}
\item[i)] if $2\le  h \le { 2n+1}$, $h\neq n+1$, then a $d_c$-exact  form  $\omega\in L^Q(B_\lambda, E_0^h)$
admits a primitive in $\phi\in \mc C(\overline B, E_0^{h-1})$ such that
\begin{equation}\label{feb 19 eq:3}
\|\phi\|_{\mc C(\overline B, E_0^{h-1})} \le C \|\omega\|_{L^Q(B_\lambda, E_0^h)};
\end{equation}
\item[ii)] a $d_c$-exact form  $\omega\in L^{Q/2}(B_\lambda, E_0^{n+1})$
admits a primitive in $\phi\in \mc C(\overline B, E_0^{n})$ such that
\begin{equation}\label{feb 19 eq:4}
\|\phi\|_{\mc C(\overline B, E_0^{n})} \le C \|\omega\|_{L^{Q/2}(B_\lambda, E_0^{n+1})}.
\end{equation}
\end{itemize}

\end{theorem}

A result for general compact Riemannian and contact subRiemannian manifolds follows.

\begin{theorem}\label{main compact}
Let $M$ be a compact $m$-dimensional manifold.
\begin{itemize}
\item[i)] First assume that $M$ is equipped with a smooth Riemannian metric. There exists $C>0$ such that every exact $h$-form $\omega$ on $M$ with $L^m$ coefficients, $h\ge 2$, admits a continuous primitive $\phi$ such that
\begin{equation}\label{compacti}
\|\phi\|_{\mc C(M, \bigwedge^{h-1})} \le C \|\omega\|_{L^m(M, \bigwedge^h)}.
\end{equation}

Next assume that $m=2n+1$ is odd, that $M$ carries a contact structure and a smooth subRiemannian metric. Let $Q=2n+2$.
\item[ii)] if $2\le  h \le { 2n+1}$, $h\neq n+1$, then every $d_c$-exact  form  $\omega\in L^Q(M, E_0^h)$
admits a primitive in $\phi\in \mc C(M, E_0^{h-1})$ such that
\begin{equation}\label{compactii}
\|\phi\|_{\mc C(M, E_0^{h-1})} \le C \|\omega\|_{L^Q(M, E_0^h)};
\end{equation}
\item[iii)] every $d_c$-exact form  $\omega\in L^{Q/2}(M, E_0^{n+1})$
admits a primitive in $\phi\in \mc C(M, E_0^{n})$ such that
\begin{equation}\label{compactiii}
\|\phi\|_{\mc C(M, E_0^{n})} \le C \|\omega\|_{L^{Q/2}(M, E_0^{n+1})}.
\end{equation}
\end{itemize}

\end{theorem}

%{\color{green}
%\begin{definition} \label{equiv global Sobolev}
%If $1\le h\le 2n+1$, $1\le p \le q \le \infty $,
%we say that the global
%$\he{}$-$\mathrm{Sobolev}_{p,q}(h)$ inequality holds if there exists a constant $C$ such that 
% for every compactly supported $d_c$-exact differential $h$-form $\omega$ in $L^p(\he n;E_0^h)$ there exists a 
%compactly supported differential $(h-1)$-form $\phi$ in $L^q(\he n,E_0^{h-1})$ such that $d_c\phi=\omega$ 
%and
%\begin{eqnarray}\label{H global Sobolev}
%\|\phi\|_{L^{q}(\he n,E_0^{h-1})}\leq C\,\|\omega\|_{L^{p}(\he n,E_0^h)} \quad 
%\mbox{ 
%$\mathrm{global}\, \he{}$-$\mathrm{Sobolev}_{p,q}(h)$.
%}
%\end{eqnarray}
%\end{definition}
%
%\begin{theorem}
%\label{sobolev q infty}
%If\, $2\le  h \le { 2n+1}$,
%we have:
%\begin{itemize}
%\item[i)] if $h \neq n+1$, then the interior
%$\he{}$-$\mathrm{Sobolev}_{Q, \infty}(h)$ holds;
%\item[ii)] if $h = n+1$, then the interior
%$\he{}$-$\mathrm{Sobolev}_{Q/2, \infty}(n)$ holds.
%\end{itemize}
%Analogous statements hold for global Sobolev inequalities on $\he n$.
%\end{theorem}
%}

\section{Bourgain-Brezis' duality argument for the global Poincar\'e inequality}\label{BB duality}

The basic idea for proving the Poincar\'e inequalities stated in Theorem \ref{main global} is a duality argument inspired by Bourgain-Brezis.

Let us  collect here a few facts of functional analysis as stated in Brezis' book \cite{brezis}.

\begin{theorem}[\cite{brezis}, Section 2.7]
\label{predual}
Let $A:\mc D(A)\subset E\to F$ be a closed unbounded operator between Banach spaces. 
Assume that $\mc D(A)$ is dense in $E$. Then the adjoint $A^*:\mc D(A^*)\subset F^* \to E^*$ is uniquely defined, and closed. 

1. The following are equivalent:
\begin{enumerate}
  \item $A(\mc D(A))\subset F$ is closed.
  \item $A^*(\mc D(A^*))\subset E^*$ is closed.
  \item $A(\mc D(A))=\mathrm{Ker}(A^*)^\perp$.
  \item $A^*(\mc D(A^*))=\mathrm{Ker}(A)^\perp$.
\end{enumerate}

2. The following are equivalent:
\begin{enumerate}
  \item $\exists C$, $\forall f^*\in \mc D(A^*)$, $\|f^*\|_{F^*}\leq C\,\|A^*f^*\|_{E^*}$.
  \item $A(\mc D(A))=F$.
  \item $A^*(\mc D(A^*))$ is closed and $\mathrm{Ker}(A^*)=0$.
\end{enumerate}

3. The following are equivalent:
\begin{enumerate}
  \item $\exists C$, $\forall e\in \mc D(A)$, $\|e\|_{E}\leq C\,\|Ae\|_{F}$.
  \item $A^*(\mc D(A^*))=E^*$.
  \item $A(\mc D(A))$ is closed and $\mathrm{Ker}(A)=0$.
\end{enumerate}
\end{theorem}

Combining all items of Theorem \ref{predual}, one gets

\begin{lemma}
\label{predualpoin}
Let $A:\mc D(A)\subset E\to F$ be a closed unbounded operator between Banach spaces. 
Assume that $\mc D(A)$ is dense in $E$. Then the following are equivalent.
\begin{enumerate}
  \item $A(\mc D(A))\subset F$ is closed.
  \item $\exists C$, $\forall e^*\in A^*(\mc D(A^*))$, $\exists f^*\in \mc D(A^*)$, $A^*f^*=e^*$ and $\|f^*\|_{F^*}\leq C\,\|e^*\|_{E^*}$.
   \item $A^*(\mc D(A^*))\subset E^*$ is closed.
    \item $\exists C$, $\forall f\in A(\mc D(A))$, $\exists e\in \mc D(A)$, $Ae=f$ and $\|e\|_{E}\leq C\,\|f\|_{F}$.
\end{enumerate}
\end{lemma}

\begin{proof}
Let $\tilde F$ denote the closure of $A(\mc D(A))$ in $F$. Let $\mc D(B)=\mc D(A)$ and $B:\mc D(B)\to \tilde F$ coincide with $A$. Then $B$ is a closed operator and its domain is dense. Since $A(\mc D(A))^\perp=\mathrm{Ker}(A)$, $\mc D(B^*)$ identifies with the quotient $\mc D(A^*)/\mathrm{Ker}(A^*)$, and $A^*$ factors via this quotient, yielding $B^*:\mc D(B^*)\to E^*$. Since $\mathrm{Ker}(B^*)=0$, Lemma \ref{predual} 2 gives
\begin{align*}
B(\mc D(B))&=\tilde F \iff \exists C,\,\forall [f^*]\in \mc D(B^*),\,\|[f^*]\|_{F^*/\mathrm{Ker}(A^*)}\leq C\,\|B^*[f^*]\|_{E^*}.
\end{align*}
Since $B(\mc D(B))=A(\mc D(A))$, this translates into (1)$\iff$(2).

Let $\tilde E=E/\mathrm{Ker}(A)$. Since $A$ is closed, $\mathrm{Ker}(A)$ is a closed subspace of $E$, so $\tilde E$ is a Banach space. Let $\mc D(\tilde B)=\mc D(A)/\mathrm{Ker}(A)$. The operator $A$ factors via this quotient, yielding $\tilde B:\mc D(\tilde B)\to F$ which is a closed operator, with dense domain. Since $\mathrm{Ker}(\tilde B)=0$, Lemma \ref{predual} 3 gives
\begin{align*}
\tilde B(\mc D(\tilde B))&=\tilde E \iff \exists C,\,\forall [e]\in \mc D(\tilde B),\,\|[e]\|_{E/\mathrm{Ker}(A)}\leq C\,\|\tilde B[e]\|_{F}.
\end{align*}
Since $\tilde B(\mc D(\tilde B))=A(\mc D(A))$, this translates into (3)$\iff$(4).

By Lemma \ref{predual} 1, (1)$\iff$(3), so all four properties are equivalent.

\end{proof}

Next on the paper, we apply these general facts to the exterior differential on spaces of differential forms. Namely, let us denote by $\mc C_0$ the Banach space of continuous functions vanishing at infinity
with the $L^\infty$-norm; in this paper, we shall deal with the space
$$
E:= \mc C_0(\he n, E_0^{h-1}),
$$
and the operator $A$ will be  the Rumin's exterior differential $d_c$ (in a suitable weak sense). 

In any case, the first step we have to do, if we want to use a duality argument, is to identify the dual $E^*$ of $E$; this is done in the next section where we prove that $E^*$ can be
identified with the set of
currents with finite mass.

\section{Currents and measures}  \label{currents and measures} 

\begin{definition}\label{corrente in gen} Let $0\le h\le 2n+1$.
 If $\Omega\subset\G$ is an open set,
we say that $T$ is a  $h$-current on $\Omega$
if $T$ is a continuous linear functional on $\mc D(\Omega, E_0^h)$
endowed with the usual topology. We write $T\in \mc D'(\Omega, E_0^h)$
and we denote by $\phi\to\Scal{T}{\phi}$ its action on $\mc D(\Omega, E_0^h)$.

The definition of $\mc E'(\Omega, E_0^h)$ (compactly supported currents) is given
analogously.

If $T\in \mc D'(\Omega, E_0^h)$, $h\ge 1$,
 we define its boundary $\partial_c T\in \mc D'(\Omega, E_0^{h-1})$
 by the identity
 $$
 \Scal{\partial_c  T}{\phi} = \Scal{T}{d_c\phi}\qquad \mbox{for all $\mc D(\Omega, E_0^{h-1})$.}
 $$
\end{definition}

\begin{proposition}\label{corrente by distribuzione}
If $\Omega \subset \G$ is an open set, and
$T\in\mathcal D'(\Omega)$ is a (usual) distribution, then
$T$ can be identified canonically with a $(2n+1)$-current $\tilde
T \in  \mc D'(\Omega, E_0^{2n+1})$ through the formula
\begin{equation}\label{cbd}
\Scal{\tilde T}{\alpha}:=\Scal{T}{\ast \alpha}
\end{equation}
for any $\alpha\in \mc D(\Omega, E_0^{2n+1})$.
Conversely, by (\ref{cbd}), any $(2n+1)$-current $\tilde T$
can be identified with an usual distribution $T \in D'(\Omega)$.
\end{proposition}

\begin{proof}
See \cite{dieudonne}, Section 17.5, and \cite{BFT1},
Proposition 4.
\end{proof}

Following \cite{federer}, 4.1.7, we give the following definition.
\begin{definition}
If $T\in \mc D'(\Omega, E_0^{2n+1})$, and $\varphi \in 
\mc E(\Omega, E_0^h)$, with $0\le h\le 2n+1$, we define $T\res
\varphi\in \mathcal \mc D'(\Omega, E_0^{2n+1-h})$ by the identity
$$
\Scal{T\res \varphi}{\alpha}:=\Scal{T}{ \alpha\wedge\varphi}
$$
for any $\alpha \in \mc D(\Omega, E_0^{2n+1-h})$.

%We notice that, following \cite{dieudonne},
% an alternative notation for $T\res \varphi$
% could be $T\Wedge \varphi$.
\end{definition}

The following result is taken from \cite{BFT1}, Propositions 5 and 6,
and Definition 10, but we refer also to \cite{dieudonne}, Sections 17.3,
17.4 and 17.5.
%{\color{yellow}
%\begin{proposition}\label{corrente in coordinate}
%Let $\Omega \subset \G$ be an open set.
%If $1\le h\le n$,  $\Xi_0^h=\{\xi_1^h,\dots \xi^h_{\mathrm{dim}\;E_0^h}\}$ is a left invariant  
%basis
%of $E_0^h$ and $T\in \mc D'(\Omega, E_0^{h})$, then
%\begin{enumerate}
%\item[i)] there
%exist (uniquely determined)
%$T_1,\dots,T_{\mathrm{dim}\;E_0^h}\in \mathcal D'(\Omega)$ such that we
%can write
%$$
%T=\sum_j\tilde T_j\res (\ast\xi^h_j);
%$$
%%Notice that $\ast\xi_j\Wedge\alpha$ is well defined for any
%%$\alpha\in   \mc D_{\he{}}^{m}(\mathcal U)$, for $\ast\xi_j\in
%%\mc E_{\he{}}^{2n+1-m}(\mathcal U)$, and the pair
%%$(2n+1-m,m)$ is always wedge--admissible, since
%%$2n+1-m$ and $m$ can not be both less than $n$;
%\item[ii)] if
%$\alpha\in E(\Omega, E_0^{h})$, then $\alpha$ can be
%identified canonically with a $h$-current $T_\alpha$ through the
%formula
%$$
%\Scal{T_\alpha}{\beta}:=\int_{\Omega}\ast\alpha\wedge\beta
%$$ for any $\beta\in \mc D(\Omega, E_0^{h})$.
%%Notice that $\ast \alpha \Wedge \beta$ is well defined.
%Moreover, if
%$\alpha=\sum_j\alpha_j\xi_j^h$ then
%$$
%T_\alpha=\sum_j\tilde\alpha_j\res(\ast\xi_j^h);
%$$
%\item[iii)] we
%say that $T $ is smooth
%in $\mathcal U$ when $T_1,\dots,T_{\mathrm{dim}\;E_0^h}$
%are (identified with) smooth functions.
%This is clearly equivalent to say that there exists $\beta\in
%E(\Omega, E_0^{h})$ such that
% $$
% \Scal{T}{\alpha}=\int_{\mc U}\scal{\beta}{\alpha}\,dV
% $$
% for any $\alpha\in  \mc D(\Omega, E_0^{h})$
% (in fact, we choose $\beta=\sum_jT_j\xi_j^h$).
%\end{enumerate}
%\end{proposition}
%}

\begin{definition} Let $1\le h\le 2n+1$. A form $\alpha\in L^1_{\mathrm{loc}}(\he n,E_0^{2n+1-h})$
 can be identified with a current $T_\alpha\in \mc D'(\he n, E_0^h)$
through the action
\begin{equation}\label{nov5 eq:1}
\Scal{T_\alpha}{\phi} := \int_{\he n}\alpha\wedge\phi
= \int_{\he n}\phi\wedge\alpha
\qquad\mbox{for $\phi\in \mc D(\he, E_0^h)$.}
\end{equation}
Indeed $\alpha\wedge\phi = (-1)^{h(2n+1-h)}\phi\wedge\alpha$, and
$2n+1-h$ is even if $h$ is odd.
\end{definition}

%If $T\in \mc D'(\he n, E_0^h)$, then its boundary $\partial T 
%\in \mc D'(\he n, E_0^{h-1})$ is defined by
%$$
%\Scal{\partial T}{\phi}:= \Scal{T}{d_c\phi}\qquad\mbox{for 
%all $\phi\in \mc D(\he, E_0^{h-1})$.}
%$$
\begin{remark}\label{mar 2 rem:1}If $\alpha\in L^1_{\mathrm{loc}}(\he n,E_0^{2n+1-h})$, we notice
that $\partial_c  T_\alpha=0$ if and only if $d_c \alpha=0$ in the
sense of distributions. Indeed, if $\phi\in \mc D(\he n, E_0^{h-1})$
(and hence $\star \phi\in \mc D(\he n, E_0^{2n+2-h})$),
keeping in mind Remark \ref{mar 2 rem:1}, we have
\begin{equation*}\begin{split}
\Scal{\partial_c  T_\alpha}{\phi} & = \Scal{T_\alpha}{d_c\phi} = \int_{\he n} \alpha\wedge d_c\phi
= \int_{\he n}\scal{\star\alpha}{d_c\phi}\, dV
\\&
=  \int_{\he n}\scal{\star\alpha}{\star\star d_c\phi}\, dV =  \int_{\he n}\scal{\alpha}{\star d_c\phi}\, dV 
=  \int_{\he n}\scal{\alpha}{\star d_c\star\star\phi}\, dV 
\\&
=  (-1)^ h \int_{\he n}\scal{\alpha}{d_c^*\star\phi}\, dV 
\end{split}\end{equation*}

\end{remark}

\begin{proposition}\label{oct19 prop:1}
 If $\alpha\in L^1_{\mathrm{loc}}(\he n,E_0^{2n+1-h}) \cap d_c^{-1}(L^1_{\mathrm{loc}}(\he n,E_0^{2n+2-h}))$,
then
$$
\partial_c  T_\alpha = (-1)^{h-1}T_{d_c\alpha}.
$$

\end{proposition}

\begin{proof}
If $\phi\in \mc D(\he, E_0^{h-1})$, then
\begin{equation*}\begin{split}
\Scal{\partial_c  T_\alpha}{\phi} &= \Scal{T_\alpha}{d_c \phi} =  \int_{\he n}\alpha\wedge d_c\phi
\\&
= (-1)^{h-1} \int_{\he n} d_c \alpha\wedge\phi= (-1)^{h-1}\Scal{T_{d_c\alpha}}{\phi}.
\end{split}\end{equation*}

\end{proof}

The notion of convolution can be extended by duality
to currents.
\begin{definition}\label{regolarizzazione di una corrente}  Let $0\le h\le 2n+1$ and
 $\varphi\in\mc D(\he n)$. Let
$T\in \mc E'(\he n, E_0^{h})$ be given, and
 denote by $\ccheck \varphi$ the function defined by
${\phantom o}^{\mathrm v} \varphi(p) := \varphi(p^{-1})$.
Then we set
$$
\Scal{\varphi *T}{\phi}:= \Scal{T}{\ccheck \varphi  * \phi}
$$
for any $\phi\in  \mc D(\he n, E_0^{h})$.
\end{definition}

\begin{lemma} If $\alpha\in L^1_{\mathrm{loc}}(\he n, E_0^h)$, $1\le h\le 2n+1$,
and $J= \ccheck J\in\mc D(\he n)$, then
\begin{equation}\label{oct17 eq:1}
J\ast\partial_c  T_\alpha = {\partial_c (J\ast T_\alpha)}=\partial_c  T_{J\ast\alpha}.
\end{equation}

\begin{proof}
By definition, $T_\alpha$ is a $(2n+1-h)$-current. If $\phi\in \mc D(\he n, E_0^{2n-h})$ is a test form,
then, by Definition \ref{regolarizzazione di una corrente}, 
\begin{equation*}\begin{split}
\Scal{J\ast\partial_c  T_\alpha}{\phi} &= \Scal{\partial_c  T_\alpha}{J\ast\phi} = \Scal{T_\alpha}{d_c(J\ast\phi)}
= \Scal{T_\alpha}{J\ast d_c\phi} 
\\&
= \Scal{J\ast T_\alpha}{ d_c\phi} = \Scal{\partial_c (J\ast T_\alpha)}{ \phi}.
\end{split}\end{equation*}

On the other hand,
\begin{equation*}\begin{split}
\Scal{\partial_c  T_{J\ast\alpha}}{\phi} &= \Scal{ T_{J\ast\alpha}}{d_c\phi} = \int\scal{\star(J\ast\alpha)}{ d_c\phi}dV
\\&
= \int\scal{J\ast(\star\alpha)}{ d_c\phi}dV = \int\scal{\star\alpha}{J\ast d_c(\phi)}dV
\\&
=\Scal{ T_{\alpha}}{J\ast d_c\phi}=\Scal{J\ast T_{\alpha}}{ d_c\phi}=\Scal{\partial_c (J\ast T_{\alpha})}{ \phi}.
\end{split}\end{equation*}

\end{proof}

\end{lemma}

\begin{definition} If $T\in \mc D'(\he n,E_0^h)$, we define its mass $\mc M(T)$ by
$$
\mc M(T):= \sup\{\Scal{T}{\phi},\, \phi\in \mc D(\he n,E_0^h),\, |\phi|\le 1\}.
$$
\end{definition}

\begin{lemma}\label{nov7 lemma:1}
If  $\alpha\in  L^1(\he n, E_0^{2n+1-h})$, then
\begin{equation}\label{mass L^1}
\mc M(T_\alpha) = \|\alpha\|_{L^1(\he n, E_0^{2n+1-h}).}
\end{equation}

Moreover, if $\alpha\in L^1_{\mathrm{loc}}(\he n, E_0^{2n+1-h})$ and $\mc M(T_\alpha)
<\infty$, then 
$$
\alpha\in  L^1(\he n, E_0^{2n+1-h})\qquad\mbox{
and $\eqref{mass L^1}$ holds.}
$$

\end{lemma}

\begin{proof} The first assertion is basically contained in Exercise 4.26 of \cite{brezis}.
Suppose now $\alpha\in L^1_{\mathrm{loc}}(\he n, E_0^{2n+1-h})$ and $\mc M(T_\alpha)
<\infty$. Consider now a sequence of compactly supported smooth functions $(\rho_k)_{k\in\mathbb N}$,
$0\le \rho_k\le 1$, $\rho_k \to 1$ in $\he n$ as $k\to\infty$ and set $\alpha_k:=\rho_k\alpha
\in L^1(\he n, E_0^{2n+1-h})$.
If $\phi\in \mc D(\he n, E_0^h)$, $\|\phi\|_{L^\infty} \le 1$, then
$$
\Scal{T_{\alpha_k}}{\phi} = \int_{\he n} \rho_k\alpha\wedge \phi = \int_{\he n} \alpha\wedge \rho_k\phi
\le \mc M(T_\alpha),
$$
so that $\mc M(T_{\alpha_k})\le \mc M(T_\alpha)$, and the assertion follows by Fatou's lemma.

\end{proof}

As for Federer-Fleming currents, the mass of currents is lower semicontinuous
with respect to weak* convergence.

\begin{lemma}\label{lsc} If $1\le h\le 2n+1$, $T,T_k\in \mc D'(\he n, E_0^h)$ for $k\in\mathbb N$
and $T_k\to T$ in  $\mc D'(\he n, E_0^h)$, then
$$
\mc M(T)\le \liminf_k M(T_k).
$$

\end{lemma}

 If $1\le h\le 2n+1$,  and $\Xi_0^h$ is a left invariant basis
of $E_0^h$, then the linear maps on $E_0^h$
$$
\alpha  \to (\xi_j^h)^*(\alpha) := \star (\alpha\wedge \star \xi_j^h)
$$
belong to $(E_0^h)^*$ (the dual of $E_0^h$ as a finite-dimensional vector space) and
$$
(\xi_j^h)^*(\xi_i^h) = \star (\xi_i^h \wedge \star \xi_j^h) = \delta_{i,j} \star dV = \delta_{i,j},
$$
i.e. $(\Xi_0^h)^*=\{(\xi_1^h)^* ,\dots, (\xi^h_{N_h})^* \}$ is a left invariant  
dual basis of $(E_0^h)^*$.

Let us also remind the notion of distribution section of a finite-dimensional
vector bundle $\mc F$: a distribution section is a continuous linear
map on the {space of} compactly supported sections of the dual vector bundle
$\mc F^*$ (see, e.g., \cite{treves}, p. 77). Then we can give the following remark.

\begin{remark}\label{current simply}

Let $T$ be a current on $E_0^h$, 
$$
T=\sum_j\tilde T_j\res (\star\xi^h_j),
$$
where $T_1,\dots,T_{N_h}\in \mathcal D'(\Omega)$.
Then $T$ can be seen as a section of $(E_0^h)^*$.
Indeed, if $\alpha=\sum_i \alpha_i \xi_i^h\in \mc D(\Omega, E_0^h)$
\begin{equation*}\begin{split}
\Scal{T}{\alpha} & = \sum_{j} \Scal{\tilde T_j\res  (\star\xi_j^h)}{\alpha} 
= \sum_{j} \Scal{\tilde T_j}{ \alpha\wedge (\star\xi_j^h)}
\\&
=  \sum_{j} \Scal{T_j}{ \alpha_j} =  \sum_{i,j} \Scal{T_j}{ (\xi_j^h)^*(\alpha_i\xi_i^h)}
 = \sum_j   \Scal{T_j }{(\xi_j^h)^*(\alpha)},
 \end{split}\end{equation*}
 where the dualities in the first line are meant as dualities between currents
 and test forms, where the dualities in the second line are meant as dualities {between}
 distributions and test functions.
Thus we can write
formally 
\begin{equation}\label{Aug11 eq:4}
T=\sum_j T_j (\xi_j^h)^*
\end{equation}
and we can identify $T$ with a vector-valed distribution $(T_1,\dots,T_{N_h})$.

We notice also that, if $\alpha=\sum_j\alpha_j \xi_j\in \mc E(\Omega, E_0^h)$,
then
$$
T_\alpha=\sum_j \alpha_j (\xi_j^h)^*.
$$
\end{remark}

We are now in position to identify the dual of $\mc C_0(\he n,E_0^h)$.
\begin{remark}\label{riesz} Let $\,T\in \mc D'(\he n, E_0^{h})$ be a current of finite mass $\mc M(T)$.
By density of $\mc D(\he n,E_0^h)$ in $\mc C_0(\he n,E_0^h)$, $T$ can be continued as a linear
bounded functional on $\mc C_0 (\he n,E_0^h)$ with norm $\mc M(T)$. For sake of simplicity
we still denote by $T$ its extension and by $ \Scal{T}{\phi}_{\mc C_0^*,\mc C_0}$ the action of the extension on
$\phi \in \mc C_0(\he n,E_0^h)$.
By Riesz' representation
theorem (see e.g. \cite{AT}, Theorem 1.2.4)
there exists
a vector-valued Borel measure $(\mu_1,\dots,\mu_{N_h})\in \mc D'(\he n)^{N_h}$ such that, for any 
$\phi \in \mc C_0(\he n,E_0^h)$ (identified as above with  $(\phi_1,\dots,\phi_{N_h})\in \mc C_0(\he n)^{N_h}$),
\begin{equation}\label{jan 27 eq:1}
\Scal{T}{\phi}_{\mc C_0^*,\mc C_0} = \sum_j \int_{\he n} \phi_j\,d\mu_j.
\end{equation}
In addition
\begin{equation}\label{jan 24 eq:1}
\mc M(T) =\|T\|_{\mc C_0^*}= |\mu|(\he n).
\end{equation}

\medskip

Conversely, if $\mu=(\mu_1,\dots,\mu_{N_h})$ is a finite vector-valued Borel measure,
the map $T:\mc D(\he n,E_0^h)\to \mathbb R$ defined by
\begin{equation}\label{nov 9 eq:1}
T: \sum_j \phi_j \xi_j \to \sum_j \int_{\he n} \phi_j\,d\mu_j
\end{equation}
is a $h$-current of finite mass $\mc M(T)= |\mu|(\he n)$.

In particular, if $T\in C_0^*$ (and therefore admits an expression as in \eqref{nov 9 eq:1})
\begin{equation}\label{nov9 eq:2}
\Scal{T}{\phi}_{\mc C_0^*,\mc C_0} = \Scal{T}{\phi}_{\mc D',\mc D}  \quad\mbox{for all $\phi\in \mc D(\he n, E_0^h)$.}
\end{equation}

\end{remark}

 \subsection{Poincar\'e and Sobolev inequalities for currents via an approximation result}\label{s: AG}

Theorem $1.1-(2)$ in \cite{BFP3} contains the following result.

\begin{theorem}[Global Poincar\'e and Sobolev inequalities in degree $h\neq n+1$]\label{poincareglobal}
 Let $h=1,\ldots,2n$, $h\neq n+1$. For every $d_c$-exact $h$-form $\omega\in L^1(\he n, E_0^h)$,
  there exists an $(h-1)$-form $\phi\in L^{Q/(Q-1)}(\he n,E_0^{h-1})$, such that 
$$
d_c\phi=\omega \qquad\mbox{and}\qquad \|\phi\|_{L^{Q/(Q-1)}}(\he n, E_0^{h-1})\leq C\,\|\omega\|_{L^1(\he n, E_0^{h})}.
$$
Furthermore, if $\omega$ is compactly supported, so is $\phi$.
\end{theorem}

\begin{remark} As in \cite{BFP3}, Section 1.2, we stress that Poincar\'e inequality fails to hold
in top degree (see also \cite{PT}). 

\end{remark}

We show  that the previous result can be reformulated in terms of currents. This is done after we prove the following approximation result
(see also \cite{baldi_pini} for Euclidean currents).

 \begin{proposition}\label{AG}  If $1\le h\le 2n+1$ and $\,T\in \mc D'(\he n, E_{0}^h)$ is a current of finite mass $\mc M(T)$ (identified with a vector-valued
measure as in Proposition \ref{riesz}), then there exists
a family $(\omega^\eps)_{\eps>0}$ of forms in $\mc E(\he n, E_0^{h})\cap L^1(\he n, E_0^h)$ 
such that, if we set $T_\eps:= T_{\star \omega^\eps}\in \mc D'(\he n, E_{0}^h)$,
then
\begin{itemize} 
\item[i)] $\mc M(T_{\eps})<\infty$ for all $\eps>0$;
\item[ii)]  $(T_\eps)_{s>0}$ converges weakly* to $T$ as $\eps\to 0$;
\item[iii)] $\|\omega^\eps\|_{L^1(\he n, E_0^h)}=\mc M(T_{\eps}) \to \mc M(T)$ as $\eps\to 0$;
 \item[iv)] if $T=\partial_c  S$, $S\in \mc D'(\he n, E_{0}^{h+1})$, then $d_c (\star\omega^\eps) =0$ for all $\eps>0$.
\end{itemize}

If $T$ is 
 compactly supported, then the $\omega^\eps$'s are supported in a neighborhood of $\supp T$.
 \end{proposition}

\begin{proof} 
%Let us suppose first that $T$ is compactly supported.
We
denote by $\mu$ the finite vector-valued Borel measure associated with $T$ as in Proposition
\ref{riesz}. Arguing as in \cite{ambrosio_fusco_pallara}, Section 2.1, if $\eps>0$ let $J_\eps = \ccheck J_\eps$ 
is an usual (group) Friedrichs' mollifier, we define the fuctions
$$
f^\eps(\eta):=  J_\eps\ast \mu (\eta) := \int_{\he n}J_\eps(\eta\cdot p^{-1})\, d\mu(p)
= \int_{\he n}J_\eps(p\cdot \eta^{-1})\, d\mu(p).
$$
%We stress that $\mu^\eps(p)$ coincides with the regularization of $\mu$ as
%vector-valued distribution in $(\mc D'(\he n))^{N_h}$. 
We point out that the family of functions $\{f^\eps\, , \, \eps >0\}$ is 
bounded in $( L^1(\he n))^{N_h}$.
Indeed, if $j=1,\dots,N_h$
\begin{equation}\label{Feb 2 eq:1}\begin{split}
\int_{\he n}  & |f^\eps(\eta)|\, d\eta = \int_{\he n} \big| \int_{\he n}J_\eps(\eta\cdot p^{-1})\, d\mu(p)\big| \, d\eta
\\&
\le \int_{\he n} \big( \int_{\he n}J_\eps(\eta\cdot p^{-1})\, d|\mu|(p)\big) \, d\eta
\\&
=\int_{\he n} \big( \int_{\he n}J_\eps(\eta\cdot p^{-1})\, d\eta \big) \, d|\mu|(p)
=
|\mu|(\he n).
\end{split}\end{equation}

We notice that, for all $\eps>0$, we can associate with $f^\eps$ the measure
$\mu^\eps(p):=f^\eps\, d\mathcal L^{2n+1}$.
%\footnote{We use the same symbol $\mu^\eps$ for the .vector-valued Radon measure
%and for its density with respect to Lebesgue measure $\mathcal L^{2n+1}$
%as long as this
%doen't lead to a misunderstanding.}
Moreover,  $f^\eps\in  (\mc E(\he n))^{N_h}$.  Indeed, 
%all the $f^\eps$'s are supported in a 
%neighborhood of $\supp T$; in addition, 
if $j=1,\dots,N_h$ and $W$ is a horizontal vector field, 
then the following identity holds
in the sense of distributions
$$
W\big(J_\eps\ast \mu_j\big) = J_\eps \ast W\mu_j
= \ccheck W \ccheck  J_\eps \ast \mu_j,
$$
so that, arguing as in \eqref{Feb 2 eq:1}, the $f^\eps$'s are smooth.

We can also write
$$
f^\eps = ( J_\eps\ast \mu_1, \dots,  J_\eps\ast \mu_{N_h}) =: (f_1^\eps, \dots, f_{N_h}^\eps).
$$
Thus we can define a family of smooth forms in $L^1(\he n, E_0^h)$
$$
\omega^\eps =  \sum_j f_j^{\eps} \xi_j
$$
and a family of $h$-currents $T_\eps:=T_{\star\omega^\eps}$.
By \eqref{mass L^1}, 
\begin{equation}\label{jan 27 eq:2}
\mc M(T_\eps)= \|\star\omega^\eps\|_{L^1(\he n, E_0^h)}
= \|\omega^\eps\|_{L^1(\he n, E_0^h)} = |\mu^\eps| < \infty.
\end{equation}
Thus i) is proved.

We notice now that 
%\begin{equation}\label{jan 12 eq:2}
%\Scal{J_\eps\ast T}{\phi} = \Scal{T_{\star\mu^\eps}}{\phi}
%\end{equation}
for all test form $\phi=\sum_j\phi_j\xi_j\in\mc D(\he n,E_0^{h})$, 
if we identify $\phi$ with the vector-valued function
$(\phi_1, \dots,\phi_{N_h})$, we have
%First of all, 
%we notice
%that 
\begin{equation}\label{jan 16 eq:1}
\Scal{T}{J_\eps\ast \phi} =  \int_{\he n}\scal{ f^\eps}{\phi}\, d\mc L^{2n+1}
=: \int_{\he n}\scal{\phi}{d \mu^\eps},
\end{equation}
that we can also written as
\begin{equation}\label{Feb 3 eq:1}
\Scal{T}{J_\eps\ast \phi} =  \int_{\he n}\scal{ \omega^\eps}{\phi}\, d\mc L^{2n+1}.
\end{equation}
Indeed,
first of all, 
we notice
that if $j=1,\dots,N_h$,
$$
\int_{\he n}\big( \int _{\he n} J_\eps(q) \,|\phi_j(q^{-1}p)| \, d\mu_j(p)\big) \, dq< \infty.
$$
Thus, by \eqref{jan 27 eq:1},
\begin{equation*}\begin{split}
\Scal{T}{J_\eps\ast \phi} &
= \sum_j \int_{\he n}(J_\eps\ast\phi_j)(p)\,d\mu_j(p)
\\& =  \sum_j \int_{\he n}\Big(\int_{\he n} J_\eps(q)\phi_j(q^{-1}p) \, dq \Big)\, d\mu_j(p)
\\&
=\sum_j \int_{\he n}\Big(\int_{\he n} J_\eps(p\eta^{-1})\phi_j(\eta) \, d\eta \Big)\, d\mu_j(p)
\\&
=\sum_j \int_{\he n}\Big(\int_{\he n} J_\eps(p\eta^{-1}) \, d\mu_j(p) \Big)\phi_j(\eta)\, d\eta
\\&
=\sum_j \int_{\he n}\Big(\int_{\he n} J_\eps(\eta p^{-1}) \, d\mu_j(p) \Big)\phi_j(\eta)\, d\eta \qquad\mbox{(since $J
=\ccheck J$)}
\\&
= \sum_j \int_{\he n}( J_\eps\ast \mu_j)(\eta)\phi_j(\eta)\, d\eta
=
\int_{\he n}\scal{  f^\eps}{\phi}(\eta)\, d\eta = \int_{\he n}\scal{\phi}{d \mu^\eps}.
\end{split}\end{equation*}

This proves \eqref{jan 16 eq:1}.

Let us prove ii). Take again $\phi\in \mc D(\he n, E_0^h)$: By
\eqref{Feb 3 eq:1}, we have:
\begin{equation*}
\Scal{T_\eps}{\phi} = \Scal{T_{\star\omega^\eps}}{\phi}=
\int_{\he n} \star\omega^\eps\wedge\phi =
\int_{\he n} \scal{\omega^\eps}{\phi}\, d\mc L^{2n+1} =
\Scal{T}{J_\eps\ast\phi} \to \Scal{T}{\phi}
\end{equation*}
 as $\eps\to 0$,
since $J_\eps\ast\phi\to \phi$ in $\mc D(\he n,E_0^h)$.

In particular, $\int_{\he n} \scal{f^\eps}{\phi}\, d\mc L^{2n+1} 
 \to \Scal{T}{\phi} $ for all $\phi\in \mc D(\he n,E_0^h)$.
 Thanks to the density of $\mc D(\he n,E_0^h)$ in $C_{\mathrm{comp}} (\he n,E_0^h)$,
 by \eqref{Feb 2 eq:1}, $\mu^\eps\to\mu$ weak* in the sense of measures.
 
Now, we can prove iii). By \cite{ambrosio_fusco_pallara}, Theorem 1.59
 $$
 \mc M(T) = |\mu |(\he n) \le\liminf_{\eps\to 0}  |\mu^\eps (\he n)|
 =  \liminf_{\eps\to 0}  \mc M(T_\eps).
 $$
 On the other hand, by \eqref{Feb 2 eq:1}
$$
\limsup_{\eps\to 0}  \mc M(T_\eps) \le \mc M(T),
$$
and iii) follows.

Finally, let us prove  that $\star \omega^\eps$ is $d_c$-closed. Take again 
$\phi\in\mc D(\he n,E_0^{h+1})$. Keeping in mind that $\partial_c  T=0$,
by \eqref{jan 16 eq:1} we have:
\begin{equation*}\begin{split}
\int_{\he n}\scal{d_c\star \omega^\eps}{ \phi}\, d\mc L^{2n+1} &=
\int_{\he n}\scal{\star \omega^\eps}{d_c^*  \phi}\, d\mc L^{2n+1} 
=\pm\int_{\he n}\scal{\star \omega^\eps}{\star d_c\star  \phi}\, d\mc L^{2n+1}
\\& =\pm\int_{\he n}\scal{ \omega^\eps}{ d_c\star  \phi}\, d\mc L^{2n+1}
= \pm \Scal{T}{J_\eps\ast ( d_c\star \phi)}
\\&
= \pm \Scal{T}{d_c (J_\eps\ast  (\star \phi))}
= \pm \Scal{\partial_c  T}{J_\eps\ast  (\star \phi)}=0,
\end{split}\end{equation*}
and hence $d_c(\star \omega^\eps)=0$. Thus iv) is proved.

\end{proof}

Thanks to previous result, theorem \ref{poincareglobal}  can reformulated in terms of currents as follows:

\begin{theorem}[Global Poincar\'e and Sobolev inequalities for currents]\label{poincareglobal currents uno}
  Let $h=1,\ldots,2n$. 
 % \begin{itemize}   
 %\item[a)] 
If $h\not=n$ and  $\,T\in \mc D'(\he n,E_0^{h})$ is a current of  finite mass of the form 
 $T=\partial_c  S$ with $S\in \mc D'(\he n, E_0^{h+1})$,
then there exists a form $\phi\in L^{Q/(Q-1)}(\he n,E_0^{2n-h})$, such that
$$
\partial_c  T_{ \phi}=T \qquad\mbox{and}\qquad \|\phi\|_{L^{Q/(Q-1)}(\he n, E_0^{2n-h})}\leq C\,\mc M(T).
$$
%\item [b)] If $\,T\in \mc D'(\he n,E_0^{n})$ is a current of  finite mass of the form 
 %$T=\partial_c  S$ with $S\in \mc D'(\he n, E_0^{n+1})$, then there exists 
 %a form $\phi\in \BLh{1}{Q/(Q-1)}{n}$, such that
%$$
%\partial_c  T_{ \phi}=T \qquad\mbox{and}\qquad \|\phi\|_{\BLh{1}{Q/(Q-1)}{n}}\leq C\,\mc M(T).
%$$
%\end{itemize}

Furthermore, 
 if $T$ is compactly supported, so is $\phi$.
\end{theorem}

\begin{proof} Let us prove a). By Proposition \ref{AG} with $\eps=\eps_k\to 0$
as $k\to \infty$, with the notations
therein, for any $k\in\mathbb N$
there exists a sequence $(\star \omega_k)_{k\in\mathbb N}
\in \mc E(\he n, E_0^{2n+1-h})\cap L^1(\he n, E_0^{2n+1-h})$
satisfying i), ii), iii), iv). In particular, $d_c \star \omega_k=0$.
Therefore, since $2n+1-h \neq n+1$ (by hypothesis $h\neq n$), Theorem \ref{poincareglobal} implies that there exist
 $\phi_k\in L^{Q/(Q-1)}(\he n,E_0^{2n-h})$, such that 
$$
d_c\phi_k=\star\omega_k \qquad\mbox{and}\qquad \|\phi_k\|_{L^{Q/(Q-1)}(\he n, E_0^{2n -h})}
\leq C\,\|\star \omega_k\|_{L^1(\he n, E_0^{2n+1-h})}.
$$
By Proposition \ref{AG}, iii), 
 we know also  that 
$$\|\star\omega_k\|_{L^1(\he n, E_0^{2n+1-h})}
= \|\omega_k\|_{L^1(\he n, E_0^{h})}
\to \mc M(T)$$ as $k\to\infty$.
Therefore the $\phi_k$'s are equibounded in $L^{Q/(Q-1)}(\he n,E_0^{h-1})$ and hence we can
assume that 
$\phi_k\to \tilde\phi$ weakly in $L^{Q/(Q-1)}(\he n,E_0^{2n-h})$.

 Thus
$$
\|\tilde\phi\|_{L^{Q/(Q-1)}(\he n,E_0^{2n-h})} \le C\,\mc M(T).
$$
By definition, with $\phi:=(-1)^h\tilde\phi$ we can associate the current $T_\phi\in \mc D'(\he n, E_0^{h+1})$.
Eventually, we are left with the proof of $\partial_c  T_{\phi} = T$. Let $\psi\in\mc D(\he n, E_0^h)$ be a test form. We have:
\begin{equation*}\begin{split}
\Scal{\partial_c  T_\phi}{\psi} &=\Scal{T_\phi}{d_c\psi} =
\int_{\he n} \phi\wedge d_c\psi = (-1)^h\lim_{k\to\infty} \int_{\he n} \phi_k\wedge d_c\psi 
\\&
= (-1)^{2h}  \lim_{k\to\infty} \int_{\he n} d_c\phi_k\wedge \psi
=  \lim_{k\to\infty} \int_{\he n} \star\omega_k\wedge \psi
\\&
=  \lim_{k\to\infty} \Scal{T_{\star\omega_k}}{ \psi}
= \Scal{T}{ \psi},
\end{split}\end{equation*}
by \ref{AG}, ii). This completes the proof of the statement.

\end{proof}

\section{Proof of Theorem \ref{main global}--i): continuous primitives of forms in $E_0^h$,  $h\neq n+1$} 
\label{proof global}

The proof of Theorem \ref{main global} will follows by duality. Having in mind Section \ref{BB duality}, we start with a few definitions.
\begin{definition}\label{spaces}
Denote again by $\mc C_0$ the Banach space of continuous functions vanishing at infinity
with the $L^\infty$-norm. We set
$$
E:= \mc C_0(\he n, E_0^{h-1}),
$$
By Remark \ref{riesz}, the dual space $E^*$ can be identified with the
set of $(h-1)$-currents with finite mass. 
If $2\le  h \le { 2n+1}$, $h\neq n+1$, we set
$$
\mc D(A):= \{\psi\in E, \, d_c\psi\in L^Q(\he n, E_0^{h})\} \subset E,
$$
and $A:\mc D(A)\to F$, where
$$
F=L^Q(\he n, E_0^{h}) \qquad \mbox{and}\qquad A\psi:= d_c\psi.
$$

\end{definition}

\begin{remark}

Notice that $\mc D(A)$ is dense since contains $\mc D(\he n, E_0^{h-1})$
and $A$ is closed since is a differential operator.

In addition $F^*$ can be identified with $L^{Q/(Q-1)}(\he n, E_0^{2n+1-h})$
through the identity: if $\beta\in L^{Q/(Q-1)}(\he n, E_0^{h})$, then we put
$$
f(\alpha)=\int_{\he n} \beta\wedge\alpha
$$
for $\alpha\in L^Q(\he n, E_0^{h})$.
\end{remark}

\begin{lemma} \label{ambrogio: 2} Suppose $h\neq n+1$. If $\psi\in \mc D(A)$, then there exists
a sequence $(\psi_k)_{k\in\mathbb N}$ in $\mc D (\he n, E_0^{h-1})$ such that
\begin{equation}\label{ambrogio: 1}
\psi_k \to \psi \quad \mbox{in $E$}\qquad\mbox{and}\qquad d_c\psi_k\to d_c\psi \quad \mbox{in $F$.}
\end{equation}

\end{lemma}

\begin{proof} Let show preliminarily that there exist a sequence 
$$
(\psi_k)_{k\in\mathbb N}\quad
\mbox{in}\quad \mc C_{\mathrm{comp}} (\he n, E_0^{h-1})
$$
 satisfying \eqref{ambrogio: 1}. Let us take a sequence of smooth
cut-off functions $(\phi_k)_{k\in\mathbb N}$, $0\le \phi_k\le 1$, $\supp \phi_k \subset  B(e,2k)$,
$\psi \equiv 1$ in $B(e,k)$, $|W\phi_k|\le C/k$, where $C$ is independent of $k$.
Set $\psi_k:= \phi_k \psi$. Obviously, $\psi_k \to \psi \quad \mbox{in $E$}$.
By Leibniz' formula (see Lemma \ref{leibniz}), if $h\neq n$, 
\begin{equation*}\begin{split}
\|d_c \psi_k& - d_c\psi\|_F^Q 
\lessapprox \|(\phi_k-1) d_c \psi \|_F^Q
+ \int_{\he n}  |W\phi_k|^Q\, |\psi |^Q\, dp
\\&
\lessapprox \|(\phi_k-1) d_c \psi \|_F^Q
+ (\frac1k)^Q
\int_{k\le \|p\| \le 2k}   |\psi |^Q\, dp 
\\&
\le \|(\phi_k-1) d_c \psi \|_F^Q
+(\frac1k)^Q |B(e,2k)| \sup_{k\le \|p\| }|\psi| \to 0
\end{split}\end{equation*}
as $k\to\infty$, since $\psi\in E$ and $d_c\psi\in F$ . 
%On the other hand, if $h= n$, again by \cite{BFP3}, Lemma 4.1,
%\begin{equation*}\begin{split}
%\|d_c \psi_k& - d_c\psi)\|_F^Q 
%\lessapprox \|(\phi_k-1) d_c \psi \|_F^Q
%\\&
%+ \int_{\he n}  |W \phi_k|^Q\, |W\psi |^Q\, dp
%+ \int_{\he n}  |W^2 \phi_k|^Q\, |\psi |^Q\, dp
%\end{split}\end{equation*}
%As above, 
%$$
%\|(\phi_k-1) d_c \psi \|_F^Q + \int_{\he n}  |W^2 \phi_k|^Q\, |\psi |^Q\, dp
%\to 0 
%$$
%as $k\to\infty$.
%
%Since $|W\phi_k|\le C$ for all $k\in\mathbb N$ and $ |W \phi_k|^Q\to 0$ as $k\to\infty$
%a.e. in $\he n$, if we prove that $|W\psi |\in L^q(\he n, E_0^h)$, then 
%also
%$$
% \int_{\he n}  |W \phi_k|^Q\, |W\psi |^Q\, dp \to 0
% $$
% as $k\to\infty$. 
%
%
%Thus, we are reduce to prove the lemma when $\psi\in \mc D(A) \cap 
%\mc C_c (\he n, E_0^h)$.
%
%Denote by $\omega_k$ the group mollifier supported in $B(e,\frac1k)$. We set
%$\psi_k:= \omega_k\ast \psi$ (group convolution). Obviously, $\psi_k\in \mc 
%D(\he n, E_0^h)$ and $\psi_k\to \psi$ in $E$. Moreover, since $d_c$ contains
%only horizontal derivatives, 
%$$
%d_c\psi_k= \omega_k\ast d_c\psi \to d_c\psi\qquad\mbox{in $F$.}
%$$

\end{proof}

\begin{proposition} We have:
$\mc D(A^*)$ is dense in $F^*$
and 
\begin{equation}\label{domain A* ter}
\mc D(A^*) =\{ \beta\in F^*, \, 
\mc M(\partial_c  T_\beta)<\infty \}.
\end{equation}
In addition, if $\beta\in\mc D(A^*)$,
\begin{equation}\label{A* ter}
A^*\beta = \partial_c  T_\beta.
\end{equation}

\end{proposition}

\begin{proof}

Since $F$ is reflexive,
then $\mc D(A^*)$ is dense in $F^*$ (see \cite{brezis}, Remark 15 of
Section 2.6). 
To prove \eqref{domain A* ter} we have first to show that, if $\beta\in F^*$ and there exists $c_\beta$ such that
\begin{equation}\label{apr 22 eq:1} 
| \int_{\he n}\beta\wedge d_c\psi  | \le c_\beta \|\psi\|_E \qquad\mbox{for all $\psi\in \mc D(A)$},
\end{equation}
then $\mc M(\partial_c  T_\beta)<\infty$. The assertion is straightforward since, if $\psi\in\mc D(\he n,E_0^{h-1})\subset \mc D(A)$, $\| \psi\|_E\le 1 $ then
\begin{equation*}\begin{split}
 |\Scal{\partial_c  T_\beta}{\psi} | =  |\Scal{T_\beta}{d_c\psi} |   =  | \int_{\he n}\beta\wedge d_c\psi  |  \le c_\beta.
\end{split}\end{equation*}
This shows that $\mc D(A^*)\subset \{ \beta\in F^*, \, \mc M(\partial_c  T_\beta)<\infty \}$.

Let us prove now the reverse inclusion. 
Suppose  $\beta\in F^*$ is such that $ \mc M(\partial_c  T_\beta)<\infty$. 
If $\psi\in \mc D(A)$, by Lemma \ref{ambrogio: 2}
there exists
a sequence $(\psi_k)_{k\in\mathbb N}$ in $\mc D (\he n, E_0^{h-1})$ such that
\begin{equation*}
\psi_k \to \psi \quad \mbox{in $E$}\qquad\mbox{and}\qquad d_c\psi_k\to d_c\psi \quad \mbox{in $F$.}
\end{equation*} Hence, \eqref{apr 22 eq:1} holds. Indeed,
\begin{equation*}\begin{split}
 | \int_{\he n}&\beta\wedge d_c\psi  | =
  \lim_k | \int_{\he n}\beta\wedge d_c\psi_k  |
  =  \lim_k |\Scal{\partial_c  T_\beta}{\psi_k}|
  \\&
 \le \mc M(\partial_c  T_\beta) \limsup_k \|\psi_k\|_E
 = \mc M(\partial_c  T_\beta)  \|\psi\|_E.
\end{split}\end{equation*}
This completes the proof of \eqref{domain A* ter}. 
%Since $F$ is reflexive,
%then $\mc D(A^*)$ is dense in $F^*$ (see \cite{brezis}, Remark 15 of
%Section 2.6.

Again following \cite{brezis}, Section 2.6, if $\beta\in\mc D(A^*)$, then $A^*\beta$ is uniquely determined by the identity
$$
\scal{A^*(\beta)}{\psi}_{E^*,E} = \scal{\beta}{A\psi}_{F^*,F}\qquad\mbox{for $\psi\in\mc D(A)$.} 
$$
Taking $\psi\in \mc D(\he n, E_0^{h-1})$, then \eqref{A* ter} follows.

\end{proof}

\begin{proof}[Proof of Theorem \ref{main global}-i)]
With the notations of Definition \ref{spaces},
let us show preliminarily that
\begin{equation}\label{feb 15 eq:1}
A^*(\mc D(A^*))\qquad\mbox{is closed in $E^*$.}
\end{equation}
To this end, let $(T_n)_{n\in\mathbb N}$ a sequence 
in $A^*(\mc D(A^*))$ that converges to a $h-1$-current 
$T\in E^*$ (i.e. in the mass norm). Hence, $\mc M(T_n)
= \|T_n\|_{E^*}\le C_1$ for all $n\in\mathbb N$.
Moreover, in particular
 $T_n\to T$ weakly in  $\mc D'$, i.e.
 $$
 \Scal{T}{\sigma}=\lim_{n\to\infty} \Scal{T_n}{\sigma}
 $$
 for all $\sigma\in \mc D(\he n,E_0^{h-1})$.
By \eqref{domain A* ter} and \eqref{A* ter} there exists a corresponding
sequence $(S_n)_{n\in\mathbb N}$ in $\mc D'(\he n, E_0^h)$ such that
\begin{equation}\label{feb 15 eq:2}
T_n = \partial_c  S_n, \quad\mbox{with $S_n=T_{\beta_n}$, \;$\beta_n\in 
 L^{Q/(Q-1)}(\he n, E_0^{2n+1-h})$} 
\end{equation}
for any $n\in\mathbb N$. Since the $(h-1)$-currents $\partial_c  S_n$'s 
satisfy
$\mc M(\partial_c  S_n) = \mc M(T_n) <\infty$, by Theorem \ref{poincareglobal currents uno}
(that we can apply since 
$1\le h-1\le 2n$ and $h-1\neq n$), again for any $n\in\mathbb N$ there exists $\phi_n\in L^{Q/(Q-1)}(\he n, E_0^{2n+1-h})$
such that $\partial_c  T_{\phi_n} = \partial_c  S_n=\partial_c  T_{\beta_n}$, and
$$
\|\phi_n\|_{L^{Q/(Q-1)}(\he n, E_0^{2n+1-h})} \le C_2\mc M(\partial_c  S_n) 
= C\mc M(T_n) \le C_1C_2.
$$
Since $Q/(Q-1)>1$ we cas assume that 
$$\phi_n\to \phi\qquad\mbox{weakly in
$L^{Q/(Q-1)}(\he n, E_0^{2n+1-h}) $.}
$$
Thus
$T_{\phi_n} \to T_{\phi}$ weakly* in $\mc D'$ and $T_n = \partial_c  S_n 
=\partial_c  T_{\phi_n} \to \partial_c  T_{\phi}$ weakly* in $\mc D'$; therefore,
since also $T_n\to T$ weakly*, it follows that $T= \partial_c  T_{\phi}$.
To prove that $T\in A^*(\mc D(A^*))$, by \eqref{domain A* ter} we have only to 
show that $\mc M(\partial_c  T_\phi)<\infty$. Because of the lower
semicontinuity of the mass with respect to the weak* convergence,
we have
$$
\mc M(\partial_c  T_\phi)\le\liminf_{n\to \infty}\mc M(\partial_c  S_n)
= \liminf_{n\to \infty}\mc M(T_n) = \mc M(T).
$$
Thus \eqref{feb 15 eq:1} is proved.

By Theorem \ref{predual}, 1), \eqref{feb 15 eq:1} implies that
\begin{equation}\label{feb 19 eq:1}
A(\mc D(A)) = (\ker A^*)^\perp.
\end{equation}
Moreover, by Lemma \ref{predualpoin}, \eqref{feb 15 eq:1} implies that
there exists $C>0$ such that  for all  $f\in A(D(A))$ there exists $e\in D(A)$, 
satisfying 
\begin{equation}\label{mar 3 eq:1}
Ae=f\quad\mbox{and}\quad \|e\|_{E}\leq C\,\|f\|_{F}.
\end{equation}

We are left to show that $\{\alpha\in L^Q(\he n, E_0^{h}), \, d_c \alpha = 0\}\subset A(D(A))$. This will be done by showing that
%On the other hand, 
\begin{equation}\label{feb 19 eq:2}
\{\alpha\in L^Q(\he n, E_0^{h}), \, d_c \alpha = 0\}
\subset (\ker A^*)^\perp.
\end{equation}
Indeed, let $\alpha\in L^Q(\he n, E_0^{h})$ be a closed
form, and take $\beta\in L^{Q/(Q-1)}(\he n, E_0^{2n+1-h})\in \ker A^*$,
i.e. $\beta\in L^{Q/(Q-1)}(\he n, E_0^{2n+1-h})$ is a closed form, by Remark \ref{mar 2 rem:1}. Then,
by Proposition \ref{parabolic}, 
$$
\int_{\he n} \alpha\wedge\beta = 0,
$$
i.e. \eqref{feb 19 eq:2} holds.

Thus, combining  \eqref{feb 19 eq:2}  and \eqref{feb 19 eq:1}, 
\begin{equation*}
\{\alpha\in L^Q(\he n, E_0^{h}), \, d_c \alpha = 0\} \subset
A(\mc D(A)),
\end{equation*}
and hence, by \eqref{mar 3 eq:1}, we have proved the theorem.

\end{proof}

\section {The main global result in degree $n+1$: Poincar\'e inequalities with Beppo Levi-Sobolev spaces}\label{caso 2}

The proof of our main result when $h=n+1$ is technically more delicate since in this case the equation $d_c\phi=\omega$  involves a differential operator of order two.  In this case, we  prove a more general result involving the Beppo Levi-Sobolev space $\BL{1}{Q/Q-1}$ that implies, as a corollary, statement ii) of Theorem \ref{main global}. Let us start by mentioning some general facts regarding Beppo Levi-Sobolev spaces $\BL{1}{p}$.

\subsection {Beppo Levi-Sobolev spaces $\BL{1}{p}$}

\begin{definition} If $1<p<Q$, we denote by $\BL{1}{p}$ the homogeneous Sobolev space
(called also  Beppo Levi space) defined as the completion of $\mc D(\he n)$ with respect to the norm
$$
\| u\|_{\BL{1}{p}} := \sum_{j=1}^{2n}\|W_j u\|_{L^p(\he n)}.
$$
\end{definition}

\begin{remark}\label{tau}
Since $\BL{1}{p}$ is reflexive, it can be identified with its bidual via the canonical isomorphism
$\tau(u)(f) = f(u)$ for all $\BL{1}{p}$ and $f\in (\BL{1}{p})^*$. 
\end{remark}

\begin{proposition} If $1<p<Q$, then
$$
(\BL{1}{p})^* = \{T\in \mc D'(\he n)\, ; \, T= \sum_j W_j f_j \, , \, f_j\in L^{p'}(\he n)\}.
$$
If $ {F} = (f_1,\dots, f_{2n})$ is a horizontal vector field, then we set
$$
{\mathrm{div}}_{\mathbb{H}}\, {F} := \sum_j W_j f_j.
$$

\end{proposition}

By Sobolev inequality with sharp exponents in Carnot groups
(see e.g., \cite{Lu94})
 it is easy to prove the following theorem.

\begin{theorem}\label{embedding} If $1<p<Q$, we have:
$
u\in \BL{1}{p}$ if and only if 
$$
u\in L^{pQ/(Q-p)}(\he n)\quad\mbox{and}\quad W_ju\in L^p(\he n),\, j=1,\cdots,2n,
$$
and the two norms are equivalent, i.e.,
$$
\|u\|_{\BL{1}{p}}\approx \|u\|_{L^{pQ/(Q-p)}(\he n)} + \sum_j \|W_j u\|_{L^p(\he n)}. 
$$

In particular,
the embedding $\BL{1}{p}\subset L^{pQ/(Q-p)}(\he n)$ is continuous.

\end{theorem} 

 For sake of completeness, let us provide a full proof
of the above statement. Let us state preliminarily the following
(trivial) result. 
\begin{lemma}\label{trivial} Let $X$ be a Banach space, and $X_0\subset X$ a dense
subspace. Let $Y_0$ be a normed space and $Y$ its completion.
If $T_0:X_0\to Y_0$ is a linear isomorphism
such that
\begin{equation}\label{dec26 eq:1}
\|u\|_{X_0} \approx \| T_0u\|_{Y_0}\qquad \mbox{for $u\in X_0$,}
\end{equation}
then $T_0$ can be continued as a linear isomorphism $T$ between $X$ and $Y$.
\end{lemma}

\begin{proof} Take $u\in X$; then there exists a sequence $(u_k)_{k\in N}$
in $X_0$ converging to $u$ in $X$.  Then, by \eqref{dec26 eq:1}, 
$(Tu_k)_{k\in N}$ is a Cauchy sequence in $Y_0$ converging to an
element of $Y$. We set $Tu:=\lim_k T_0u_k$. Trivially, the continuation
$T$ is a well defined  bounded linear operator between $X$ and $Y$.
Clearly, $T$ is injective. On the other hand, take $v\in Y$. By definition,
there exists a sequence $(v_k)_{k\in N}$ in $Y_0$ converging to $v$
in $Y$. Again by \eqref{dec26 eq:1}, $(T_0^{-1}v_k)_{k\in N}$ has a
limit $u\in X$. By continuity, $Tu=\lim_k TT_0^{-1}v_k = \lim_k T_0T_0^{-1}v_k
=v$. Therefore, $T$ is also surjective and hence is a linear isomorphism.

\end{proof}

\begin{proof}[Proof of Theorem \ref{embedding}]
With the notation of previous Lemma \ref{trivial}, set
$$
X:= L^{pQ/(Q-p)}(\he n)\cap \{ u\in \mc D'(\he n)\,,  |Wu| \in L^p(\he n)\},
$$
where we set $|Wu|:=\sum_{j=1}^{2n}|W_j u|$, and (with a slight abuse of notation)
endowed with the norm
$$
\|u\|_X = \|u\|_{L^{pQ/(Q-p)}(\he n)} + \|  Wu\|_{L^p(\he n)}.
$$
A standard argument shows that $X$ is a Banach space. Let
us prove that $X_0:=\mc D'(\he n)$ is dense in $X$. We notice
first that $\mc E(\he n)\cap X$ is dense in $X$.
Indeed, if $\eps>0$, let $J_\eps = \ccheck J_\eps$ denote again a (group) Friedrichs' mollifier.
If $u\in X$, then $J_\eps\ast u\in\mc E(\he n)$, and
$$
J_\eps\ast u \to u\qquad\mbox{ in $L^{pQ/(Q-p)}(\he n)$,}
$$
and 
$$
W(J_\eps\ast u) = J_\eps\ast Wu \to Wu \qquad\mbox{ in $L^{p}(\he n)^{2n}$}.
$$
as $\eps\to 0$.
Therefore, from now on we can assume $u\in \mc E(\he n)\cap X$.
Let us fix a sequence of cut-off functions $\{\chi_k\}_{i\in\N}\subset\mc D(\he n)$  such that for any $k\in\N$   ${\rm supp} (\chi_k)\subset {B(e,2k)}$, 
$\chi_k\equiv 1$ in  ${B(e,k)}$, and $k\,|W\chi_k|\le C$ for all $k\in\mathbb N$.
Set $u_k:= \chi_k u$. Obviously,
$$
u_k \to u\qquad\mbox{ in $L^{pQ/(Q-p)}(\he n)$ as $k\to\infty$.}
$$
Moreover, since
$$
|Wu_k|\le C\frac1k |u| + \chi_k |Wu|,
$$
then
$$
|Wu_k| \to |Wu| \qquad\mbox{ in $L^{p}(\he n)$ as $k\to\infty$.}
$$
Again with the notations of previous Lemma \ref{trivial}, we choose
$Y_0:= \mc D(\he n)$ endowed with the norm $\|u\|_{Y_0} :=
\| |Wu| \|_{L^p(\he n)}$. The inequality \eqref{dec26 eq:1} is nothing
but the sharp Poincar\'e inequality of \cite{Lu94}, and the statement
is proved.

\end{proof}

\begin{definition} If $1\le h\le 2n+1$, then a form $\alpha$ belongs to $\BLh{1}{p}{h}$ if and
only if all its components with respect to a fixed left invariant basis 
\begin{equation*}
\Xi_0^h=\{\xi_1^h,\dots \xi^h_{N_h}\}
\end{equation*}
of $ E_0^h$ belong
to $\BL{1}{p}$.

\end{definition}

\begin{proposition}\label{BL dual}
The dual space
$
( \BLh{1}{p}{h} )^*
$
can be identified with a family of currents $T\in\mc D'(\he n, E_0^{h})$ such that,
with the notation of \eqref{Aug11 eq:4}
\begin{equation*}
T=\sum_j T_j (\xi_j^h)^*,
\end{equation*}
with $T_j\in\mc D'(\he n)$, $j=1,\dots, N_h$ of the form
$$
T_j = \mathrm{div}_{\he{}}\, {F_j}.
$$
with ${F_j} \in ( L^{p'})^{2n}$, $j=1,\dots, N_h$.

More precisely, by the density of $\mc D(\he n,E_0^h)$ in $ \BLh{1}{p}{h}$,
an element of $( \BLh{1}{p}{h} )^*$ is fully identified by its restriction
to $\mc D(\he n,E_0^h)$. In particular, if $T\in ( \BLh{1}{p}{h} )^*\subset \mc D'
(\he n, E_0^h)$ and, if for sake of simplicity we write
$\mc W:  =\BLh{1}{p}{h}$, we have
\begin{equation}\label{nov2 eq:2}
\Scal{T}{\phi}_{\mc W^*,\mc W} = \Scal{T}{\phi}_{\mc D',\mc D}.
\end{equation}

\end{proposition}

\begin{proof} 
Let $T\in ( \BLh{1}{p}{h} )^*$ be given,
and consider a sequence $(\phi_k)_{k\in\mathbb N}$ of test forms,
$\phi_k\to 0$ in $\mc D(\he n,E_0^h)$. Since there exists a compact set $K$ 
such that $\supp \phi_k\subset K$ for all $k\in\mathbb N$ and $\phi_k\to 0$
uniformly with all its derivatives, then
$$
\| \phi_k\|_{\BLh{1}{p}{h}} \to 0 \qquad\mbox{as $k\to\infty$,}
$$
so that $\Scal{T}{\phi_k}_{\mc W^*,\mc W}\to 0$ as $k\to\infty$. This proves that $T$ is a $h$-current.

Take now $\phi\in \mc D(\he n)$. If $1\le k\le N_h$, define
\begin{equation*}\begin{split}
\Scal{T_k}{\phi}_{\mc D',\mc D}= \Scal{T}{\phi \xi_k^h}_{\mc W^*,\mc W}   \le C\| \phi \xi_k^h\|_{\BLh{1}{p}{h}}
= C\| \phi \|_{\BLh{1}{p}{h}},
\end{split}\end{equation*}
so that $T_k\in (\BLh{1}{p}{h})^*$. By Proposition \ref{BL dual} there exists a horizontal 
vector field in ${F_k}\in ( L^{p'})^{2n}$  such that
$$
T_k (\xi_k^h)^* = (\mathrm{div}_{\he{}}\, {F_k})  (\xi_k^h)^*
$$
and
\begin{equation}\label{}
T = \sum_j   (\mathrm{div}_{\he{}}\, {F_j})  (\xi_j^h)^*.
\end{equation}
\end{proof}

In the sequel, we will often use the identification given in \eqref{BL dual}, omitting the subscript. 
%\begin{lemma}\label{oct19 eq:1}  If $\omega\in L^{Q/2}(\he n, E_0^{n+1})$, then $T_\omega\in ( \BLh{1}{Q/(Q-1)}{n} )^*$
%and
%$$
%\| T_\omega\|_{( \BLh{1}{Q/(Q-1)}{n} )^*} \le C \|\omega\|_{L^{Q/2}(\he n, E_0^{n+1})}.
%$$
%
%\end{lemma}
%
%\begin{proof} Indeed, if $\phi\in\mc D(\he n, E_0^n)$, then
%\begin{equation*}\begin{split}
%\Scal{T_\omega}{\phi} & = \int_{\he n} \omega\wedge\phi \le \|\omega\|_{L^{Q/2}(\he n, E_0^{n+1})}
%\|\phi\|_{L^{Q/(Q-2)}(\he n,  E_0^n)}
%\\&
%\le C  \|\omega\|_{L^{Q/2}(\he n, E_0^{n+1})}
%\|\phi\|_{\BLh{1}{Q/(Q-1)}{n}},
%\end{split}\end{equation*}
%by Lemma \ref{oct19 eq:1}.
%\end{proof}

The next two assertions follow straightforwardly from Theorem \ref{embedding} .
\begin{proposition}\label{weak and weak*} If $\phi_k\to \phi$ weakly in $\BLh{1}{p}{h}$, then 
$$
T_{\phi_k}\to  T_\phi \qquad\mbox{weak* in $\mc D'(\he n, E_0^{2n+1-h})$.}
$$

\end{proposition}

\begin{lemma}\label{oct19 eq:1}  If $\omega\in L^{Q/2}(\he n, E_0^{n+1})$, then $T_\omega\in ( \BLh{1}{Q/(Q-1)}{n} )^*$
and
$$
\| T_\omega\|_{( \BLh{1}{Q/(Q-1)}{n} )^*} \le C \|\omega\|_{L^{Q/2}(\he n, E_0^{n+1})}.
$$

\end{lemma}

\begin{proof} Let $\phi\in\mc D(\he n, E_0^n)$, then
\begin{equation*}\begin{split}
\Scal{T_\omega}{\phi} & = \int_{\he n} \omega\wedge\phi \le \|\omega\|_{L^{Q/2}(\he n, E_0^{n+1})}
\|\phi\|_{L^{Q/(Q-2)}(\he n,  E_0^n)}
\\&
\le C  \|\omega\|_{L^{Q/2}(\he n, E_0^{n+1})}
\|\phi\|_{\BLh{1}{Q/(Q-1)}{n}},
\end{split}\end{equation*}
by Theorem \ref{embedding}
%Lemma \ref{oct19 eq:1}.
\end{proof}

\subsection{Poincar\'e inequalities with the Beppo Levi-Sobolev space: the case of primitive in $E_0^n$}
One of the main results of this paper is contained in the following theorem.

\begin{theorem}\label{main global bis} 
Let $\Omega$ be a $n$-current identified with an element
of 
$$
( \BLh{1}{Q/(Q-1)}{n} )^*
$$
(i.e. of the form described in Proposition \ref{BL dual}). 
 
 If $\partial_c \Omega=0$, then there exists $\phi\in \mc C_0(\he n, E_0^{n})$
such that $\partial_c  T_\phi = \Omega$ and
$$
\|\phi\|_{\mc C_0(\he n, E_0^{n})} \le C \|\Omega\|_{( \BLh{1}{Q/(Q-1)}{n} )^*}.
$$

\end{theorem}

As a corollary we get:
\begin{proof}[Proof of Theorem \ref{main global}-ii)]
By Lemma \ref{oct19 eq:1}, the statement  
 follows from Theorem \ref{main global bis}.
\end{proof}

The first step in order to prove Theorem \ref{main global bis}, is a result that contains an improvement of Theorem 1.1 (2) in \cite{BFP3} in the
case $h=n+1$.

\begin{theorem}[Global Poincar\'e and Sobolev inequalities in degree $h= n+1$]\label{poincareglobal bis}
 For every $d_c$-exact form $\omega\in L^1(\he n, E_0^{n+1})$,
  there exists an $n$-form 
  $$\phi\in \BLh{1}{Q/(Q-1)}{n}
  $$
   such that 
$$
d_c\phi=\omega \qquad\mbox{and}\qquad \|\phi\|_{ \BLh{1}{Q/(Q-1)}{n}}\leq C\,\|\omega\|_{L^1(\he n, E_0^{n+1})}.
$$
Furthermore, if $\omega$ is compactly supported, so is $\phi$.
\end{theorem}

The proof of Theorem \ref{poincareglobal bis} follows the same line
as the proof of Theorem 1.1 (2) in \cite{BFP3}, after proving
the following Gagliardo-Nirenberg type inequality, which
generalizes Theorem 1.6 in \cite{BFP1}
(see also Theorem 5.1 in  \cite{BFP3}).

\begin{theorem}
\label{bfp1}
Let $u$ be a smooth compactly supported Rumin $n$-form on $\he{n}$. Assume that $d_c^*u=0$. Then
\begin{equation}
\label{GNBFP1}
\|u\|_{\BLh{1}{Q/(Q-1)}{n}} \le C \|d_c u\|_{L^{1}(\he n,E_0^{n+1})}.
\end{equation}
\end{theorem}

\begin{proof} Remember, by Theorem \ref{embedding},
$$
\|u\|_{\BLh{1}{Q/(Q-1)}{n}} \approx \|u\|_{L^{Q/(Q-2)}(\he n, E_0^n)}
+ \| Wu \|_{L^{Q/(Q-1)}(\he n, E_0^n)}
$$
In addition, by Theorem 1.6 in \cite{BFP1}, we already know that
$$
\|u\|_{L^{Q/(Q-2)}(\he n, E_0^n)}
\le C \|d_c u\|_{L^{1}(\he n,E_0^{n+1})};
$$
thus we have but to show that
\begin{equation}\label{nov13 eq:1}
\| Wu \|_{L^{Q/(Q-1)}(\he n, E_0^n)}
\le C \|d_c u\|_{L^{1}(\he n,E_0^{n+1})}.
\end{equation}
Arguing as in \cite{BFP1}, if $\psi\in \mc D(\he n, E_0^n)$
\begin{equation}\label{representation 1}
\begin{split}
&\scal{Wu}{\psi}_{L^2(\he n, E_0^n)} =
-\scal{u}{W\psi}_{L^2(\he n, E_0^n)}
\\& \hphantom{xxx}
=
- \scal{ u}{\Delta_{\he{},n}\Delta_{\he{},n}^{-1} W\psi}_{L^2(\he n, E_0^n)}
 \\& \hphantom{xxx}
 = -\scal{ u}{(\delta_cd_c + (d_c\delta_c)^2)\Delta_{\he{},n}^{-1}W\psi}_{L^2(\he n, E_0^n)}
\\& \hphantom{xxx}
 = -\scal{ u}{\delta_cd_c \Delta_{\he{},n}^{-1}W\psi}_{L^2(\he n, E_0^n)}\\&\hskip1.1em
 -\scal{ u}{ (d_c\delta_c)^2)\Delta_{\he{},n}^{-1}W\psi}_{L^2(\he n, E_0^n)}
\\&\hphantom{xxx}
 =:I_1+I_2.
 \end{split}
\end{equation}
Consider first the term $I_2$. Since $u$ is coclosed, $I_2=0$:
$$
\scal{ u}{ (d_c\delta_c)^2)\Delta_{\he{},n}^{-1}W\psi}_{L^2(\he n, E_0^n)}
= \scal{d_c^* u}{ d_c^*d_cd_c^*)\Delta_{\he{},n}^{-1}W\psi}_{L^2(\he n, E_0^{n-1})} = 0.
$$
Consider now the term $I_1$:
$$
 \scal{ u}{d_c^*d_c \Delta_{\he{},n}^{-1}W\psi}_{L^2(\he n, E_0^n)}=  \scal{d_c u}{d_c \Delta_{\he{},n}^{-1}W\psi}_{L^2(\he n, E_0^{n+1})}.
$$
By Theorem \ref{global solution} and Proposition \ref{kernel}, $d_c \mc KW$ is a kernel of type 1. On the other hand,
as proved in \cite{BFP1}, Theorem 5.1, the components with respect to a given
basis of the closed forms $d_cu$ are linear combinations of the components of a horizontal
vector field with vanishing horizontal divergence. Thus we can apply
Theorem \ref{chanillo_van} and we get eventually
\begin{equation*}\begin{split}
|\scal{Wu}{\psi}_{L^2(\he n, E_0^n)}|& \le C\|d_cu\|_{L^1(\he n, E_0^{n+1})} \cdot
\|\nabla_{\he{}}d_c \Delta_{\he{},n}^{-1}W\psi\|_{L^Q(\he n, E_0^{n+1})}
\\&
\le C\|d_cu\|_{L^1(\he n, E_0^{n+1})} \cdot
\|\psi\|_{L^Q(\he n, E_0^{n})},
\end{split}\end{equation*}
by Theorem \ref{hls folland}-iii). Thus \eqref{nov13 eq:1} follows.

\end{proof}

\begin{proof}[Proof of Theorem \ref{poincareglobal bis}] As in \cite{BFP3}, we can take
$\phi:= d^*_c \Delta_{\he{},n+1}^{-1} \omega$. In particular, Theorems 5.2-iii)
in \cite{BFP3} shows that 
$$
\|\phi\|_{L^{Q/(Q-2)}(\he n, E_0^n)}\le 
C \|\omega\|_{L^{1}(\he n, E_0^{n+1})}.
$$
Repeating verbatim the proof of Theorems 5.2-i)
in \cite{BFP3}, combined now with Theorem \ref{bfp1}, we arrive to show that $W\phi$ satisfies
$$
\|W\phi\|_{L^{Q/(Q-1)}(\he n, E_0^n)}\le 
C \|\omega\|_{L^{1}(\he n, E_0^{n+1})}.
$$
To this end, we 
replace $u$ in \eqref{GNBFP1}
by a suitable compactly supported smooth approximation of $\phi$ (see formula
(33) in \cite{BFP3}). This completes the proof of Theorem \ref{poincareglobal bis}.

\end{proof}

\begin{remark} It is important to stress that in the statement of Theorem 5.2 in \cite{BFP3}
it is required that $\omega$ has vanishing average. However,  
Pansu \& Tripaldi in \cite{PT} proved that such an assumption can be removed.

\end{remark}

Theorem \ref{poincareglobal bis} can reformulated in terms of currents as follows:

\begin{theorem}[Global Poincar\'e and Sobolev inequalities for currents]\label{poincareglobal currents bis}
  Let $h=1,\ldots,2n$. 
 % \begin{itemize}   
 %\item[a)] If $h\not=n$ and  $\,T\in \mc D'(\he n,E_0^{h})$ is a current of  finite mass of the form 
 %$T=\partial_c  S$ with $S\in \mc D'(\he n, E_0^{h+1})$,
%then there exists a form $\phi\in L^{Q/(Q-1)}(\he n,E_0^{2n-h})$, such that
%$$
%\partial_c  T_{ \phi}=T \qquad\mbox{and}\qquad \|\phi\|_{L^{Q/(Q-1)}(\he n, E_0^{2n-h})}\leq C\,\mc M(T).
%$$
%\item [b)] 
If $\,T\in \mc D'(\he n,E_0^{n})$ is a current of  finite mass of the form 
 $T=\partial_c  S$ with $S\in \mc D'(\he n, E_0^{n+1})$, then there exists 
 a form $\phi\in \BLh{1}{Q/(Q-1)}{n}$, such that
$$
\partial_c  T_{ \phi}=T \qquad\mbox{and}\qquad \|\phi\|_{\BLh{1}{Q/(Q-1)}{n}}\leq C\,\mc M(T).
$$
%\end{itemize}

Furthermore, if $T$ is compactly supported, so is $\phi$.
\end{theorem}

The proof of this statement is omitted since it  can be carried out analogously to the proof of Theorem \ref{poincareglobal currents uno}
combining Proposition \ref{AG} and Theorem \ref{poincareglobal bis}.

\section{Some more approximation results}

\label{more}

 \begin{lemma}\label{approx bis}  If $\beta\in \BLh{1}{Q/(Q-1)}{n}$ and
$\mc M(\partial_c  T_\beta)<\infty$, then there exists a sequence $(\beta_k)_{k\in\mathbb N}$
in $\mc D(\he n, E_0^h)$ such that
\begin{itemize}
\item[i)] $\beta_k\to \beta$ as $k\to\infty$  in $ \BLh{1}{Q/(Q-1)}{n}$;
\item[ii)] $\mc M(\partial_c  T_{\beta_k})\to \mc M(\partial_c  T_{\beta})$ as $k\to\infty$.
\end{itemize}

\end{lemma}

\begin{proof} If $\eps>0$, let $J_\eps = \ccheck J_\eps$ be a (group) Friedrichs' mollifier.
Since $\beta\in L^{Q/(Q-2)}(\he n, E_0^n)$ then
$$
J_\eps\ast \beta \to \beta\qquad\mbox{in $L^{Q/(Q-2)}(\he n, E_0^n)$.}
$$
In particular, 
\begin{equation}\label{nov5 eq:4}
T_{J_\eps\ast \beta} \to T_\beta \qquad\mbox{in $\mc D'(\he n, E_0^{n+1}$).}
\end{equation}

In addition, since $W\beta\in L^{Q/(Q-1)}(\he n, E_0^n)$
$$
W(J_\eps\ast \beta)= J_\eps \ast W\beta \to W\beta \qquad\mbox{in $L^{Q/(Q-1)}(\he n, E_0^n)$.}
$$
Thus
$$
J_\eps\ast \beta \to \beta  \qquad\mbox{in $ \BLh{1}{Q/(Q-1)}{n}$.}
$$
If $\phi$ is a test form, $\| \phi\|_{L^\infty}\le 1$, we have now
\begin{equation*}\begin{split}
&\Scal{\partial_c  (J_\eps\ast T_\beta)}{\phi}_{\mc D',\mc D} = \Scal{J_\eps \ast \partial_c  T_\beta}{\phi}_{\mc D',\mc D}
\qquad\mbox{(by \eqref{oct17 eq:1})}
\\&
= \Scal{ \partial_c  T_\beta}{J_\eps \ast\phi}_{\mc D',\mc D}
\qquad\mbox{(by Definition \ref{regolarizzazione di una corrente})}
\\&
\le \mc M( \partial_c  T_\beta) \|J_\eps \ast\phi\|_{L^\infty} \le \mc M( \partial_c  T_\beta) \|\phi\|_{L^\infty},
\end{split}\end{equation*}
so that
$$
\mc M(\partial_c  (J_\eps\ast T_\beta)) \le \mc M(\partial_c   T_\beta).
$$
On the other hand, by \eqref{nov5 eq:4} and Lemma \ref{lsc},
$$
\mc M(\partial_c   T_\beta) \le \liminf_{\eps\to 0} \mc M(\partial_c  (J_\eps\ast T_\beta)),
$$
and then
\begin{equation}\label{nov5 eq:5}
\mc M(\partial_c  (J_\eps\ast T_\beta)) \to \mc M(\partial_c   T_\beta)\qquad\mbox{as $\eps\to 0$.}
\end{equation}
%Therefore there exists a decreasing sequence $(\eps_k)_{k\in\mathbb N}$ such that
%$$
%\mc M(\partial_c  (J_{\eps_k}\ast T_\beta)) < \mc M(\partial_c   T_\beta) + \frac1k.
%$$
Therefore, from now on we can assume $\beta\in \mc E(\he n, E_0^n)$.
Let us fix a sequence of cut-off functions $\{\chi_k\}_{i\in\N}\subset\mc D(\he n)$  such that for any $k\in\N$   ${\rm supp} (\chi_k)\subset {B(e,2k)}$, 
$\chi_k\equiv 1$ in  ${B(e,k)}$, and $k\,|W\chi_k|+k^2\,|W^2\chi_k| \le C$ for all $k\in\mathbb N$.

Set $\beta_k:=\chi_k\beta \in \mc D(\he n E_0^n)$. By Theorem \ref{embedding} $\beta\in L^{Q/(Q-2)(\he n, E_0^n)}$, 
and hence, obviously, 
\begin{equation}\label{nov8 eq:2}
\beta_k\to \beta \qquad\mbox{in $L^{Q/(Q-2)(\he n, E_0^n)}$ as $k\to\infty$.}
\end{equation}
In particular,
\begin{equation}\label{nov8 eq:1}
T_{\beta_k} \to T_\beta \qquad\mbox{in $\mc D'(\he n, E_0^{n+1})$
 as $k\to\infty$.}
\end{equation}
Moreover, again by Theorem \ref{embedding}, if $W$ is a horizontal
derivative, then $|W\beta| \in L^{Q/(Q-1)}(\he n, E_0^n)$. 
By dominated convergence,
$\chi_k W\beta\to W\beta$ in $L^{Q/(Q-1)}(\he n, E_0^n)$
as $k\to\infty$. On the other hand, again by Theorem \ref{embedding}, 
$\beta \in L^{Q/(Q-2)}(\he n, E_0^n)$ so that
\begin{equation*}\begin{split}
& \int_{\he n} (|W\chi_k| \,|\beta|)^{Q/(Q-1)}\, dV
\\&
\le k^{-Q/(Q-1)} \int_{B(e,2k)\setminus B(e,k)} |\beta|^{Q/(Q-1)}\, dV
\\&
\le k^{-Q/(Q-1)} \Big( \int_{B(e,2k)\setminus B(e,k)} |\beta|^{Q/(Q-2)} \, dV\Big)^{(Q-2)/(Q-1)}
\cdot\, |B(e,2k)|^{1/(Q-1)}
\\&
= C \Big( \int_{B(e,2k)\setminus B(e,k)} |\beta|^{Q/(Q-2)} \, dV\Big)^{(Q-2)/(Q-1)}
\to 0 \qquad\mbox{as $k\to\infty$.}
\end{split}\end{equation*}
Thus 
\begin{equation}\label{nov8 eq:3}
W\beta_k = \chi_k W\beta + (W\chi_k)\beta \to W\beta 
\end{equation}
in $L^p(\he n, E_0^n)$.
Combining \eqref{nov8 eq:2} and \eqref{nov8 eq:3} statement i) is proved

Now, let us prove ii). If $\psi$ is a test form, $\|\psi\|_{L^\infty}\le 1$,
\begin{equation*}\begin{split}
& \Scal{\partial_c  T_{\beta_k}}{\psi} = \Scal{T_{\chi_k \beta}}{d_c\psi} 
= \int_{\he n}\chi_k \beta \wedge d_c\psi  = \int_{\he n}d_c(\chi_k \beta )\wedge \psi
\\&
=
 \int_{\he n}d_c \beta \wedge \chi_k\psi + 
  \int_{\he n} P_0(W^2\chi_k)\beta\wedge\psi +  \int_{\he n} P_1(W\chi_k)\beta\wedge\psi
  \\&
  \le \mc M(\beta) \| \chi_k\psi\|_{L^\infty}
  +
  \frac{1}{k^2} \| \psi\|_{L^\infty}\int_{B(e,2k)\setminus B(e,k)} |\beta| \, dV
  \\&
  +
   \frac{1}{k} \| \psi\|_{L^\infty}\int_{B(e,2k)\setminus B(e,k)} |W\beta| \, dV
   \\&
  \le \mc M(\beta)
  +
  \frac{1}{k^2}  \Big(\int_{B(e,2k)\setminus B(e,k)} |\beta|^{Q/(Q-2)} \, dV\Big)^{(Q-2)/Q}
  \cdot | B(e,2k)|^{2/Q}
  \\&
  +
   \frac{1}{k} \Big(\int_{B(e,2k)\setminus B(e,k)} |W\beta|^{Q/(Q-1)} \, dV\Big)^{(Q-1)/Q} \cdot | B(e,2k)|^{1/Q}
     \\&
  \le \mc M(\beta)
  +
 C \Big(\int_{B(e,2k)\setminus B(e,k)} |\beta|^{Q/(Q-2)} \, dV\Big)^{(Q-2)/Q}
  \\&
  +
  C \Big(\int_{B(e,2k)\setminus B(e,k)} |W\beta|^{Q/(Q-1)} \, dV\Big)^{(Q-1)/Q} .
\end{split}\end{equation*}
Taking the supremum with respect to $\psi$, it follows that
\begin{equation*}\begin{split}
&\mc M(\partial_c  T_{\beta_k})
\le
 \mc M(\beta)
  +
 C \Big(\int_{B(e,2k)\setminus B(e,k)} |\beta|^{Q/(Q-2)} \, dV\Big)^{(Q-2)/Q}
  \\&
  +
  C \Big(\int_{B(e,2k)\setminus B(e,k)} |W\beta|^{Q/(Q-1)} \, dV\Big)^{(Q-1)/Q} .
\end{split}\end{equation*}
By Theorem \ref{embedding} $\beta\in L^{Q/(Q-2)}(\he n, E_0^n)$, and hence
$$
\int_{B(e,2k)\setminus B(e,k)} |\beta|^{Q/(Q-2)} \, dV \to 0 \qquad\mbox{as $k\to\infty$}.
$$
In addition, $|W\beta|\in L^{Q/(Q-1)(\he n, E_0^n)}$, and hence
$$
\int_{B(e,2k)\setminus B(e,k)} |W\beta|^{Q/(Q-1)} \, dV \to 0 \qquad\mbox{as $k\to\infty$}
$$
Thus
\begin{equation*}
\limsup_k  \ \mc M(\partial_c  T_{\beta_k })  \le \mc M(\beta).
\end{equation*}
On the other hand, combining \eqref{nov8 eq:1} and Lemma \ref{lsc},
$$
\mc M(\beta) \le \liminf_k  \ \mc M(\partial_c  T_{\beta_k })
$$
and statement ii) follows.

\end{proof}

\section{Continuous primitives of forms in $E_0^n$: proof of Theorem \ref{main global bis} }
\label{e0n}

Let us consider the following abstract setting:

\begin{definition}\label{spaces bis}
 We set
$$
E:=\mc C_0(\he n, E_0^{n}), \qquad F:= ( \BLh{1}{Q/(Q-1)}{n} )^*.
$$
$$
\mc D(A):= \{ \psi\in E, \partial_c  T_\psi\in F\}, \qquad   A\psi:= \partial_c  T_\psi,
$$
where, according to Proposition \ref{BL dual}, we use the identification of $( \BLh{1}{p}{n} )^*$
with a space of $n-$currents.

Again by Remark \ref{riesz}, the dual space of $E$ can be identified with the
set of $n$-currents with finite mass. 
\end{definition}

\begin{lemma} The domain $\mc D(A)$ is dense in $E$ and the map $A: \mc D(A)\to F$ is closed.
\end{lemma}

\begin{proof} The first assertion follows since $\mc D(\he n,E_0^n)\subset \mc D(A)$.
Moreover, if $\psi \in  \mc D(\he n,E_0^n)$, then $d_c \psi\in   \mc D(\he n,E_0^{n+1})
\subset L^{Q/2}(\he n,E_0^{n+1})$. Hence, by Lemma \ref{oct19 eq:1}, $T_{d_c\psi} \in
F$. By Proposition \ref{oct19 prop:1}, up to a sign, $T_{d_c\psi}=\partial_c  T_\psi$, hence we can conclude that $\psi\in\mc D(A)$.

Suppose now $(\psi_k)_{k\in\mathbb N}$ is a sequence in $\mc D(A)$ such that
$\psi_k\to \psi$ in $E$ and $\partial_c  T_{\psi_k}\to T$ in $F$. By definition,  $\mc D(\he n)$
 is dense in $  \BL{1}{Q/(Q-1)}$, hence for any $\varphi\in\mc D(\he n, E_0^n)$ we have
\begin{equation*}\begin{split}
&\Scal{\partial_c  T_\psi}{\varphi}= \Scal{ T_\psi}{d_c\varphi}=\int_{\he n}\psi\wedge d_c\varphi
\\&
=\lim_k\int_{\he n}\psi_k\wedge d_c\varphi=\lim_k\Scal{T_{\psi_k}}{d_c\varphi} = \Scal{ T}{\varphi}.
\end{split}\end{equation*}
Therefore $\partial_c  T_\psi = T \in  ( \BL{1}{Q/(Q-1)} )^*$.

\end{proof}

\begin{lemma}\label{approx} If  $\psi\in \mc D(A)$, then there exists a sequence 
$$(\psi_k)_{k\in\mathbb N}\quad\mbox{in}\quad
\mc E(\he n,E_0^n)\cap E
$$ such that
\begin{equation}\label{oct19 eq:3}
\psi_k \to \psi\quad\mbox{in $E\quad$ and}\quad A\psi_k\to A\psi \quad\mbox{weakly in $F$}
\end{equation}
as $k\to\infty$. 
%{\color{green} In addition, if $\beta\in W$, then
%\begin{equation}\label{oct29 eq:1}
%|\Scal{J(\beta)}{A\psi_k}_{F^*,F}| = |\int_{\he n} \beta\wedge d_c\psi_k|
%\end{equation}
%}

\end{lemma} 

\begin{proof} Consider a family of group Friedrichs' mollifiers $(J_\eps)_{\eps>0}$ with
$ J_\eps = \ccheck J_\eps$. Obviously  $J_\eps\ast \psi \in \mc E(\he n,E_0^n)\cap E$.
In addition, if we take $\varphi\in\mc D(\he n,E_0^n)$, $\|\varphi\|_{ \BL{1}{Q/(Q-1)}(\he n, E_0^n)}
\le 1$, by \eqref{oct17 eq:1},
\begin{equation*}\begin{split}
\Scal{\partial_c  T_{J_\eps*\psi }}{\varphi} &
=\Scal{J_\eps * \partial_c  T_\psi}{\varphi}
= \Scal{\partial_c  T_\psi}{J_\eps * \varphi}
\\&
\le
\| \partial_c  T_\psi\|_F \| J_\eps * \varphi\|_{\BLh{1}{Q/(Q-1)}{n} }
\le
C \| \partial_c  T_\psi\|_F.
\end{split}\end{equation*}
Therefore $\{ {\partial_c  T_{J_\eps*\psi }}\} $ is bounded in $F$ and then we may assume
that there is a subsequence weakly convergent in $F$. Then, since $A$ is closed, \eqref{oct19 eq:3}
follows.

%{\color{green}
%As for \eqref{oct29 eq:1}
%\begin{equation*}\begin{split}
% |\Scal{J(\beta)}{A\psi_k}_{F^*,F}| &= |\Scal{A\psi_k}{\beta}_{W^*,W}| 
%= |\Scal{\partial_c  T_{\psi_k}}{\beta}_{W^*,W}|
%\\&
% = |\Scal{T_{d_c\psi_k}}{\beta}_{W^*,W}|
%=  |\Scal{\beta}{T_{d_c\psi_k}}_{W^*,W}|
%\\& = |\int_{\he n} \beta\wedge d_c\psi_k|
%\end{split}\end{equation*}
%}
 
\end{proof}

\begin{proposition} The domain $\mc D(A^*)$ is dense in $F^*$
and 
\begin{equation}\label{oct19 eq:4}
\mc D (A^*) = \{\tau(\beta)\,:\, \beta \in  \BLh{1}{Q/(Q-1)}{n}\, ; \mc M(\partial_c  T_\beta)<\infty\}.
\end{equation}

In addition, if $\beta\in \mc D (A^*)$, then, by definition, for any $\psi\in \mc D(A)$
\begin{equation*}
\Scal{A^*\beta}{\psi}_{E^*,E} =  \Scal{\beta}{A\psi}_{F^*,F}.
\end{equation*}
In particular, if $\psi\in \mc D(\he n. E_0^n)$, then
\begin{equation}\label{nov8 eq:5}
\Scal{A^*\tau(\beta)}{\psi}_{E^*,E} = \Scal{A^*(\tau(\beta))}{\psi}_{\mc D',\mc D} =  \Scal{\partial_c  T_{\beta}}{\psi}_{\mc D',\mc D},
\end{equation}
i.e. 
$$
A^*(\tau(\beta)) = \partial_c  T_{\beta}
$$
in the sense of currents.
\end{proposition}

\begin{proof}
Since $F$ is reflexive,
then $\mc D(A^*)$ is dense in $F^*$ (see \cite{brezis}, Remark 15 of
Section 2.6). In the sequel, to avoid cumbersome notations, we shall also write
$$\mc W:  =\BLh{1}{Q/(Q-1)}{n},$$
 so that $F=\mc W^*$.
To prove \eqref{oct19 eq:4} we  first show that $\mc D (A^*)\subset\{\tau(\beta)\,:, \beta \in \mc W\, ; \mc M(\partial_c  T_\beta)<\infty\}$ that is, if $\tau(\beta)\in F^*$ and there exists $c_\beta$ such that
\begin{equation}\label{apr 22 eq:1 bis}
|\Scal{\tau(\beta)}{A\psi}_{F^*,F}|
\le c_\beta \|\psi\|_E \qquad\mbox{for all $\psi\in \mc D(A)$},
\end{equation}
then $\mc M(\partial_c  T_\beta)<\infty$.

 First of all, Remark \ref{tau} says that
\begin{equation}\label{tau2}
\Scal{\tau(\beta)}{A\psi}_{F^*,F} = \Scal{A\psi}{\beta}_{\mc W^*,\mc W},
\end{equation}
so that \eqref{apr 22 eq:1 bis} becomes
\begin{equation}\label{apr 22 eq:1 ter}
|\Scal{\partial_c  T_\psi}{\beta}_{\mc W^*, \mc W}| = |\Scal{A\psi}{\beta}_{\mc W^*,\mc W}|
\le c_\beta \|\psi\|_E \qquad\mbox{for all $\psi\in \mc D(A)$}.
\end{equation}
Now, if $\beta\in \mc W$ satisfies \eqref{apr 22 eq:1 ter}, by definition
there exists a sequence $(\beta_N)_{N\in\mathbb N}$ of test forms
converging to $\beta$ in $\mc W$, and then, by Theorem \ref{embedding}, converging to $\beta$ also in
$L^1_{\mathrm{loc}} (\he n,E_0^n)$.  Thus, 
$$
T_{\beta_N}\to T_\beta\qquad\mbox{in $\mc D'(\he n, E_0^{n+1})$}
$$
and
$$
\partial_c  T_{\beta_N}\to \partial_c  T_\beta\qquad\mbox{in $\mc D'(\he n,E_0^n)$.}
$$
Take now $\varphi\in \mc D(\he n, E_0^n)$, $\|\varphi\|_E \le 1$. By Lemma \ref{by parts} we have:
\begin{equation*}\begin{split}
&\big|\Scal{\partial_c  T_\beta}{\varphi}_{\mc D',\mc D} \big|  =
\big| \lim_N \Scal{\partial_c  T_{\beta_N}}{\varphi}_{\mc D',\mc D}\big| = \big|\lim_N \Scal{T_{\beta_N}}{d_c\varphi}_{\mc D',\mc D}\big|
\\& \hphantom{xxx}=  \big|\lim_N \int_{\he n} \beta_N\wedge d_c\varphi\big|
=  \big|\lim_N \int_{\he n} d_c\beta_N\wedge \varphi\big|
\\& \hphantom{xxx} =  \big| \lim_N \Scal{T_\varphi}{d_c\beta_N}_{\mc D',\mc D}\big| =  \big| \lim_N \Scal{\partial_c  T_\varphi}{\beta_N}_{\mc D',\mc D}\big|
\\& \hphantom{xxx} = \big| \lim_N \Scal{A\varphi}{\beta_N}_{\mc D',\mc D}\big|
 = \big| \lim_N \Scal{A\varphi}{\beta_N}_{\mc W^*,W}\big|
\\&  \hphantom{xxx} \qquad\mbox{(by \eqref{nov2 eq:2}, since $\partial_c  T_\varphi = \pm T_{d_c\varphi}$, and $d_c\varphi\in L^1_{\mathrm{loc}}$)}
\\& \hphantom{xxx}= \big|  \Scal{A\varphi}{\beta}_{\mc W^*, \mc W}\big|
\le c_\beta \|\varphi\|_E \le c_\beta ,
\end{split}\end{equation*}
by \eqref{apr 22 eq:1 ter}, since $\mc D(\he n, E_0^n) \subset \mc D(A)$. Taking the
supremum with respect to $\varphi$ we conclude that $\mc M(\partial_c  T_\beta) < \infty$.

Now we have to prove the reverse inclusion: take $\beta \in  \BLh{1}{Q/(Q-1)}{n} $ such that
$\mc M(\partial_c  T_\beta)<\infty$. We must prove that \eqref{apr 22 eq:1 bis} holds.
By Lemma \ref{approx}, there exists a sequence $(\psi_k)_{k\in\mathbb N}$
in $\mc E(\he n,E_0^n)\cap E$ such that
\begin{equation}\label{oct19 eq:3bis}
\psi_k \to \psi\quad\mbox{in $E\quad$ and}\quad A\psi_k\to A\psi \quad\mbox{weakly in $F$}
\end{equation}
as $k\to\infty$. Hence
$$
|\Scal{\tau(\beta)}{A\psi}_{F^*,F}| = \lim_k |\Scal{\tau(\beta)}{A\psi_k}_{F^*,F}|.
$$
Suppose for a while we have proved \eqref{apr 22 eq:1 bis} for 
$\psi_k\in \mc E(\he n,E_0^n)\cap \mc D(A)$; then
$$
|\Scal{\tau(\beta)}{A\psi_k}_{F^*,F}| \le c_\beta \lim_k \|\psi_k\|_E = c_\beta \|\psi\|_E,
$$
i.e. \eqref{apr 22 eq:1 bis} holds for any $\psi\in E$. Therefore, we may suppose $\psi\in \mc E(\he n,E_0^n)\cap \mc D(A)$.

By definition of $\mc D(A)$, $A\psi \in \mc W^*$. Let
$(\beta_k)_{k\in\mathbb N}$ the approximation of $\beta$ of Lemma \ref{approx bis}. 
By \eqref{nov2 eq:2}, we have
\begin{equation}\label{nov5 eq:2}\begin{split}
&\big| \Scal{\tau(\beta)}{A\psi}_{F^*,F}\big| =\big| \Scal{A\psi}{\beta}_{\mc W^*,\mc W}\big|
=\lim_k \big|\Scal{A\psi}{\beta_k}_{\mc W^*,\mc W}\big|
\\&  \hphantom{xxx}
= \lim_k\big| \Scal{A\psi}{\beta_k}_{\mc D',\mc D}\big| = \lim_k \big|\Scal{\partial_c  T_\psi}{\beta_k}_{\mc D',\mc D}\big|
\\&  \hphantom{xxx}
 = \lim_k \big| \Scal{ T_\psi}{d_c\beta_k}_{\mc D',\mc D}\big| = \lim_k \big| \int_{\he n}\psi\wedge d_c\beta_k\big|.
% \\&  \hphantom{xxx}
%= \lim_k \int_{\he n}d_c\psi\wedge \beta_k\qquad\mbox{(by Lemma \ref{by parts})}
\end{split}\end{equation}
Let now $k\in\mathbb N$ be fixed for a while, and let $\chi\in \mc D(\he n)$, $0\le \chi\le 1$
be a cut-off function, $\chi\equiv 1$ on $\supp \beta_k$. Then, keeping in mind that
$\chi\psi$ is a test form,
\begin{equation}\label{nov7 eq:1}\begin{split}
& \big | \int_{\he n}\psi\wedge d_c\beta_k\big | = \big |\int_{\he n}(\chi \psi)\wedge d_c\beta_k\big |
\\&  \hphantom{xxx}
= \big |\int_{\he n} d_c(\chi \psi)\wedge \beta_k\big |  \qquad\mbox{(by Lemma \ref{by parts})}
\\&  \hphantom{xxx}
= \big |\Scal{T_{\beta_k}}{d_c(\chi \psi)} \big | \qquad\mbox{(by \eqref{nov5 eq:1})}
\\&  \hphantom{xxx}
=\big |\Scal{\partial_c  T_{\beta_k}}{\chi \psi} \big |
\le \mc M(\partial_c  T_{\beta_k}) \| \chi \psi\|_E \le  \mc M(\partial_c  T_{\beta_k}) \| \psi\|_E.
\end{split} \end{equation}
Thus, combining \eqref{nov7 eq:1} with \eqref{nov5 eq:2} and keeping in mind Lemma \ref{approx bis}, ii),
we obtain eventually
\begin{equation*}
\big| \Scal{\tau(\beta)}{A\psi}_{F^*,F}\big| \le \mc M(\partial_c  T_{\beta}) \| \psi\|_E,
\end{equation*}
i.e. \eqref{apr 22 eq:1 bis} is proved.

We are now left to prove \eqref{nov8 eq:5}. First, we remind that, following \cite{brezis}, Section 2.6, 
if $\tau(\beta)\in\mc D(A^*)$, then $A^*(\tau(\beta))$ is uniquely determined by the identity
$$
\Scal{A^*(\tau(\beta))}{\psi}_{E^*,E} = \Scal{\tau(\beta)}{A\psi}_{F^*,F}\qquad\mbox{for $\psi\in\mc D(A)$.} 
$$

Finally, take $\psi\in \mc D(\he n, E_0^h)$.
%and
%set $\tilde\beta:=\tau^{-1}\beta$.
By Lemma \ref{approx bis}, i) there exists a sequence compactly supported smooth
forms $(\beta_k)_{k\in\mathbb N}$ converging to $\beta$ in $\mc W$.
Now, by \eqref{tau2},
\begin{equation*}\begin{split}
&
\Scal{A^*(\tau(\beta))}{\psi}_{E^*,E} = \Scal{\tau(\beta)}{A\psi}_{F^*,F} 
= \Scal{A\psi}{\beta}_{\mc W^*,\mc W} 
\\&
=\lim_k \Scal{A\psi}{\beta_k}_{\mc W^*,\mc W}
=\lim_k \Scal{A\psi}{\beta_k}_{\mc D',\mc D}
\\&
= \lim_k \Scal{\partial_c  T_\psi}{\beta_k}_{\mc D',\mc D}
= \lim_k\int_{\he n} \psi\wedge d_c \beta_k
\\&
=\lim_k\int_{\he n} d_c\psi\wedge \beta_k
= 
\int_{\he n} d_c\psi\wedge \beta
= \Scal{\partial_c  T_{\beta}}{\psi}_{\mc D', \mc D},
\end{split}\end{equation*}
and we are done.

\end{proof}

In order to prove Theorem \ref{poincareglobal bis} we  need a last
 result, that is akin to Lemma \ref{parabolic} which was used in the proof of Theorem \ref{main global}. To
avoid cumbersome notations, let us write again $\mc W:= \BLh{1}{Q/(Q-1)}{n}$.

\begin{lemma}\label{parabolic bis} Suppose $\beta\in \BLh{1}{Q/(Q-1)}{n}$ 
and $\Omega\in(\BLh{1}{Q/(Q-1}{n})^*$  are such that 
$$\partial_c  T_\beta=0 \qquad\mbox{and}\qquad \partial_c \Omega = 0.$$
Then
$$
\Scal{\Omega}{\beta}_{\mc W^*,\mc W}=0.
$$
\end{lemma}

\begin{proof} By Remark \ref{mar 2 rem:1}, $d_c\beta=0$ in the sense of distributions.
If $\eps>0$, let $J_\eps = \ccheck J_\eps$ be a (group) Friedrichs' mollifier. Then
$$
J_\eps \ast \beta \to \beta \quad\mbox{in $\quad \BLh{1}{Q/(Q-1)}{n}$}
$$
and
$$
d_c(J_\eps \ast \beta) = 0.
$$
Thus, without lack of generality, we can assume $\beta\in \mc W\cap \mc E(\he n, E_0^n)$.
In particular, $\beta\in L^{Q/(Q-2)}(\he n, E_0^n)$.
Keeping in mind Theorem \ref{hls folland}, let us set
$$
\phi:= d^*d_c d_c^*\Delta_c^{-1} \beta \in L^{Q/(Q-3)}(\he n, E_0^{n-1}),
$$
since $d^*d_c d_c^*\Delta_c^{-1}$ is a kernel of type 1. As in \cite{BFP2}, Section 1.6,
$$
d_c\phi = \beta.
$$
Let us fix a sequence of cut-off functions $\{\chi_N\}_{N\in\N}\subset\mc D(\he n)$  such that for any $N\in\N$,   ${\rm supp} (\chi_N)\subset {B(e,2N)}$, 
$\chi_N\equiv 1$ in  ${B(e,N)}$, and 
$$N\,|W\chi_N| +N^2\,|W^2\chi_N| \le C$$ 
for all $N\in\mathbb N$.

We notice  that $\chi_N \beta\in \mc D(\he n, E_0^n)$ and
$$
\chi_N \beta \to \beta\qquad \mbox{in $\mc W$.}
$$
The proof of this fact is already done in Lemma \ref{approx bis}.
%{\color{blue} but for reader convenience we repeat it. {\bf FOR ME WE CAN OMIT THIS PART}. 
%Indeed, trivially, $\chi_N \beta \to \beta$ in $L^{Q/(Q-2)}(\he n,E_0^n)$. In addition, 
%if $W_i$ is a horizontal vector field,
%$$
%W_i(\chi_N \beta) = \chi_N W_i\beta + (W\chi_N)\beta.
%$$
%Again, trivially, $\chi_N W_i\beta\to W_i\beta$ 
% in $L^{Q/(Q-1)}(\he n, E_0^n)$.
% 
% On the other hand,
% \begin{equation*}\begin{split}
%\int_{\he n} &|(W\chi_N)\beta |^{Q/(Q-1}\, dp \le C
%N^{-Q/(Q-1} \int_{\he n} |\beta |^{Q/(Q-1}\, dp
%\\&
%\le C N^{-Q/(Q-1} \Big(\int_{B(e,2N)\setminus B(e,N)} |\beta |^{Q/(Q-2}\, dp\Big)^{(Q-2)/(Q-1)} |B(e,N)|^{1/(Q-1)}
%\\&
%\le
%C \Big(\int_{B(e,2N)\setminus B(e,N)} |\beta |^{Q/(Q-2}\, dp\Big)^{(Q-2)/(Q-1)} \to 0
%\end{split}\end{equation*}
%as $N\to\infty$ since $\beta\in L^{Q/(Q-2)}(\he n E_0^n)$.}

Thus we can write
$$
\Scal{\Omega}{\beta}_{\mc W^*,\mc W} = \lim_N  \Scal{\Omega}{\chi_N\beta}_{\mc W^*,\mc W}
 = \lim_N  \Scal{\Omega}{\chi_N d_c\phi}_{\mc W^*,\mc W}.
$$
Therefore, the statement of the Lemma will follow by showing that
\begin{equation}\label{dec5 eq:1}
\lim_N  \Scal{\Omega}{\chi_N d_c\phi}_{\mc W^*,\mc W} = 0.
\end{equation}
First of all, by Lemma \ref{leibniz},
$$
d_c(\chi_N \phi) = \chi_N d_c\phi + P_0^{n-1}(W\chi_N )\phi,
$$
where $P_0^{n-1}(W\chi_N ): E_0^{n-1} \to E_0^{n}$ is a linear homogeneous 
differential operator of order zero with coefficients depending
only on the horizontal derivatives of $\chi_N$. Thus, by \eqref{nov2 eq:2},
\begin{equation}\label{dec5 eq:2}\begin{split}
&\Scal{\Omega}{\chi_N d_c\phi}_{\mc W^*,\mc W}
\\&
  \hphantom{xxx} = \Scal{\Omega}{d_c(\chi_N \phi)}_{\mc W^*,\mc W} - \Scal{\Omega}{P_0^{n-1}(W\chi_N )\phi}_{\mc W^*,\mc W}
\\&
\hphantom{xxx} = \Scal{\Omega}{d_c(\chi_N \phi)}_{\mc D',\mc D} - \Scal{\Omega}{P_0^{n-1}(W\chi_N )\phi}_{\mc W^*,\mc W}
\\&
\hphantom{xxx}
= \Scal{\partial_c \Omega}{\chi_N \phi}_{\mc D',\mc D} - \Scal{\Omega}{P_0^{n-1}(W\chi_N )\phi}_{\mc W^*,\mc W}
\\&
\hphantom{xxx} = - \Scal{\Omega}{P_0^{n-1}(W\chi_N )\phi}_{\mc W^*,\mc W},
\end{split}\end{equation}
since $\partial_c \Omega=0$.

We notice now that $P_0^{n-1}(W\chi_N )\phi$ can be seen as a linear combination of the element
of the basis $\Xi_0^n$ of $E_0^n$ of the form $(W_k\chi_N)\phi_\ell$, where the $\phi_\ell$ are
coordinates of $\phi$ with respect to the basis $\Xi_0^{n-1}$.

Thus, in order to prove \eqref{dec5 eq:1}, we have but to show that
\begin{equation}\label{dec5 eq:3}
(W_k\chi_N)\phi_\ell \to 0\quad\mbox{in $\BL{1}{Q/(Q-1)}$.}
\end{equation}
The proof of \eqref{dec5 eq:3} will be articulated in three steps.

\noindent {\bf Step 1}. Consider first
\begin{equation*}\begin{split}
&\| (W_k\chi_N) \phi_\ell\|_{L^{Q/(Q-2)}(\he n)}^{Q/(Q-2)} \le 
N^{-Q/(Q-2)} \int_{B(e,2N)\setminus B(e,N)} |\phi |^{Q/(Q-2)}\, dp
\\&
\le C N^{-Q/(Q-2)}\Big(\int_{B(e,2N)\setminus B(e,N)} |\phi |^{Q/(Q-3)}\, dp\Big)^{(Q-3)/(Q-2)} |B(e,2N)|^{1/(Q-2)}
\\&
\le C \Big(\int_{B(e,2N)\setminus B(e,N)} |\phi |^{Q/(Q-3)}\, dp\Big)^{(Q-3)/(Q-2)} \to 0
\end{split}\end{equation*}
as $N\to\infty$, since $\phi\in L^{Q/(Q-3)}(\he n, E_0^{n-1})$.

\noindent {\bf Step 2.} If $W_i$ is a horizontal derivative, consider now the $L^{Q/(Q-1)}$-norm 
of $W_i ((W_k\chi_N)  \phi_\ell)$. It holds,
$$
W_i ((W_k\chi_N)  \phi_\ell) = (W_iW_k \chi_N)\phi_\ell + (W_k\chi_N) W_i(\phi_\ell).
$$
\noindent {\bf Step 2a} Now,
\begin{equation*}\begin{split}
\| & (W_iW_k \chi_N)  \phi_\ell\|_{L^{Q/(Q-1)}(\he n)}^{Q/(Q-1)} \le 
N^{-2Q/(Q-1)} \int_{B(e,2N)\setminus B(e,N)} |\phi |^{Q/(Q-1)}\, dp
\\&
\le C N^{-2Q/(Q-1)}\Big(\int_{B(e,2N)\setminus B(e,N)} |\phi |^{Q/(Q-3)}\, dp\Big)^{(Q-3)/(Q-1)} |B(e,2N)|^{2/(Q-1)}
\\&
\le C \Big(\int_{B(e,2N)\setminus B(e,N)} |\phi |^{Q/(Q-3)}\, dp\Big)^{(Q-3)/(Q-1)} \to 0
\end{split}\end{equation*}
as $N\to\infty$, again since $\phi\in L^{Q/(Q-3)}(\he n, E_0^{n-1})$.

\noindent {\bf Step 2b.} Consider eventually the $L^{Q/(Q-1)}$-norm 
of 
$$(W_k\chi_N) W_i(\phi_\ell) =(W_k\chi_N)  W_i (d_c^*d_c d_c^* \Delta_c^{-1}\beta)_\ell =: (W_k\chi_N) \tilde\phi_\ell.
$$
We notice now that the map $\beta\to \tilde\phi_\ell$ is associated with a kernel of type 0 and then
is continuous from $L^{Q/(Q-2)}(\he n)$ to itself (again by Theorem
\ref{hls folland}, since $Q/(Q-2)>1$). Therefore
\begin{equation*}\begin{split}
\| & (W_k \chi_N)  \tilde\phi_\ell\|_{L^{Q/(Q-1)}(\he n)}^{Q/(Q-1)} \le 
N^{-Q/(Q-1)} \int_{B(e,2N)\setminus B(e,N)} |\tilde\phi_\ell |^{Q/(Q-1)}\, dp
\\&
\le C N^{-Q/(Q-1)}\Big(\int_{B(e,2N)\setminus B(e,N)} |\tilde\phi_\ell |^{Q/(Q-2)}\, dp\Big)^{(Q-2)/(Q-1)} |B(e,2N)|^{1/(Q-1)}
\\&
\le C \Big(\int_{B(e,2N)\setminus B(e,N)} |\tilde\phi |^{Q/(Q-2)}\, dp\Big)^{(Q-2)/(Q-1)} \to 0
\end{split}\end{equation*}
since $\tilde\phi_\ell\in L^{Q/(Q-2)}(\he n)$.

This completes the proof of \eqref{dec5 eq:3} and hence of the Lemma.
\end{proof}

We are now in a position to conclude the proof of Theorem \ref{main global bis}.

\begin{proof}[Proof of Theorem \ref{main global bis}] 
With the notations of Definition \ref{spaces bis},
let us show preliminarily that
\begin{equation}\label{feb 15 eq:1 bis}
A^*(\mc D(A^*))\qquad\mbox{is closed in $E^*$.}
\end{equation}
To this end, let $(T_k)_{k\in\mathbb N}$ a sequence of $n$-currents
in $A^*(\mc D(A^*))$ that converges to a current 
$T\in E^*$ (i.e. in the mass norm). Hence, $\mc M(T_k)
= \|T_k\|_{E^*}\le C_1$ for all $k\in\mathbb N$.
In particular
 $T_k\to T$ weakly* in  $\mc D'(\he n E_0^n)$, i.e.
 $$
 \Scal{T}{\sigma}=\lim_{k\to\infty} \Scal{T_k}{\sigma}
 $$
 for all $\sigma\in \mc D(\he n,E_0^n)$.
By \eqref{oct19 eq:4} and \eqref{nov8 eq:5} there exists a corresponding
sequence $\beta_k\in 
 \BLh{1}{Q/(Q-1)}{n}$ such that
\begin{equation*}\label{feb 15 eq:2 bis}
T_k = \partial_c  T_{\beta_k}
\end{equation*}
for any $k\in\mathbb N$. Since the $n$-currents $\partial_c  T_{\beta_k}$'s 
satisfy
$\mc M(\partial_c  T_{\beta_k}) = \mc M(T_k) <\infty$, by Theorem \ref{poincareglobal currents bis},
for any $k\in\mathbb N$ there exists $\phi_k\in \BLh{1}{Q/(Q-1)}{n}$
such that $\partial_c  T_{\phi_k} = \partial_c  T_{\beta_k}$, and 
$$
\|\phi_k\|_{\BLh{1}{Q/(Q-1)}{n}} \le 
 C\mc M(T_k) \le CC_1.
$$
Since $Q/(Q-1)>1$ we cas assume that 
$$\phi_k\to \phi\qquad\mbox{weakly in
$\BLh{1}{Q/(Q-1)}{n} $.}
$$
Thus, by Proposition \ref{weak and weak*},
$T_{\phi_k} \to T_{\phi}$ weakly* in $\mc D'(\he n,E_0^{n+1})$ and $T_k 
=\partial_c  T_{\phi_k} \to \partial_c  T_{\phi}$ weakly* in $\mc D'$; therefore,
since also $T_k\to T$ weakly*, it follows that $T= \partial_c  T_{\phi}$.
To prove that $T\in A^*(\mc D(A^*))$, by \eqref{oct19 eq:4} we need merely 
show that $\mc M(\partial_c  T_\phi)<\infty$. Because of the lower
semicontinuity of the mass with respect to the weak* convergence,
we have
$$
\mc M(\partial_c  T_\phi)\le\liminf_{k\to \infty}\mc M(\partial_c  T_{\beta_k})
= \liminf_{k\to \infty}\mc M(T_k) = \mc M(T).
$$
Thus \eqref{feb 15 eq:1 bis} is proved.

By Theorem \ref{predual}, 1), \eqref{feb 15 eq:1 bis} implies that
\begin{equation}\label{feb 19 eq:1 bis}
A(\mc D(A)) = (\ker A^*)^\perp.
\end{equation}
Moreover, by Lemma \ref{predualpoin}, \eqref{feb 15 eq:1 bis} implies that
there exists $C>0$ such that  for all  $f\in A(D(A))$ there exists $e\in D(A)$, 
satisfying 
\begin{equation}\label{mar 3 eq:1 bis}
Ae=f\quad\mbox{and}\quad \|e\|_{E}\leq C\,\|f\|_{F}.
\end{equation}

To complete the proof, taking into account \eqref{feb 19 eq:1 bis}, we only need to show that 
\begin{equation}\label{feb 19 eq:2 bis}
\{\Omega\in(\BLh{1}{Q/(Q-1}{n})^*, \, \partial_c \Omega = 0\}\subset (\ker A^*)^\perp.
\end{equation}

 Indeed, take $\Omega\in(\BLh{1}{Q/(Q-1}{n})^*$ such that $ \partial_c \Omega = 0$.
For any $\beta\in \BLh{1}{Q/(Q-1}{n}$ belonging to $\ker A^*$ (i.e. $\partial_c  T_\beta=0$),
by Lemma \ref{parabolic bis}
$$
\Scal{\Omega}{\beta}= 0,
$$
that proves \eqref{feb 19 eq:2 bis}.
%Thus, combining  \eqref{feb 19 eq:2 bis}  and \eqref{feb 19 eq:1 bis}, 
%\begin{equation*}
%\{\Omega\in(\BLh{1}{Q/(Q-1}{n})^*, \, \partial_c \Omega = 0\} \subset
%A(\mc D(A)),
%\end{equation*}
%and hence, by \eqref{mar 3 eq:1 bis}, 
Hence, we have proved the theorem.

\end{proof}

\section{Continuous primitives of locally defined forms: proof of Theorem \ref{main local}}
\label{local}

From the global Poincar\'e inequalities proved in the previous sections, we can now quite easily obtain the interior ones,  stated in Theorem \ref{main local}. In the next section we need first to collect some results already proven in \cite{BFP2} and \cite{BFP5}.

\subsection{Intermediate tools}

The proof of Theorem \ref{main local}, when $h=n+1$,  relies also on an approximated homotopy formula proved in \cite{BFP2} for closed forms $\omega$. Indeed, 
if we take $\lambda >\lambda' >1$  and set
$B':= B(e,\lambda')$ and  $B_\lambda:= B(e,\lambda)$,
in Theorem 
5.19 in \cite{BFP2} formula (56),  it is proved  
that\begin{equation}\label{9 nov}
\omega = d_cT\omega + S\omega \qquad\mbox{in $B'$}
\end{equation}
where
$$  T: L^{Q}(B_\lambda,E_0^{h}) \to W^{1,Q}(B',E_0^{h-1})\quad \mbox{for $h\neq n+1$} 
$$
and
$$T: L^{Q/2}(B_\lambda,E_0^{n+1}) \to W^{2,Q/2}(B',E_0^{n}).$$
Moreover, 
 $S$ is a regularizing operator, i.e. maps $L^{p}(B_\lambda,E_0^{n+1})$
into $\mc E (B',E_0^{n+1})$. In particular, $S$
is bounded from $W^{m,p}(B_\lambda,E_0^{n+1})$ to $W^{s,p}(B',E_0^{n+1})$ for any $m,s$,
$0\le m<s$ and $1<p<\infty$ (see also Theorem 5.14 and Theorem 5.15 of \cite{BFP2}). 

In addition, again in \cite{BFP2} (see formula (37) and Lemmata 5.7 and 5.8) an operator $K$ that inverts Rumin's differential $d_c$ on closed forms of the same degree has been defined.
 More
precisely, we have:
\begin{lemma}\label{homotopy 1} If $\omega$ is a smooth $d_c$-exact differential form in the ball $B'$, then
\begin{equation}\label{homotopy closed}
\omega = d_cK\omega \quad\mbox{if $1\le h\le 2n+1$.}
\end{equation}
Morever,
\begin{itemize}
\item[i)] if $1< p < \infty$ and $h=1,\dots, 2n+1$, then $K :W^{1,p}(B', E_0^h) \to
L^p(B', E_0^{h-1})$ is bounded;
\item[ii)] if $1 < p < \infty$ and $n+1<h \le 2n+1$, then $K :L^{p}(B', E_0^h) \to
L^p(B', E_0^{h-1})$ is compact;
\item[iii)] if $1< p< \infty$ and $h=n+1$, then $K:L^{p}(B', E_0^{n+1}) \to
L^p(B', E_0^{n})$ is bounded.
\end{itemize}
%In addition, if $\omega$ is compactly supported in $B$, then $J\omega$ is still compactly supported in $B$.
\end{lemma} 

Let us recall the precise definition of the operator $K$.
In \cite{IL} the authors proved  that, starting from Cartan's homotopy formula, 
if $D \subset \rn {N}$ is a convex set and
$1<p<\infty$, $1\le h \le N$, then  there exists a bounded linear map: 
\begin{equation*}
K_{\mathrm{Euc}}  : L^p(D, {\bigwedge}\vphantom{!}^h)\to W^{1,p}_{\mathrm{Euc}}(D, {\bigwedge}\vphantom{!}^{h-1})
\end{equation*}
 that
is a homotopy operator, i.e.
\begin{equation}\label{may 4 eq:1}
	\omega =  dK_{\mathrm{Euc}} \omega + K_{\mathrm{Euc}}d\omega \qquad \mbox{for all 
	$\omega\in C^\infty (D, {\bigwedge}\vphantom{!}^h)$}.
\end{equation}
(see Proposition 4.1 and Lemma 4.2 in \cite{IL}). 

The operator $K$ defined in \cite{BFP2} is
\begin{eqnarray}\label{may 4 eq:2}
K=\Pi_{E_0}\circ \Pi_E \circ K_{\mathrm{Euc}}   \circ \Pi_E\,.
\end{eqnarray}
%(for the sake of simplicity, from now on we drop the index $h$ - the degree of the form -
%writing, e.g., $K_{\mathrm{Euc}}$ instead of $K_{\mathrm{Euc},h}$).

\bigskip

The following theorem provides a  continuity result in $W^{k,p}$ of Iwaniec \& Lutoborski's kernel $K_{\mathrm{Euc}} $,
though with a loss on domain. It has been proved in \cite{BFP5}, Theorem 5.1. 

\begin{theorem}\label{chapeau}
 Let $B_{\mathrm{Euc}}(0,a)$ and $B_{\mathrm{Euc}}(0,a')$, $0<a<a'$, be concentric Euclidean balls in $\rn N$. 
Then for $k\in \mathbb N$ and $p\in[1,\infty]$, Iwaniec  and Lutoborski's homotopy $K_{\mathrm{Euc}}$ is a bounded operator
\begin{equation*}
K_{\mathrm{Euc}}  : W^{k,p}_{\mathrm{Euc}}(B_{\mathrm{Euc}}(0,a'), {\bigwedge}\vphantom{!}^h )\to W^{k,p}_{\mathrm{Euc}}
(B_{\mathrm{Euc}}(0,a), {\bigwedge}\vphantom{!}^{h-1})
\end{equation*}
(the spaces $W^{m,p}_{\mathrm{Euc}}(U,\cov{h})$  and
 $W^{m,p}_{\mathrm{Euc}}(U,E_0^{h})$ 
are defined replacing Folland-Stein Sobolev spaces by usual Sobolev spaces).
\end{theorem}

\subsection{The interior result}

\begin{theorem}\label{main local fin}
There exists $\lambda>1$, such that, if we set $B:=B(e,1)$ and $B_\lambda:=B(e,\lambda)$,  
then there exists a geometric constant
$C>0$ such that
\begin{itemize}
\item[i)] if $2\le  h \le { 2n+1}$, $h\neq n+1$, then a $d_c$-exact  form  $\omega\in L^Q(B_\lambda, E_0^h)$
admits a primitive in $\phi\in \mc C(B, E_0^{h-1})$ such that
\begin{equation}\label{feb 19 eq:3 fin}
\|\phi\|_{\mc C(\overline B, E_0^{h-1})} \le C \|\omega\|_{L^Q(B_\lambda, E_0^h)};
\end{equation}
\item[ii)] a $d_c$-exact form  $\omega\in L^{Q/2}(B_\lambda, E_0^{n+1})$
admits a primitive in $\hat\phi\in \mc C(B, E_0^{n})$ such that
\begin{equation}\label{feb 19 eq:4 fin}
\|\phi\|_{\mc C(\overline B, E_0^{h-1})} \le C \|\omega\|_{L^{Q/2}(B_\lambda, E_0^{n+1})}.
\end{equation}
\end{itemize}
\end{theorem}

\begin{proof}We suppose
first $h\neq n+1$. As above, we take $\lambda >\lambda' >1$ such that
the interior
{Poincar\'e} inequality proved in \cite{BFP2} (Theorem 5.19-i) with $p=q=Q$)  holds for the pair
$B'$ and $B_\lambda$. That is,
if $\omega\in
L^Q(B_\lambda, E_0^h)$ is a closed form, 
then there exists $\tilde \phi\in L^Q(B', E_0^{h-1})$ such that $d_c\tilde\phi=\omega$
in $B'$ and 
\begin{equation}\label{feb 22 eq:1 new}
\| \tilde\phi\|_{L^Q(B', E_0^{h-1})}\| \le C \|\omega\|_{L^Q(B_\lambda, E_0^h)}.
\end{equation}
Let now $\zeta\in \mc D(B')$, and set $\omega':=d_c(\zeta \tilde\phi)$ continued
by zero outside of $B'$, $\zeta\equiv 1$ in $B$.
Keeping in mind \eqref{feb 22 eq:1 new}, by Lemma \ref{leibniz}
\begin{equation}\begin{split}\label{feb 22 eq:2}
\| \omega' & \|_{L^Q(\he n, E_0^h)} 
= \| \omega'  \|_{L^Q(B', E_0^h)}  \le C_\zeta \big(\| \tilde\phi \|_{L^Q(B', E_0^{h-1})} +
\|d_c \tilde\phi \|_{L^Q(B', E_0^h)}\big)
\\&
\le C_\zeta \|\omega \|_{L^Q(B_\lambda, E_0^h)}.
\end{split}\end{equation}

Since $\omega'$ is closed,  Theorem \ref{main global}
yields the existence of a continuous primitive $\phi$ of $\omega'$,
 $\phi\in \mc C_0(\he n, E_0^{h-1})$ such that
\begin{equation}\label{feb 22 eq:3}
\|\phi\|_{\mc C_0(\he n, E_0^{h-1})} \le C \|\omega' \|_{L^Q(\he n, E_0^h)}
%= C \|\omega' \|_{L^Q(B', E_0^h)}.
\end{equation}
On the other hand, in $B$
$$
d_c\phi = \omega' = d_c(\zeta \tilde\phi) = d_c\tilde\phi = \omega.
$$
Moreover, by \eqref{feb 22 eq:3} and \eqref{feb 22 eq:2},
\begin{equation*}\begin{split}
\|\phi & \|_{\mc C(\overline B, E_0^{h-1})} \le \|\phi\|_{\mc C_0(\he n, E_0^{h-1})}
%\\&
\le C \|\omega' \|_{L^Q(\he n, E_0^h)} \le C_\zeta \|\omega \|_{L^Q(B_\lambda, E_0^h)}.
\end{split}\end{equation*}
This proves Theorem \ref{main local} when $h\neq n+1$.

\medskip

Now we deal with the case $h=n+1$.

We use the homotopy formula 
 \eqref{9 nov}. In particular, $S\omega$ is closed in $B'$, so that we can
apply Lemma \ref{homotopy 1} to $\omega$. Set now
$$
\Phi:= (KS + T) \omega \qquad\mbox{in $B'$.}
$$
Trivially, by \eqref{9 nov} and Lemma \ref{homotopy 1}, $d_c\Phi = d_cKS\omega + d_cT\omega = S\omega +
d_cT\omega = \omega$ in $B'$. With the notations of \eqref{balls inclusion}, take $\lambda >\lambda'>\lambda'' >1$ so that 
$c_0\sqrt{\lambda''}< \lambda'$. Then
\begin{equation}\label{befana}\begin{split}
B'' &:= B(e,\lambda'') \Subset B_E:=B_{\mathrm{Euc}}(0, c_0\sqrt{\lambda''}) 
\\&\Subset
B_{\mathrm{Euc}}':=B_{\mathrm{Euc}}(0, 2c_0\sqrt{\lambda''})
\Subset B' \Subset B_\lambda.
\end{split}\end{equation}

Let us prove that
\begin{equation}\label{speriamo}
\| \Phi \|_{W^{1,Q/2}(B'',E_0^n)}  \le C   \|\omega\|_{L^{Q/2}(B_\lambda,E_0^{n+1})} \,.
\end{equation}

Indeed, we have
\begin{equation*}\begin{split}
 \| & \Phi \|_{W^{1,Q/2}(B'',E_0^n)} 
  \le
\|  KS\omega \|_{W^{1,Q/2}(B'',E_0^n)} + \|  T\omega \|_{W^{2,Q/2}(B'',E_0^n)}
\end{split}\end{equation*}
By the continuity of the operator $T$ we have that 
\begin{equation}\label{oggi}
	\|  T\omega \|_{W^{2,Q/2}(B'',E_0^n)}\le  \|\omega\|_{L^{Q/2}(B'', E_0^{n+1})}\,.
\end{equation}
Let us consider now the term $\|  KS\omega \|_{W^{2,Q/2}(B'',E_0^n)}$ and remember that, in $B''$, 
 by Theorem \ref{folland stein varia}, v), 
$W^{1,Q/2}_{\mathrm{Euc}} (B'') \subset W^{1,Q/2}(B'')  $.
By \cite{adams}, Theorem 4.12, keeping in mind that $\Pi_E$ is an operator of order 0 on $(n+1)$-forms and it is a differential operator of order 1 on $n$-forms, we have
%\begin{equation}\label{oggi 2}\begin{split}
%\| & KS\omega \|_{W^{2,Q/2}(B'',E_0^n)} \le C \|KS\omega\|_{W^{1,Q/2}_{\mathrm{Euc}}(B'',E_0^{n})} 
%\\&
%= C \|(\Pi_{E_0}\circ \Pi_E \circ K_{\mathrm{Euc}}   \circ \Pi_E) S\omega\|_{W^{1,Q/2}_{\mathrm{Euc}}(B'',E_0^{n})}
%\\&
%\le C \|(\Pi_E \circ K_{\mathrm{Euc}}   \circ \Pi_E) S\omega\|_{W^{1,Q/2}_{\mathrm{Euc}}(B'', {\bigwedge}\vphantom{!}^n )}
%\\&
%\le C \|( K_{\mathrm{Euc}}   \circ \Pi_E) S\omega\|_{W^{2,Q/2}_{\mathrm{Euc}}(B'', {\bigwedge}\vphantom{!}^{n})}
%\\&
%\le C \| \Pi_E S\omega\|_{W^{2,q}_{\mathrm{Euc}}(B', {\bigwedge}\vphantom{!}^{n+1})} \qquad\mbox{(by Theorem \ref{chapeau})}
%\\&
%\le C \|S\omega\|_{W^{2,Q/2}_{\mathrm{Euc}}(B',E_0^{n+1})} \le C   \|\omega\|_{L^{Q/2}(B_\lambda,E_0^{n+1})} \,,
%\end{split}\end{equation}
\begin{equation}\label{oggi 2}\begin{split}
\|  KS & \omega \|_{W^{1,Q/2}(B'',E_0^n)} 
\le C \|  KS\omega\|_{W^{1,Q/2}_{\mathrm{Euc}}(B'',E_0^{n})} 
\\&
= 
C \|(\Pi_{E_0}\circ \Pi_E \circ K_{\mathrm{Euc}}   \circ \Pi_E) S\omega\|_{W^{1,Q/2}_{\mathrm{Euc}}(B'',E_0^{n})}
%\\&
%\le C \|(\Pi_{E_0}\circ \Pi_E \circ K_{\mathrm{Euc}}   \circ \Pi_E) S\omega\|_{W^{1,Q/2}(B'',E_0^{n})}
\\&
\le C \|(\Pi_E \circ K_{\mathrm{Euc}}   \circ \Pi_E) S\omega\|_{W^{1,Q/2}_{\mathrm{Euc}}(B'', {\bigwedge}\vphantom{!}^{n})}
\\&
\le C \|( K_{\mathrm{Euc}}   \circ \Pi_E) S\omega\|_{W^{2,Q/2}_{\mathrm{Euc}}(B'', {\bigwedge}\vphantom{!}^{n})}
\\&
\le C \|( K_{\mathrm{Euc}}   \circ \Pi_E) S\omega\|_{W^{2,Q/2}_{\mathrm{Euc}}(B_E, {\bigwedge}\vphantom{!}^{n})}
\\&
\le C \| \Pi_E S\omega\|_{W^{2,Q/2}_{\mathrm{Euc}}(B_E', {\bigwedge}\vphantom{!}^{n+1})} \qquad\mbox{(by Theorem \ref{chapeau})}
\\&
\le C \| \Pi_E S\omega\|_{W^{2,Q/2}_{\mathrm{Euc}}(B', {\bigwedge}\vphantom{!}^{n+1})}
%\\&
%\le C \| \Pi_E S\omega\|_{W^{2,Q/2}_{\mathrm{Euc}}(B_E, {\bigwedge}\vphantom{!}^{n+1})} 
%\\&
%= C \| \Pi_E\Pi_{E_0}\Pi_E S\omega\|_{W^{2,Q/2}_{\mathrm{Euc}}(B', {\bigwedge}\vphantom{!}^{n+1})} 
%\\&
%\le C \| \Pi_{E_0}\Pi_E S\omega\|_{W^{2,Q/2}_{\mathrm{Euc}}(B', {\bigwedge}\vphantom{!}^{n+1})} 
%\\&
%\le C \| \Pi_E S\omega\|_{W^{2,Q/2}_{\mathrm{Euc}}(B',E_0^{n+1})} 
\\&
\le C \| S\omega\|_{W^{2,Q/2}_{\mathrm{Euc}}(B', E_0^{n+1})}
\\&
 \le C \|S\omega\|_{W^{4,Q/2}(B',E_0^{n+1})}  \qquad\mbox{(by Theorem \ref{folland stein varia}-vi))}
\\&
  \le C   \|\omega\|_{L^{Q/2}(B_\lambda,E_0^{n+1})} 
\qquad\mbox{(by \eqref{befana}),} 
\end{split}\end{equation}
where the last inequality uses the fact that $S$ is a smoothing operator.
Therefore, we get \eqref{speriamo}.

Let now $\zeta\in \mc D(B'')$, $\zeta\equiv 1$ in $B$, and set $\omega'':=d_c(\zeta \Phi)$ continued
by zero outside of $B''$.

Keeping in mind \eqref{speriamo} and by Leibniz formula stated in Lemma \ref{leibniz},
\begin{equation}\label{feb 22 eq:2 bis}\begin{split}
\| \omega'' & \|_{L^{Q/2}(\he n, E_0^{n+1})} 
= \| \omega''  \|_{L^{Q/2}(B'', E_0^{n+1})} \\& \le C_\zeta \big(\| \Phi \|_{W^{1,Q/2}(B'', E_0^{n})} +
\|d_c \Phi \|_{L^{Q/2}(B'', E_0^{n+1})}\big)
%\\&
%=
%C_\zeta \big(\|\tilde \phi \|_{L^\infty(B', E_0^{n})} +
%\|\omega \|_{L^{Q/2}(B_\lambda, E_0^{n+1})}\big)
\\&
\le C_\zeta \|\omega \|_{L^Q(B_\lambda, E_0^h)}.
\end{split}\end{equation}

Since $\omega''$ is closed,  Theorem \ref{main global}
yields the existence of a continuous primitive $\hat\phi$ of $\omega''$,
 $\hat\phi\in \mc C_0(\he n, E_0^{n})$ such that
\begin{equation}\label{feb 22 eq:3 bis}
\|\hat\phi\|_{\mc C_0(\he n, E_0^{n})} \le C \|\omega'' \|_{L^{Q/2}(\he n, E_0^{n+1})}
%= C \|\omega' \|_{L^Q(B', E_0^h)}.
\end{equation}
On the other hand, in $B$
$$
d_c\hat\phi = \omega'' = d_c(\zeta \Phi) = d_c\Phi = \omega.
$$
Moreover, by \eqref{feb 22 eq:3 bis} and \eqref{feb 22 eq:2 bis},
\begin{equation*}\begin{split}
\|\hat\phi & \|_{\mc C(\overline B, E_0^{n})} \le \|\hat\phi\|_{\mc C_0(\he n, E_0^{n})}
\le C \|\omega'' \|_{L^{Q/2}(\he n, E_0^{n+1})} \le C \|\omega \|_{L^{Q/2}(B_\lambda, E_0^{n+1})}.
\end{split}\end{equation*}
and we are done.

\end{proof}

\section{Compact Riemannian and contact subRiemannian manifolds}
\label{compact}

\subsection{A smoothing homotopy}

\begin{proposition}\label{sh}
There exists $\lambda>1$ such that, if we denote by $B:=B(e,1)$ and $B_\lambda:=B(e,\lambda)$ two concentric Heisenberg balls,   
there exist operators $S:\mc C^{\infty}(B_\lambda,E_0^\bullet)\to \mc C^{\infty}(B,E_0^{\bullet})$ and $T:\mc C^{\infty}(B_\lambda,E_0^\bullet)\to \mc C^{\infty}(B,E_0^{\bullet-1})$ such that $S+d_cT+Td_c$ equals the restriction operator from $B_\lambda$ to $B$. Furthermore, $T$ extends to a bounded operator on $\mc C^{0}$, i.e. from $\mc C(B_\lambda,E_0^\bullet)$ to $\mc C(B,E_0^{\bullet-1})$ and on $\mc C^1$, i.e., from $\mc C^1(B_\lambda,E_0^\bullet)$ to $\mc C^1(B,E_0^{\bullet-1})$. $S$ extends to a bounded operator from $\mc C(B_\lambda,E_0^\bullet)$ to $\mc C^\ell(B,E_0^{\bullet-1})$ for every $\ell\in\mathbb{N}$. % In degree $n+1$, $T$ gains one derivative, i.e. it is bounded from $C(B_\lambda,E_0^\bullet)$ to $C^1(B_\lambda,E_0^{\cdot-1})$.
\end{proposition}

\begin{proof}
Pick $q>Q$. Then Theorem 5.14 of \cite{BFP2} provides such a homotopy, where $T$ is bounded from $L^q$ to $W^{1,q}$. Since $\mc C(B,E_0^\bullet)\subset L^q(B,E_0^\bullet)$ and Sobolev's embedding implies that $W^{1,q}(B,E_0^{\bullet-1})\subset C(B,E_0^{\bullet-1})$, $T$ is bounded from $\mc C^0$ to $\mc C^0$. A slight extension of the proof of Theorem 5.14 of \cite{BFP2} shows that $T$ is bounded from $W^{1,q}$ to $W^{2,q}$, hence from $\mc C^1$ to $\mc C^1$. 
\end{proof}

\subsection{A $C^0$ Poincar\'e inequality}

\begin{proposition}\label{C0}
There exists $\lambda>1$ such that, if we denote by $B:=B(e,1)$ and $B_\lambda:=B(e,\lambda)$ two concentric Heisenberg balls,   
there exists an operator $P:\mc C^{\infty}(B_\lambda,E_0^\bullet)\to \mc C^{\infty}(B_\lambda,E_0^{\cdot-1})$ such that $d_cP+Pd_c$ equals the restriction operator from $B_\lambda$ to $B$. Furthermore, $P$ extends to a bounded operator on $\mc C^{0}$, i.e. from $\mc C(B_\lambda,E_0^\bullet)$ to $\mc C(B_\lambda,E_0^{\cdot-1})$.
\end{proposition}

\begin{proof}
The proof of Theorem 5.19 of \cite{BFP2} provides such an operator. It is of the form $P=KS+T$ where $S,T$ is the homotopy occurring in Proposition \ref{sh} and $K$ is cooked up from the Euclidean homotopy of Iwaniec-Lutoborsky and Rumin's homotopy between de Rham and Rumin complexes. For every $q>Q$, $P$ is bounded from $L^q(B_\lambda,E_0^\bullet)$ to $W^{1,q}(B_\lambda,E_0^{\bullet-1})$, hence on $\mc C^0$.
\end{proof}

\subsection{Leray's acyclic covering theorem}

Leray's acyclic covering theorem relates the de Rham cohomology of a manifold and the simplicial cohomology of a simplicial complex, the nerve of an acyclic covering. Acyclic means that  de Rham cohomology of all intersections of pieces of the covering vanishes. The nerve is the complex with one vertex per piece, and with a simplex through a set of vertices each time the corresponding pieces have a nonempty common intersection.

An application of Leray's acyclic covering theorem to Poincar\'e inequalities on bounded geometry Riemannian manifolds is given in \cite{PR}. A persistent version of Leray's acyclic covering theorem is described in \cite{Pcup}. It provides the same conclusion under weaker requirements on the covering. The analytical ingredients are
\begin{itemize}
  \item A smoothing homotopy, as in Proposition \ref{sh}, to circumvent the obstacle arising from the Leibniz formula.
  \item Poincar\'e inequalities provided by linear operators, with losses on domains allowed, as in Proposition \ref{C0}.   \end{itemize}

The context here is slightly different from \cite{PR} and \cite{Pcup}. In these papers, manifolds are noncompact and are assumed to have vanishing cohomology. Here, manifolds are compact and the vanishing assumption is replaced with the following fact: there exists a radius $r_0$ such that balls of radius $<r_0$ have vanishing cohomology. 

There is a second difference. To prove Theorem \ref{main compact} in degree $h$, one uses the vector of exponents $(\infty,\ldots,\infty,Q)$, meaning that the $L^Q$ norm is used on $h$-forms (replace $Q$ with $Q/2$ if $h=n+1$), and the $L^\infty$ norm on spaces of continuous $j$-forms is used when $j<h$. To apply the machinery of \cite{Pcup}, one would need a linear operator $P$ that solves $d_c$ on a pair of concentric Heisenberg balls $(B,B_\lambda)$, with estimates: $P$ is bounded from $L^{Q}(B_\lambda,E_0^h)$ to $\mc C(B,E_0^{h-1})$. As observed by Bourgain and Brezis, such a linear operator cannot exist. Theorem \ref{main local} merely provides a nonlinear map $\omega\mapsto\phi$. It turns out that this does not seriously affect the argument.

Morally (this is the point of view adopted in \cite{PR}), Leray's method goes as follows. Pick a suitable finite covering $\{U_i\}$ by small enough balls, so that each of them comes with a nearly isometric contactomorphism to an open set in Heisenberg group. Given a $d_c$-closed $L^Q$ $h$-form $\omega$ on $M$, pick a continuous primitive on each $U_i$, using Theorem \ref{main local}. This yields a 0-cochain $\omega^1$ of the covering, with values in spaces of continuous $(h-1)$ forms on pieces $U_i$. Let $\delta\omega^1$ denote the 1-cochain which on a pair $(i,i')$ is equal to the restriction to $U_i\cap U_{i'}$ of $\omega^1_i -\omega^1_{i'}$. It is $d_c$-closed. Pick a primitive on each $U_i\cap U_{i'}$, using the operator $P$ of Proposition \ref{C0}. 
This yields a 1-cochain $\omega^2$ of the covering with values in spaces of continuous $(h-2)$-forms on intersections $U_i\cap U_{i'}$. And so on. After $h$ steps, one gets a $h$-cochain with values in spaces of constant functions, i.e. a combinatorial $h$-cochain. 

Conversely, given a $j$-cocycle $\psi$ with values in spaces of continuous $j$-forms on $h+1$-uple intersections $U_{i}\cap U_{i'}\cap\cdots$, using a smooth partition of unity to extend forms to intersections of less pieces and adding them up, one produces a $(j-1)$-cochain $\epsilon\psi$ such that $\delta\epsilon\psi=\psi$. In order to be able to apply $d_c$ to $\epsilon\phi$, one applies to it the smoothing operator $S$ of Proposition \ref{sh}. Given a combinatorial $h$-cocycle $\kappa\in \ell^{Q}K^h$, i.e. a skew-symmetric real valued function on 
$h$-simplices of the nerve which satisfies $\delta\kappa=0$ and is $Q$-summable, one defines $\kappa^1=d_c S\epsilon\kappa$, which is a $(h-1)$-cocycle with values in spaces of smooth $1$-forms on $h$-uple intersections, and which is again $Q$-summable. And so on. After $h$ steps, one gets a $0$-cocycle with values in spaces of smooth $h$-forms, i.e. a smooth globally defined $L^Q$ $h$-form. 

The machinery shows that both constructions are inverses of each other in cohomology. Indeed, the nonlinear step occurs only once, at the very beginning, so its nonlinear character does not affect the needed identities. The constructions define a biLipschitz bijection in cohomology. Since one of them is linear, one gets a linear and bounded bijection between de Rham cohomology 
$$
(\mathrm{ker}(d_c)\cap L^Q(M,E_0^{h}))/d_c(C(M,E_0^{h-1})\cap d_c^{-1}(L^Q))
$$
and combinatorial cohomology
$$
(\mathrm{ker}(\delta)\cap \ell^{Q}K^h)/\delta(\ell^{\infty}K^{h-1}\cap \delta^{-1}\ell^{Q}).
$$
It also shows that exact cohomology
$$
(\mathrm{im}(d_c)\cap L^Q(M,E_0^{h}))/d_c(C(M,E_0^{h-1})\cap d_c^{-1}(L^Q))
$$
is mapped to exact cohomology
$$
(\mathrm{im}(\delta)\cap \ell^Q K^h)/\delta(\ell^{\infty}K^{h-1})\cap \delta^{-1}\ell^{Q}),
$$
which is $0$. Indeed, the nerve is finite, all cochains belong to $\ell^Q$ or $\ell^{\infty}$, so exact cohomology vanishes. It follows that exact cohomology vanishes too on the de Rham/Rumin side: every exact $L^Q$ $h$-form is the $d_c$ of a continuous $(h-1)$-form, with estimates. This proves Theorem \ref{main compact}.

\section*{Acknowledgements}

A.B. and B.F. are supported by the University of Bologna, funds for selected research topics.
A.B. is supported by PRIN 2022 Italy (ref. 2022F4F2LH), and by GNAMPA of INdAM (Istituto Nazionale di Alta Matematica ``F. Severi''), Italy.

P.P. is supported by Agence Nationale de la Recherche, ANR-22-CE40-0004 GOFR.

\bibliographystyle{amsplain}

\bibliography{BFP6_submitted}

\bigskip
\tiny{
\noindent
Annalisa Baldi and Bruno Franchi 
\par\noindent
Universit\`a di Bologna, Dipartimento
di Matematica\par\noindent Piazza di
Porta S.~Donato 5, 40126 Bologna, Italy.
\par\noindent
e-mail:
annalisa.baldi2@unibo.it, 
bruno.franchi@unibo.it.
}

\medskip

\tiny{
\noindent
Pierre Pansu 
\par\noindent  Universit\'e Paris-Saclay, CNRS, Laboratoire de math\'ematiques d'Orsay
\par\noindent  91405, Orsay, France.
\par\noindent 
e-mail: pierre.pansu@universite-paris-saclay.fr
}

\end{document}